\documentclass[reqno]{amsart}

\def\RaymondContactDetails{
  \author[van Bommel]{Raymond van Bommel}
  \address{Raymond van Bommel, Department of Mathematics, Massachusetts
    Institute of Technology, 77 Massachusetts Avenue, Cambridge, MA 02139,
    USA}
  \email{bommel@mit.edu}
}

\def\JordanContactDetails{
  \author[Docking]{Jordan Docking}
  \address{Jordan Docking, Department of Mathematics, University College London, London WC1H 0AY, UK}
  \email{jordan.docking.18@ucl.ac.uk}
}

\def\ReynaldContactDetails{
  \author[Lercier]{Reynald Lercier}
  \address{%
    Reynald Lercier,
    DGA \& Univ Rennes, %
    CNRS, IRMAR - UMR 6625, F-35000
    Rennes, %
    France. %
  }
  \email{reynald.lercier@m4x.org}
}

\def\ElisaContactDetails{
  \author[Lorenzo Garc\'ia]{Elisa Lorenzo Garc\'ia}
  \address{
    Elisa Lorenzo Garc\'ia,
    Institut de Math\'ematiques,  Universit\'e de Neuch\^atel, rue Emile-Argand 11, 2000, Neuch\^atel,
    Switzerland \&
    Univ Rennes, CNRS, IRMAR - UMR 6625, F-35000
    Rennes, %
    France. %
  }
  \email{elisa.lorenzo@unine.ch}
}

\usepackage[utf8]{inputenc}
\usepackage[T1]{fontenc}

\usepackage{ae}
\usepackage{lmodern}
\usepackage{bm}

\advance\oddsidemargin -1.5cm
\advance\evensidemargin -1.5cm
\advance\textwidth 2cm
\advance\textheight 0.3cm %

\usepackage{color}
\usepackage[dvipsnames]{xcolor}

\usepackage{amsmath, amsthm, amssymb, amsfonts}

\def\gettexliveversion#1(#2 #3 #4#5#6#7#8)#9\relax{#4#5#6#7}
\edef\texliveversion{\expandafter\gettexliveversion\pdftexbanner\relax}
\ifnum\texliveversion<2022      %
\DeclareRobustCommand{\SkipTocEntry}[9]{}
\else %
\DeclareRobustCommand{\SkipTocEntry}[5]{}
\fi %

\usepackage{stmaryrd}

\usepackage{tikz-cd}

\usepackage{ifthen}

\usepackage{tkz-graph}
\usepackage{semtkzX}

\counterwithin{figure}{section}
\counterwithin{table}{section}
\newtheorem{theorem}{Theorem}[section]
\newtheorem*{theorem*}{Theorem}
\newtheorem{corollary}[theorem]{Corollary}
\newtheorem*{corollary*}{Corollary}
\newtheorem{conjecture}[theorem]{Conjecture}
\newtheorem*{conjecture*}{Conjecture}
\newtheorem{lemma}[theorem]{Lemma}

\newtheorem{proposition}[theorem]{Proposition}
\newtheorem*{proposition*}{Proposition}

\theoremstyle{definition}

\newtheorem*{definition*}{Definition}
\newtheorem{example}[theorem]{Example}
\newtheorem*{example*}{Example}
\newtheorem{remark}[theorem]{Remark}

\newtheorem*{notation*}{Notation}

\usepackage{booktabs}
\usepackage{array}
\usepackage{collcell}
\usepackage{multirow}
\usepackage{colortbl}

\definecolor{mygray}{gray}{0.92}
\newcommand{\ctbl}{\rowcolor{mygray}}
\usepackage{float}
\usepackage{subcaption}
\usepackage[all]{xy}

\usepackage{graphicx}

\usepackage{xspace}

\usepackage[style=alphabetic,citestyle=alphabetic,sorting=nyt,sortcites=true,autopunct=true,hyperref=true,abbreviate=true,backref=true,useprefix=false,giveninits=true,maxbibnames=99,maxcitenames=99,maxalphanames=4]{biblatex}
\AtEveryBibitem{ \clearfield{url} \clearfield{urldate} \clearfield{review} \clearfield{series} \clearfield{doi} \clearfield{isbn} \clearfield{issn} }
\defbibheading{bibempty}{}
\DeclareNolabel{\nolabel{\regexp{[\p{Z}]+}}}
\AtBeginBibliography{\small}

\definecolor{mylinkcolor}{rgb}{0.5,0.0,0.0}
\definecolor{myurlcolor}{rgb}{0.0,0.0,0.75}
\usepackage[
	colorlinks, urlcolor=myurlcolor,citecolor=myurlcolor,linkcolor=mylinkcolor,
        pdfusetitle,
]{hyperref}

\usepackage[ruled,vlined,boxed,linesnumbered,norelsize]{algorithm2e}

\usepackage{tikz}
\usetikzlibrary{shapes.geometric}
\usetikzlibrary{hobby}
\usetikzlibrary{positioning}
\usetikzlibrary{arrows}

\usetikzlibrary{external}
\tikzsetfigurename{main-figure}

\tikzset{
  g3lattice/.style={inner sep=1pt,norm/.style={red!50!blue},char/.style={blue!50!black},
    lin/.style={black!50}},cnj/.style={black!50,yshift=-2.5pt,left=-1pt of #1,scale=0.5,fill=white},
  sml/.style={scale=0.55},
  typ/.style={scale=1.0,inner sep=0.2em},
  lrg/.style={scale=0.9,inner sep=0.2em},
  fname/.style={scale=0.55},
  lin/.style={-,shorten <=-0.07em,shorten >=-0.07em},
  rem/.style={black!20,thin},
  lname/.style={scale=0.55,sloped,red,above=-0.07em,near end},
  every loop/.style={}
}

\makeatletter
\pgfkeys{
  /utils/temp/.initial/.expand once=\tikzcd,
  /utils/temp/.prefix=%
    \ifx\path\tikz@command@path %

    \else %
      \tikzexternaldisable
    \fi,
  /utils/temp/.get=\tikzcd}
\makeatother

\input cyracc.def

\DeclareMathOperator{\Id}{Id}

\DeclareMathOperator{\SL}{SL}
\DeclareMathOperator{\DO}{DO}
\DeclareMathOperator{\Rad}{Rad}

\def\Q{\mathbb{Q}}
\def\Z{\mathbb{Z}}
\def\C{\mathbb{C}}

\def\P{\mathbb{P}}

\def\Gm2{\mathbb{G}_m^2}

\def\PP{\mathbb{P}}

\newenvironment{strata}[1]{%
  \subsubsection*{{\underline{\textsc{Strata of dimension {#1}}}}}
  \ \smallskip%
}

\newcommand{\stbtype}[2]{$\substack{\scalebox{0.7}{#1}\\\mathtt{(#2)}}$}
\newcommand{\stbtypehyp}[2]{$\substack{\scalebox{0.7}{#1}\\\mathtt{(#2)_{_{\mathtt{H}}}}}$}

\newcommand{\ClId}[1]{{\bf \color{mylinkcolor} (#1)}}

\newcommand{\sA}{{\operatorname{A}}}
\newcommand{\sD}{{\operatorname{D}}}
\newcommand{\sE}{{\operatorname{E}}}
\newcommand{\sX}{{\operatorname{X}}}

\newcommand{\ifcolored}{false}

\newcommand\Smooth{\fcolorbox{white}{white}}

\newcommand\Aone{\ifthenelse{\equal{\ifcolored}{true}}
  {\fcolorbox{white}{SkyBlue}}{\fcolorbox{white}{white}}}             %
\newcommand\Aoneptwo{\ifthenelse{\equal{\ifcolored}{true}}
  {\fcolorbox{white}{Cyan}}{\fcolorbox{white}{white}}}            %
\newcommand\Aonepthree{\ifthenelse{\equal{\ifcolored}{true}}
  {\fcolorbox{white}{NavyBlue}}{\fcolorbox{white}{white}}}      %
\newcommand\RAonepthree{\ifthenelse{\equal{\ifcolored}{true}}
  {\fcolorbox{white}{MidnightBlue}}{\fcolorbox{white}{white}}} %
\newcommand\RAonepfoura{\ifthenelse{\equal{\ifcolored}{true}}
  {\fcolorbox{white}{Blue}}{\fcolorbox{white}{white}}}         %
\newcommand\RAonepfourb{\ifthenelse{\equal{\ifcolored}{true}}
  {\fcolorbox{white}{Periwinkle}}{\fcolorbox{white}{white}}}   %
\newcommand\RAonepfive{\ifthenelse{\equal{\ifcolored}{true}}
  {\fcolorbox{white}{Orchid}}{\fcolorbox{white}{white}}}        %
\newcommand\RAonepsix{\ifthenelse{\equal{\ifcolored}{true}}
  {\fcolorbox{white}{Fuchsia}}{\fcolorbox{white}{white}}}        %

\newcommand\Atwo{\ifthenelse{\equal{\ifcolored}{true}}
  {\fcolorbox{white}{pink}}{\fcolorbox{white}{white}}}    %
\newcommand\AoneAtwo{\ifthenelse{\equal{\ifcolored}{true}}
  {\fcolorbox{white}{Lavender}}{\fcolorbox{white}{white}}} %
\newcommand\AoneptwoAtwo{\ifthenelse{\equal{\ifcolored}{true}}
  {\fcolorbox{white}{Magenta}}{\fcolorbox{white}{white}}} %
\newcommand\AonepthreeAtwo{\ifthenelse{\equal{\ifcolored}{true}}
  {\fcolorbox{white}{Red}}{\fcolorbox{white}{white}}} %

\newcommand\Atwoptwo{\ifthenelse{\equal{\ifcolored}{true}}
  {\fcolorbox{white}{Goldenrod}}{\fcolorbox{white}{white}}} %
\newcommand\AoneAtwoptwo{\ifthenelse{\equal{\ifcolored}{true}}
  {\fcolorbox{white}{Apricot}}{\fcolorbox{white}{white}}} %

\newcommand\Athree{\ifthenelse{\equal{\ifcolored}{true}}
  {\fcolorbox{white}{Green}}{\fcolorbox{white}{white}}} %
\newcommand\RAoneAthree{\ifthenelse{\equal{\ifcolored}{true}}
  {\fcolorbox{white}{Green}}{\fcolorbox{white}{white}}} %

\newcommand\Atwopthree{\ifthenelse{\equal{\ifcolored}{true}}
  {\fcolorbox{white}{orange}}{\fcolorbox{white}{white}}} %

\colorlet{coltypeI} {black}
\colorlet{coltypeII}{black}

\iftrue
\newcommand{\comm}[1]{#1}
\newcommand{\rey}[1]{{\color{Green} \noindent Reynald: #1}}

\newcommand{\reynald}[1]{\rey{#1}}
\newcommand{\raymond}[1]{{\color{blue} \noindent Raymond: #1}}
\newcommand{\jordan}[1]{{\color{purple} \noindent Jordan: #1}}
\newcommand{\elisa}[1]{{\color{red} \noindent Elisa: #1}}

\else
\newcommand{\comm}[1]{}
\newcommand{\rey}[1]{}\newcommand{\reynald}[1]{}
\newcommand{\raymond}[1]{}
\newcommand{\jordan}[1]{}
\newcommand{\elisa}[1]{}
\fi

\subjclass[2020]{}

\tikzset{external/only named=true}

\def \globalscale {0.004}

\def\SngRedAfive{%
  \begin{tikzpicture}[y=1cm, x=1cm,
    yscale=0.9*\globalscale,xscale=0.8*\globalscale,%
    baseline={([yshift=-1mm]current bounding box.center)}]
    \begin{scope}[fill=black,cm={1,0,0,1,(0, 0)}]
      \path[fill=black] (129.2, 108.1).. controls (126.6, 107.4) and (125.4,
      106.7) .. (123.0, 104.6).. controls (120.5, 102.3) and (107.1, 95.3)
      .. (102.7, 93.9).. controls (100.0, 93.1) and (99.1, 93.0) .. (94.0,
      93.0).. controls (87.8, 93.0) and (85.1, 93.5) .. (80.2,
      95.6).. controls (77.2, 96.9) and (72.1, 97.9) .. (67.0,
      98.4).. controls (62.3, 98.9) and (56.7, 98.5) .. (52.4,
      97.4).. controls (49.0, 96.5) and (42.9, 94.1) .. (39.7,
      92.3).. controls (37.0, 90.8) and (28.5, 84.5) .. (25.1,
      81.5).. controls (21.6, 78.4) and (12.6, 67.8) .. (11.0,
      64.8).. controls (9.0, 61.0) and (11.0, 59.1) .. (14.9, 61.1).. controls
      (16.4, 61.9) and (18.9, 64.3) .. (25.0, 70.8).. controls (29.0, 75.2)
      and (38.3, 83.8) .. (42.2, 86.9).. controls (47.9, 91.2) and (54.6,
      93.8) .. (61.2, 94.1).. controls (66.4, 94.3) and (69.0, 93.9) .. (77.2,
      91.2).. controls (82.4, 89.5) and (85.1, 88.9) .. (87.8,
      88.5).. controls (96.8, 87.5) and (105.3, 89.4) .. (115.7,
      94.5).. controls (124.9, 99.1) and (131.4, 103.0) .. (132.4,
      104.5).. controls (133.5, 105.9) and (132.9, 108.7) .. (131.7,
      108.6).. controls (131.3, 108.6) and (130.2, 108.4) .. (129.2, 108.1) --
      cycle;
      \path[fill=black] (120.4, 77.9).. controls (120.2, 77.6) and (119.0,
      75.4) .. (117.7, 72.8).. controls (108.9, 56.3) and (101.8, 47.5)
      .. (94.0, 43.9).. controls (88.5, 41.4) and (91.2, 41.5) .. (45.0,
      41.3).. controls (6.3, 41.1) and (3.2, 41.1) .. (2.0, 40.5).. controls
      (-0.4, 39.3) and (-0.4, 37.6) .. (2.2, 36.5).. controls (3.6, 35.8) and
      (5.0, 35.8) .. (21.3, 35.6) -- (38.9, 35.5) -- (37.0, 34.1).. controls
      (34.5, 32.5) and (32.7, 30.4) .. (30.6, 26.5).. controls (27.0, 20.1)
      and (24.3, 13.5) .. (24.3, 10.9).. controls (24.3, 10.4) and (24.0, 8.9)
      .. (23.6, 7.7).. controls (23.2, 6.5) and (22.9, 5.0) .. (22.9,
      4.5).. controls (22.9, 3.1) and (24.0, 0.7) .. (24.7, 0.4).. controls
      (26.4, -0.5) and (27.1, 0.8) .. (30.0, 9.8).. controls (31.2, 13.4) and
      (32.9, 17.9) .. (33.9, 19.8).. controls (37.5, 26.7) and (43.3, 32.4)
      .. (49.4, 34.8).. controls (50.9, 35.4) and (54.0, 35.5) .. (98.4,
      35.6).. controls (148.4, 35.8) and (147.2, 35.8) .. (149.0,
      37.3).. controls (149.3, 37.6) and (149.6, 38.2) .. (149.6,
      38.7).. controls (149.6, 39.4) and (149.2, 39.7) .. (147.9,
      40.3).. controls (146.2, 41.1) and (146.0, 41.1) .. (122.8, 41.3) --
      (99.4, 41.5) -- (102.2, 43.7).. controls (105.9, 46.5) and (112.1, 52.9)
      .. (115.0, 56.9).. controls (116.9, 59.5) and (119.8, 64.4) .. (123.6,
      71.8).. controls (125.1, 74.6) and (124.1, 78.0) .. (121.8,
      78.2).. controls (121.2, 78.3) and (120.6, 78.1) .. (120.4, 77.9) --
      cycle;
    \end{scope}
  \end{tikzpicture}
}

\def\SngRedDsix{%
  \begin{tikzpicture}[y=1cm, x=1cm,
    yscale=\globalscale,xscale=\globalscale,%
    baseline={([yshift=-1mm]current bounding box.center)}]
    \begin{scope}[fill=black,cm={1,0,0,1,(0, 0)}]
      \path[fill=black] (39.9, 75.1).. controls (39.1, 73.0) and (38.6, 68.5)
      .. (38.7, 63.5).. controls (38.7, 58.3) and (38.7, 57.9) .. (39.7,
      55.6).. controls (41.6, 50.6) and (44.2, 46.2) .. (47.9, 41.6) -- (51.5,
      37.1) -- (28.5, 37.0).. controls (6.3, 36.9) and (5.4, 36.8) .. (4.3,
      36.2).. controls (2.3, 34.9) and (2.5, 33.2) .. (4.9, 32.2).. controls
      (6.0, 31.6) and (9.0, 31.6) .. (35.2, 31.5) -- (64.2, 31.4) -- (64.8,
      30.1).. controls (65.2, 29.0) and (65.3, 26.9) .. (65.2,
      16.8).. controls (65.2, 10.2) and (65.3, 4.1) .. (65.4, 3.4).. controls
      (66.0, 0.2) and (67.5, -1.0) .. (68.4, 0.9).. controls (69.3, 2.8) and
      (69.9, 7.9) .. (69.9, 13.7).. controls (69.9, 22.7) and (70.3, 28.2)
      .. (71.1, 29.8) -- (71.8, 31.2) -- (99.5, 31.4).. controls (126.6, 31.6)
      and (127.2, 31.6) .. (128.4, 32.3).. controls (130.9, 33.8) and (129.9,
      35.9) .. (126.3, 36.7).. controls (125.1, 36.9) and (116.9, 37.0)
      .. (102.7, 37.0) -- (81.0, 37.0) -- (82.9, 38.7).. controls (88.5, 43.4)
      and (94.1, 50.2) .. (96.3, 54.9).. controls (98.2, 58.9) and (98.9,
      61.3) .. (98.9, 64.6).. controls (98.9, 68.8) and (96.9, 72.6) .. (93.6,
      74.8).. controls (91.9, 75.8) and (91.7, 75.8) .. (87.8, 75.8) -- (83.7,
      75.8) -- (80.6, 73.8).. controls (77.2, 71.5) and (74.0, 68.5) .. (72.6,
      65.9).. controls (70.4, 62.3) and (68.2, 54.3) .. (66.5,
      45.0).. controls (65.5, 39.2) and (65.2, 38.3) .. (64.0,
      37.5).. controls (62.2, 36.3) and (57.7, 37.9) .. (54.5,
      41.0).. controls (52.7, 42.7) and (46.2, 51.8) .. (44.7,
      54.9).. controls (43.0, 58.2) and (42.3, 65.9) .. (43.1,
      71.4).. controls (43.8, 75.6) and (43.7, 75.8) .. (41.8,
      75.8).. controls (40.5, 75.8) and (40.1, 75.7) .. (39.9, 75.1) --
      cycle(91.2, 70.6).. controls (93.4, 69.1) and (94.7, 66.7) .. (94.8,
      63.9).. controls (95.0, 61.6) and (94.9, 61.3) .. (93.5,
      58.7).. controls (92.0, 55.8) and (85.9, 47.9) .. (81.4,
      42.8).. controls (79.4, 40.7) and (78.1, 39.7) .. (76.3,
      38.7).. controls (73.1, 37.0) and (71.6, 36.8) .. (70.7,
      37.7).. controls (69.2, 39.4) and (70.3, 48.0) .. (73.2,
      56.4).. controls (75.8, 64.2) and (78.7, 68.0) .. (84.0,
      70.6).. controls (87.5, 72.2) and (88.8, 72.2) .. (91.2, 70.6) -- cycle;
    \end{scope}
  \end{tikzpicture}
}

\def\SngRedAseven{%
  \begin{tikzpicture}[y=1cm, x=1cm,
    yscale=0.9*\globalscale,xscale=\globalscale,%
    baseline={([yshift=-1mm]current bounding box.center)}]
    \begin{scope}[fill=black,cm={1,0,0,1,(0, 0)}]
      \path[fill=black] (53.6, 106.7).. controls (39.7, 104.7) and (27.9,
      99.6) .. (18.1, 91.3).. controls (9.1, 83.8) and (2.6, 72.9) .. (0.6,
      62.1).. controls (0.2, 60.2) and (0.0, 57.2) .. (0.0, 53.6).. controls
      (0.0, 47.2) and (0.6, 43.1) .. (2.5, 37.9).. controls (9.3, 18.6) and
      (28.6, 4.1) .. (52.1, 0.5).. controls (56.9, -0.2) and (69.8, -0.2)
      .. (74.5, 0.5).. controls (97.9, 4.1) and (117.4, 18.7) .. (124.1,
      37.9).. controls (126.0, 43.1) and (126.6, 47.2) .. (126.6,
      53.6).. controls (126.6, 57.2) and (126.4, 60.2) .. (126.1,
      62.1).. controls (124.0, 73.1) and (117.4, 83.9) .. (108.1,
      91.7).. controls (98.5, 99.7) and (87.0, 104.6) .. (73.1,
      106.7).. controls (68.8, 107.4) and (58.0, 107.4) .. (53.6, 106.7) --
      cycle(73.7, 102.0).. controls (77.3, 101.0) and (82.9, 98.1) .. (85.8,
      96.0).. controls (100.0, 85.2) and (104.4, 65.9) .. (96.3,
      50.1).. controls (88.4, 34.7) and (70.6, 26.9) .. (53.9,
      31.4).. controls (52.5, 31.8) and (49.6, 33.0) .. (47.4,
      34.0).. controls (31.6, 41.8) and (23.5, 59.8) .. (28.4,
      76.8).. controls (32.1, 89.7) and (42.7, 99.6) .. (55.7,
      102.5).. controls (59.3, 103.3) and (70.2, 103.0) .. (73.7, 102.0) --
      cycle(34.0, 94.8).. controls (28.2, 88.7) and (24.6, 81.3) .. (23.2,
      73.1).. controls (22.6, 69.2) and (22.8, 62.6) .. (23.6,
      58.6).. controls (28.3, 35.5) and (51.6, 21.2) .. (74.4,
      27.3).. controls (90.8, 31.8) and (102.8, 46.0) .. (104.4,
      63.1).. controls (105.5, 74.9) and (101.4, 86.4) .. (93.2, 95.0) --
      (90.9, 97.4) -- (92.3, 96.8).. controls (99.3, 93.9) and (108.5, 86.6)
      .. (113.4, 80.1).. controls (121.2, 69.9) and (124.3, 57.0) .. (121.9,
      45.3).. controls (118.9, 30.7) and (108.5, 18.2) .. (93.3,
      10.7).. controls (68.2, -1.6) and (36.3, 3.1) .. (17.6, 21.9).. controls
      (3.4, 36.3) and (0.0, 55.8) .. (8.9, 73.1).. controls (12.8, 80.8) and
      (18.6, 87.1) .. (26.5, 92.4).. controls (29.1, 94.1) and (35.8, 97.7)
      .. (36.5, 97.7).. controls (36.6, 97.7) and (35.5, 96.4) .. (34.0, 94.8)
      -- cycle;
    \end{scope}
  \end{tikzpicture}
}

\def\SngRedEseven{%
  \begin{tikzpicture}[y=1cm, x=1cm,
    yscale=\globalscale,xscale=\globalscale,
    baseline={([yshift=-1.5mm]current bounding box.center)}]
    \begin{scope}[fill=black,cm={1,0,0,1,(0, 0)}]
      \path[fill=black] (44.6, 108.8) -- (43.7, 108.0) -- (43.7,
      90.8).. controls (43.7, 71.8) and (43.7, 71.4) .. (41.4,
      66.4).. controls (38.9, 60.9) and (36.9, 58.8) .. (27.5,
      51.7).. controls (21.1, 46.9) and (18.0, 45.2) .. (10.4,
      42.5).. controls (3.9, 40.2) and (1.2, 38.9) .. (0.5, 37.9).. controls
      (0.0, 37.3) and (0.1, 37.0) .. (0.6, 36.3).. controls (1.2, 35.7) and
      (1.6, 35.6) .. (3.4, 35.6).. controls (10.0, 35.6) and (21.5, 39.9)
      .. (30.6, 46.0).. controls (34.8, 48.8) and (37.5, 51.4) .. (40.3,
      55.2).. controls (41.4, 56.8) and (42.7, 58.2) .. (43.0, 58.5) -- (43.7,
      58.9) -- (43.7, 30.3) -- (43.7, 1.7) -- (44.6, 0.9).. controls (45.6,
      -0.1) and (46.1, -0.2) .. (47.2, 0.6) -- (48.0, 1.1) -- (48.0,
      29.7).. controls (48.0, 49.1) and (48.1, 58.2) .. (48.3,
      58.0).. controls (48.5, 57.9) and (49.7, 56.3) .. (50.9,
      54.6).. controls (56.1, 46.8) and (67.2, 38.2) .. (75.6,
      35.7).. controls (78.6, 34.7) and (79.6, 34.7) .. (80.4,
      35.6).. controls (81.4, 36.6) and (81.4, 36.8) .. (80.0,
      38.3).. controls (78.7, 39.9) and (76.1, 41.6) .. (71.4,
      44.1).. controls (68.3, 45.9) and (67.3, 46.6) .. (62.1,
      51.9).. controls (56.1, 57.8) and (52.3, 62.4) .. (50.7,
      65.8).. controls (48.4, 70.4) and (48.0, 75.1) .. (48.0, 94.3) -- (48.0,
      108.6) -- (47.2, 109.1).. controls (46.1, 109.9) and (45.6, 109.9)
      .. (44.6, 108.8) -- cycle;
    \end{scope}
  \end{tikzpicture}
}

\def\SngRedXnine{%
  \begin{tikzpicture}[y=1cm, x=1cm,
    yscale=0.9*\globalscale,xscale=0.9*\globalscale,
    baseline={([yshift=-1mm]current bounding box.center)}]
    \begin{scope}[fill=black,cm={1,0,0,1,(0, 0)}]
      \path[fill=black] (75.5, 117.1).. controls (74.8, 115.6) and (74.8,
      114.5) .. (74.8, 91.2) -- (74.8, 66.9) -- (51.4, 90.4).. controls (28.3,
      113.6) and (26.0, 115.8) .. (24.5, 115.2).. controls (23.7, 114.9) and
      (23.9, 113.5) .. (25.1, 111.6).. controls (25.8, 110.7) and (36.4, 99.7)
      .. (48.8, 87.1) -- (71.2, 64.4) -- (37.6, 64.2).. controls (3.1, 64.0)
      and (2.9, 64.0) .. (1.8, 62.7).. controls (1.4, 62.1) and (1.4, 62.0)
      .. (1.8, 61.6).. controls (3.3, 60.2) and (3.6, 60.2) .. (36.2, 60.0) --
      (67.6, 59.8) -- (40.2, 43.6).. controls (25.2, 34.7) and (12.0, 26.8)
      .. (11.0, 26.0).. controls (8.7, 24.3) and (7.6, 22.7) .. (8.2,
      21.9).. controls (8.7, 21.3) and (10.3, 21.6) .. (12.9, 22.7).. controls
      (13.9, 23.2) and (27.9, 31.4) .. (44.1, 41.0).. controls (60.3, 50.7)
      and (73.8, 58.6) .. (74.2, 58.7).. controls (74.7, 59.0) and (74.8,
      57.2) .. (74.9, 31.9).. controls (75.0, 6.6) and (75.0, 4.7) .. (75.6,
      3.8).. controls (76.5, 2.4) and (77.5, 2.5) .. (78.2, 4.1).. controls
      (78.8, 5.3) and (78.8, 8.0) .. (79.0, 30.9) -- (79.2, 56.4) -- (101.2,
      34.0).. controls (113.4, 21.7) and (124.1, 11.0) .. (125.1,
      10.2).. controls (126.2, 9.4) and (127.5, 8.7) .. (128.3, 8.6) --
      (129.6, 8.4) -- (129.4, 9.7).. controls (129.3, 10.4) and (128.7, 11.6)
      .. (128.1, 12.5).. controls (127.5, 13.3) and (116.7, 24.3) .. (104.2,
      37.0) -- (81.4, 60.0) -- (108.9, 60.0).. controls (125.2, 60.0) and
      (137.3, 60.1) .. (138.8, 60.3).. controls (143.2, 61.0) and (144.2,
      62.1) .. (141.6, 63.3).. controls (140.3, 64.0) and (139.3, 64.0)
      .. (111.8, 64.2) -- (83.5, 64.4) -- (106.4, 78.0).. controls (119.0,
      85.4) and (130.4, 92.2) .. (131.6, 93.0).. controls (134.6, 94.9) and
      (136.9, 97.0) .. (136.9, 98.0).. controls (136.9, 98.6) and (136.7,
      98.8) .. (135.6, 98.8).. controls (134.9, 98.8) and (133.5, 98.4)
      .. (132.4, 98.0).. controls (131.3, 97.5) and (119.0, 90.3) .. (105.0,
      82.1).. controls (91.1, 73.8) and (79.5, 66.9) .. (79.2,
      66.8).. controls (78.8, 66.6) and (78.7, 70.2) .. (78.6,
      91.6).. controls (78.5, 114.9) and (78.5, 116.7) .. (77.9,
      117.6).. controls (77.0, 118.9) and (76.3, 118.8) .. (75.5, 117.1) --
      cycle;
    \end{scope}
  \end{tikzpicture}
}

\def\SngRedAoneAthree{%
  \begin{tikzpicture}[y=1cm, x=1cm,
    yscale=1.2*\globalscale,xscale=\globalscale,%
    baseline={([yshift=-1mm]current bounding box.center)}]
    \begin{scope}[fill=black,cm={1,0,0,1,(0, 0)}]
      \path[fill=black] (82.7, 49.0).. controls (78.1, 47.4) and (70.9, 42.6)
      .. (67.3, 38.8).. controls (66.1, 37.6) and (62.9, 34.0) .. (60.3,
      31.0).. controls (51.5, 20.9) and (51.7, 21.1) .. (46.2,
      20.4).. controls (42.2, 19.9) and (39.0, 20.2) .. (36.9,
      21.1).. controls (33.3, 22.7) and (29.9, 27.2) .. (29.0,
      31.5).. controls (28.8, 32.4) and (28.3, 35.8) .. (28.0,
      39.1).. controls (27.4, 46.0) and (27.2, 47.1) .. (26.6,
      48.4).. controls (26.0, 49.6) and (25.6, 49.6) .. (24.7,
      48.7).. controls (23.7, 47.7) and (23.7, 44.6) .. (24.7,
      36.8).. controls (25.9, 27.1) and (26.8, 24.4) .. (30.4, 21.4) -- (32.5,
      19.8) -- (17.0, 19.8) -- (1.4, 19.8) -- (0.7, 18.8).. controls (-0.2,
      17.7) and (-0.2, 17.1) .. (0.7, 16.2).. controls (1.4, 15.5) and (1.9,
      15.5) .. (49.4, 15.5) -- (97.3, 15.5) -- (97.5, 14.4).. controls (97.6,
      13.8) and (97.8, 11.0) .. (97.8, 8.3).. controls (98.0, 2.6) and (98.6,
      0.4) .. (100.0, 0.4).. controls (100.5, 0.4) and (100.9, 0.7) .. (101.2,
      1.3).. controls (101.7, 2.4) and (102.3, 8.1) .. (102.3, 12.5) --
      (102.3, 15.5) -- (112.9, 15.5).. controls (123.0, 15.5) and (123.5,
      15.6) .. (124.2, 16.2).. controls (125.1, 17.1) and (125.1, 17.7)
      .. (124.1, 18.8) -- (123.4, 19.8) -- (112.6, 19.8) -- (101.7, 19.8) --
      (101.5, 22.3).. controls (100.8, 28.6) and (98.2, 37.9) .. (95.8,
      42.5).. controls (93.7, 46.6) and (89.4, 49.7) .. (86.1,
      49.7).. controls (85.2, 49.7) and (83.6, 49.4) .. (82.7, 49.0) --
      cycle(89.1, 44.2).. controls (90.9, 42.8) and (92.8, 39.6) .. (94.3,
      34.9).. controls (95.6, 31.0) and (97.0, 24.0) .. (97.0, 21.3) -- (97.0,
      19.8) -- (76.7, 19.8) -- (56.4, 19.8) -- (58.6, 22.3).. controls (59.9,
      23.7) and (62.6, 27.0) .. (64.7, 29.5).. controls (73.7, 40.1) and
      (80.5, 45.1) .. (85.9, 45.2).. controls (87.4, 45.2) and (88.0, 44.9)
      .. (89.1, 44.2) -- cycle;
    \end{scope}
  \end{tikzpicture}
}

\def\SngRedAoneDfour{%
  \begin{tikzpicture}[y=1cm, x=1cm,
    yscale=1.2*\globalscale,xscale=\globalscale,%
    baseline={([yshift=-1mm]current bounding box.center)}]
    \begin{scope}[fill=black,cm={1,0,0,1,(0, 0)}]
      \path[fill=black] (36.7, 71.2).. controls (35.9, 69.9) and (36.5, 62.0)
      .. (38.5, 47.8).. controls (39.4, 40.4) and (40.3, 34.1) .. (40.4,
      33.8).. controls (40.5, 33.4) and (40.7, 32.8) .. (40.8,
      32.5).. controls (41.0, 31.8) and (40.8, 31.8) .. (21.2,
      31.8).. controls (1.9, 31.8) and (1.4, 31.8) .. (0.7, 31.0).. controls
      (-0.1, 30.2) and (-0.2, 29.3) .. (0.5, 28.4).. controls (1.0, 27.7) and
      (1.7, 27.7) .. (21.4, 27.5) -- (41.8, 27.3) -- (42.1, 25.9).. controls
      (43.8, 18.1) and (45.6, 12.4) .. (47.4, 9.3).. controls (49.4, 5.9) and
      (53.8, 2.4) .. (57.6, 1.1).. controls (64.0, -1.1) and (75.5, 2.2)
      .. (84.8, 8.7).. controls (88.8, 11.6) and (94.4, 17.3) .. (96.9,
      21.1).. controls (98.1, 22.8) and (99.6, 24.6) .. (100.1,
      25.0).. controls (101.0, 25.7) and (101.0, 25.8) .. (101.7,
      25.2).. controls (102.9, 24.1) and (109.4, 15.3) .. (113.1,
      9.7).. controls (115.1, 6.8) and (117.3, 3.9) .. (117.9, 3.2).. controls
      (120.4, 0.8) and (121.4, 2.7) .. (119.6, 6.3).. controls (118.4, 8.8)
      and (112.7, 17.1) .. (108.5, 22.9).. controls (106.8, 25.1) and (105.5,
      27.0) .. (105.5, 27.2).. controls (105.5, 27.4) and (114.4, 27.5)
      .. (127.4, 27.5).. controls (148.8, 27.5) and (149.2, 27.5) .. (149.9,
      28.2).. controls (150.3, 28.6) and (150.6, 29.2) .. (150.6,
      29.6).. controls (150.6, 30.0) and (150.3, 30.7) .. (149.9,
      31.0).. controls (149.2, 31.8) and (148.8, 31.8) .. (127.4,
      31.8).. controls (115.3, 31.8) and (105.5, 31.8) .. (105.5,
      31.9).. controls (105.5, 32.0) and (106.2, 33.2) .. (107.0,
      34.5).. controls (112.8, 43.4) and (117.8, 56.6) .. (117.8,
      63.0).. controls (117.8, 66.7) and (116.1, 70.2) .. (113.6,
      71.5).. controls (110.6, 73.0) and (105.9, 71.4) .. (102.0,
      67.5).. controls (99.2, 64.7) and (97.6, 61.6) .. (96.0,
      56.2).. controls (94.7, 51.7) and (94.2, 48.3) .. (94.2,
      43.4).. controls (94.2, 39.2) and (94.7, 36.9) .. (96.3,
      34.0).. controls (96.9, 33.1) and (97.4, 32.1) .. (97.4,
      32.0).. controls (97.4, 31.9) and (85.7, 31.8) .. (71.5, 31.8) -- (45.6,
      31.8) -- (45.2, 34.1).. controls (44.0, 41.1) and (41.2, 59.3) .. (40.7,
      63.3).. controls (40.2, 68.1) and (39.8, 69.9) .. (38.9,
      71.2).. controls (38.2, 72.2) and (37.2, 72.2) .. (36.7, 71.2) --
      cycle(111.8, 66.9).. controls (113.1, 66.2) and (113.7, 64.4) .. (113.4,
      62.5).. controls (112.6, 57.5) and (109.9, 50.0) .. (106.0,
      41.8).. controls (103.0, 35.3) and (101.5, 34.3) .. (99.9,
      37.4).. controls (98.7, 39.9) and (98.0, 43.8) .. (98.2,
      46.9).. controls (98.5, 49.7) and (100.5, 56.3) .. (102.3,
      60.1).. controls (105.1, 66.0) and (108.6, 68.5) .. (111.8, 66.9) --
      cycle(95.1, 26.6).. controls (93.4, 24.0) and (89.2, 18.6) .. (87.0,
      16.3).. controls (83.0, 12.1) and (77.8, 9.0) .. (70.7, 6.7).. controls
      (65.7, 5.0) and (62.5, 5.0) .. (59.1, 6.6).. controls (52.9, 9.6) and
      (49.7, 14.2) .. (47.5, 23.7) -- (46.6, 27.5) -- (71.2, 27.5) -- (95.7,
      27.5) -- (95.1, 26.6) -- cycle;
    \end{scope}
  \end{tikzpicture}
}

\def\SngRedAoneAfiveCubic{%
  \begin{tikzpicture}[y=1cm, x=1cm,
    yscale=1.2*\globalscale,xscale=\globalscale,%
    baseline={([yshift=-1mm]current bounding box.center)}]
    \begin{scope}[fill=black,cm={1,0,0,1,(0, 0)}]
      \path[fill=black] (124.4, 78.5).. controls (119.7, 76.8) and (117.1, 72.8) .. (117.1, 67.4).. controls (117.1, 64.6) and (117.7, 61.7) .. (119.2, 57.0).. controls (120.2, 54.0) and (120.3, 53.4) .. (119.9, 52.5).. controls (119.4, 51.3) and (116.5, 48.5) .. (110.6, 43.8).. controls (105.6, 39.7) and (101.1, 37.3) .. (95.8, 35.8) -- (92.6, 34.9) -- (47.1, 34.9) -- (1.4, 34.9) -- (0.7, 34.0).. controls (0.0, 33.1) and (0.0, 32.0) .. (0.8, 31.2).. controls (1.0, 31.0) and (10.8, 30.8) .. (28.1, 30.7) -- (55.1, 30.5) -- (53.1, 29.1).. controls (49.8, 27.0) and (42.9, 20.0) .. (40.7, 16.7).. controls (35.8, 9.3) and (32.8, 1.0) .. (34.8, 0.2).. controls (36.2, -0.4) and (37.9, 2.0) .. (40.2, 7.9).. controls (40.8, 9.4) and (41.7, 11.3) .. (42.2, 12.1).. controls (45.5, 17.4) and (51.2, 23.2) .. (55.9, 26.4).. controls (61.7, 30.2) and (61.2, 30.1) .. (74.4, 30.4).. controls (80.8, 30.5) and (100.0, 30.7) .. (117.3, 30.7).. controls (141.9, 30.7) and (148.9, 30.8) .. (149.3, 31.2).. controls (150.2, 31.9) and (150.4, 33.0) .. (149.6, 34.0) -- (148.8, 34.9) -- (127.2, 35.0) -- (105.7, 35.1) -- (107.3, 36.1).. controls (108.2, 36.6) and (110.0, 37.9) .. (111.4, 39.1).. controls (114.2, 41.4) and (124.0, 50.3) .. (124.4, 50.8).. controls (124.8, 51.5) and (128.7, 51.2) .. (135.0, 49.9).. controls (141.9, 48.5) and (143.4, 48.5) .. (143.2, 49.8).. controls (143.0, 51.2) and (140.6, 52.1) .. (133.7, 53.4).. controls (130.1, 54.2) and (127.0, 54.9) .. (126.9, 55.1).. controls (126.7, 55.2) and (127.4, 57.0) .. (128.5, 59.3).. controls (130.9, 64.1) and (131.6, 66.4) .. (131.6, 69.4).. controls (131.6, 74.0) and (130.2, 77.7) .. (128.2, 78.5).. controls (126.8, 79.1) and (126.2, 79.1) .. (124.4, 78.5) -- cycle(127.2, 75.0).. controls (128.9, 72.4) and (127.7, 63.8) .. (124.9, 58.3).. controls (123.9, 56.3) and (123.1, 56.4) .. (122.3, 58.8).. controls (120.8, 63.8) and (120.6, 70.6) .. (122.0, 73.3).. controls (123.3, 75.8) and (126.1, 76.7) .. (127.2, 75.0) -- cycle;
    \end{scope}
  \end{tikzpicture}
}

\def\SngSngRedAoneAfiveConic{%
  \begin{tikzpicture}[y=1cm, x=1cm,
    yscale=0.7*\globalscale,xscale=0.75*\globalscale,%
    baseline={([yshift=-1mm]current bounding box.center)}]
    \begin{scope}[fill=black,cm={1,0,0,1,(0, 0)}]
      \path[fill=black] (136.0, 123.6).. controls (130.1, 123.0) and (127.2,
      122.6) .. (122.8, 121.7).. controls (113.4, 119.8) and (103.1, 116.2)
      .. (93.9, 111.7).. controls (83.1, 106.4) and (72.1, 100.7) .. (67.4,
      98.1).. controls (56.8, 92.0) and (44.1, 83.2) .. (34.2,
      74.9).. controls (28.4, 70.0) and (16.4, 58.0) .. (12.7,
      53.3).. controls (4.5, 43.0) and (0.2, 34.2) .. (0.2, 27.3).. controls
      (0.2, 24.8) and (0.4, 24.1) .. (1.3, 22.4).. controls (3.4, 18.3) and
      (7.5, 16.3) .. (14.5, 16.0).. controls (18.2, 15.8) and (20.2, 16.0)
      .. (26.8, 17.1).. controls (28.0, 17.3) and (28.1, 17.3) .. (29.2,
      15.3).. controls (31.8, 11.1) and (37.0, 6.5) .. (42.1, 4.2).. controls
      (48.4, 1.4) and (54.6, 0.2) .. (63.3, 0.2).. controls (102.9, 0.3) and
      (153.8, 32.7) .. (172.2, 69.5).. controls (178.0, 81.1) and (179.7,
      91.8) .. (177.2, 101.4).. controls (174.4, 112.4) and (164.3, 120.5)
      .. (150.2, 122.9).. controls (147.2, 123.4) and (138.6, 123.9)
      .. (136.0, 123.6) -- cycle(144.7, 116.8).. controls (155.8, 115.9) and
      (163.9, 112.0) .. (168.1, 105.6).. controls (168.9, 104.5) and (169.9,
      102.2) .. (170.5, 100.5).. controls (171.5, 97.8) and (171.6, 97.2)
      .. (171.6, 92.8).. controls (171.6, 87.6) and (171.1, 84.8) .. (169.3,
      79.6).. controls (162.6, 60.0) and (141.3, 37.4) .. (115.9,
      23.0).. controls (111.6, 20.6) and (101.8, 15.9) .. (97.9,
      14.4).. controls (83.4, 8.7) and (69.9, 6.2) .. (58.6, 7.1).. controls
      (47.8, 7.9) and (40.2, 11.4) .. (35.6, 17.7) -- (34.5, 19.2) -- (38.7,
      20.6).. controls (58.2, 27.5) and (81.2, 41.5) .. (100.4,
      58.0).. controls (106.0, 62.9) and (116.0, 73.1) .. (119.7,
      77.6).. controls (125.2, 84.6) and (129.2, 91.2) .. (131.1,
      96.8).. controls (132.3, 100.2) and (132.6, 105.1) .. (131.8,
      107.5).. controls (130.7, 110.6) and (128.2, 112.9) .. (124.6,
      114.2).. controls (123.5, 114.5) and (122.8, 114.8) .. (122.8,
      114.9).. controls (123.2, 115.1) and (131.4, 116.5) .. (134.8,
      116.7).. controls (140.1, 117.2) and (139.9, 117.2) .. (144.7, 116.8) --
      cycle(119.9, 110.5).. controls (122.2, 109.8) and (124.6, 107.7)
      .. (125.4, 105.8).. controls (126.8, 102.4) and (126.0, 97.4) .. (123.1,
      91.5).. controls (113.1, 70.9) and (80.1, 42.9) .. (48.5,
      28.3).. controls (43.1, 25.9) and (34.6, 22.6) .. (33.5,
      22.6).. controls (31.9, 22.6) and (31.2, 32.1) .. (32.3,
      37.7).. controls (36.7, 59.4) and (58.7, 84.4) .. (88.0,
      101.1).. controls (95.8, 105.5) and (107.0, 110.5) .. (110.6,
      111.1).. controls (113.0, 111.5) and (117.5, 111.2) .. (119.9, 110.5) --
      cycle(53.8, 84.0).. controls (38.0, 69.0) and (28.0, 52.4) .. (25.5,
      37.4).. controls (24.9, 33.4) and (25.0, 26.3) .. (25.9,
      23.3).. controls (26.2, 22.2) and (26.5, 21.2) .. (26.5,
      21.0).. controls (26.5, 20.3) and (20.1, 19.6) .. (16.7,
      19.9).. controls (11.3, 20.2) and (8.0, 22.4) .. (6.8, 26.4).. controls
      (3.7, 37.1) and (19.5, 58.7) .. (45.3, 78.9).. controls (48.9, 81.7) and
      (56.9, 87.5) .. (57.2, 87.5).. controls (57.3, 87.5) and (55.7, 85.9)
      .. (53.8, 84.0) -- cycle;
    \end{scope}
  \end{tikzpicture}
}

\def\SngRedAoneDfive{%
  \begin{tikzpicture}[y=1cm, x=1cm,
    yscale=1.*\globalscale,xscale=\globalscale,%
    baseline={([yshift=-1mm]current bounding box.center)}]
    \begin{scope}[fill=black,cm={1,0,0,1,(0, 0)}]
      \path[fill=black] (35.2, 81.1).. controls (34.2, 80.5) and (33.6, 79.1)
      .. (33.2, 76.2).. controls (32.7, 72.7) and (32.1, 63.3) .. (32.0, 55.9)
      -- (31.9, 49.9) -- (16.5, 49.8).. controls (1.4, 49.7) and (1.2, 49.7)
      .. (0.6, 49.0).. controls (-0.2, 48.0) and (-0.2, 47.4) .. (0.9, 46.4)
      -- (1.7, 45.5) -- (16.9, 45.5) -- (32.0, 45.5) -- (32.4,
      39.7).. controls (33.5, 25.2) and (37.3, 15.2) .. (44.7, 8.0).. controls
      (49.4, 3.3) and (53.8, 1.2) .. (61.0, 0.4).. controls (66.6, -0.4) and
      (73.5, 1.0) .. (78.3, 3.9).. controls (80.9, 5.4) and (88.8, 13.2)
      .. (92.0, 17.5).. controls (93.2, 19.0) and (95.4, 22.9) .. (97.6,
      27.5).. controls (103.2, 38.9) and (103.9, 39.3) .. (106.5,
      32.8).. controls (108.1, 29.1) and (110.6, 25.5) .. (113.0,
      23.6).. controls (117.3, 20.2) and (124.1, 17.3) .. (128.0,
      17.3).. controls (130.3, 17.3) and (131.0, 17.7) .. (130.8,
      19.1).. controls (130.7, 20.0) and (130.9, 19.9) .. (123.1,
      23.6).. controls (114.5, 27.7) and (110.9, 32.0) .. (107.5, 41.9) --
      (106.2, 45.5) -- (122.8, 45.5) -- (139.4, 45.5) -- (140.2,
      46.4).. controls (141.3, 47.4) and (141.3, 47.9) .. (140.5, 49.0) --
      (140.0, 49.7) -- (88.7, 49.7) -- (37.4, 49.7) -- (37.4, 54.2).. controls
      (37.4, 56.7) and (37.6, 62.1) .. (37.8, 66.3).. controls (38.2, 74.8)
      and (37.9, 79.7) .. (36.8, 80.9).. controls (36.2, 81.6) and (36.0,
      81.6) .. (35.2, 81.1) -- cycle(100.4, 44.0).. controls (99.0, 39.3) and
      (93.2, 27.7) .. (89.5, 22.1).. controls (85.3, 15.8) and (77.8, 9.1)
      .. (71.8, 6.3).. controls (65.7, 3.4) and (56.8, 4.7) .. (50.3,
      9.5).. controls (48.4, 10.9) and (44.9, 15.3) .. (43.4, 18.4).. controls
      (41.2, 22.8) and (39.6, 29.1) .. (38.5, 38.3).. controls (38.0, 41.6)
      and (37.6, 44.6) .. (37.5, 44.9).. controls (37.4, 45.5) and (39.0,
      45.5) .. (69.1, 45.5) -- (100.9, 45.5) -- (100.4, 44.0) -- cycle;
    \end{scope}
  \end{tikzpicture}
}

\def\SngRedAoneDsix{%
  \begin{tikzpicture}[y=1cm, x=1cm,
    yscale=0.9*\globalscale,xscale=0.9*\globalscale,%
    baseline={([yshift=-1mm]current bounding box.center)}]
    \begin{scope}[fill=black,cm={1,0,0,1,(0, 0)}]
      \path[fill=black] (124.9, 114.6).. controls (124.5, 114.0) and (124.1,
      112.7) .. (124.0, 111.7).. controls (123.9, 110.7) and (123.7, 99.8)
      .. (123.5, 87.5) -- (123.3, 65.2) -- (121.9, 67.2).. controls (120.0,
      70.2) and (115.4, 74.6) .. (111.8, 77.0).. controls (99.2, 85.4) and
      (80.7, 88.5) .. (63.7, 85.2).. controls (60.2, 84.5) and (54.0, 82.5)
      .. (50.9, 81.1).. controls (42.9, 77.5) and (36.1, 71.5) .. (32.9,
      65.2).. controls (31.6, 62.5) and (30.3, 58.1) .. (30.3, 56.2) -- (30.3,
      55.1) -- (16.7, 55.0).. controls (4.7, 54.9) and (2.8, 54.8) .. (1.7,
      54.3).. controls (0.1, 53.6) and (-0.1, 52.5) .. (1.3, 51.9).. controls
      (3.1, 51.1) and (7.2, 50.8) .. (18.8, 50.8).. controls (29.4, 50.8) and
      (30.7, 50.7) .. (30.7, 50.2).. controls (30.7, 49.0) and (32.6, 44.2)
      .. (33.9, 42.0).. controls (43.3, 26.1) and (69.1, 17.7) .. (93.1,
      22.6).. controls (98.2, 23.6) and (102.7, 25.2) .. (107.4,
      27.5).. controls (113.5, 30.6) and (118.1, 34.2) .. (121.6, 39.0) --
      (123.2, 41.1) -- (123.0, 21.9).. controls (122.9, 1.8) and (123.0, 0.7)
      .. (124.4, 0.2).. controls (125.1, -0.1) and (125.5, 0.2) .. (126.2,
      1.5).. controls (126.8, 2.5) and (126.9, 5.0) .. (127.1, 26.7) --
      (127.5, 50.7) -- (136.7, 50.9).. controls (146.3, 51.1) and (148.6,
      51.4) .. (149.4, 52.7).. controls (149.7, 53.2) and (149.6, 53.4)
      .. (149.1, 53.8).. controls (147.8, 54.8) and (144.9, 55.0) .. (136.1,
      55.0) -- (127.3, 55.0) -- (127.5, 66.3).. controls (128.0, 96.1) and
      (128.1, 111.0) .. (127.7, 112.7).. controls (127.3, 114.4) and (126.5,
      115.7) .. (125.9, 115.7).. controls (125.7, 115.7) and (125.3, 115.2)
      .. (124.9, 114.6) -- cycle(90.9, 80.6).. controls (108.1, 77.0) and
      (120.5, 67.5) .. (122.6, 56.4) -- (122.8, 55.0) -- (78.3, 55.0) --
      (33.8, 55.0) -- (34.0, 56.7).. controls (34.2, 57.6) and (34.7, 59.4)
      .. (35.1, 60.6).. controls (38.7, 70.2) and (50.4, 78.0) .. (65.2,
      80.8).. controls (70.6, 81.8) and (70.7, 81.8) .. (79.0,
      81.7).. controls (85.5, 81.6) and (87.2, 81.4) .. (90.9, 80.6) --
      cycle(122.4, 50.0).. controls (122.4, 48.9) and (121.0, 45.5) .. (119.5,
      43.2).. controls (117.9, 40.5) and (114.1, 36.7) .. (110.6,
      34.4).. controls (93.1, 22.7) and (63.8, 23.0) .. (46.1,
      35.0).. controls (40.4, 38.8) and (35.6, 45.0) .. (34.4, 49.8) -- (34.2,
      50.8) -- (78.3, 50.8) -- (122.4, 50.8) -- (122.4, 50.0) -- cycle;
    \end{scope}
  \end{tikzpicture}
}

\def\SngRedAtwoAfive{%
  \begin{tikzpicture}[y=1cm, x=1cm,
    yscale=1.2*\globalscale,xscale=\globalscale,%
    baseline={([yshift=-1mm]current bounding box.center)}]
    \begin{scope}[fill=black,cm={1,0,0,1,(0, 0)}]
      \path[fill=black] (118.1, 101.0).. controls (118.0, 100.7) and (117.4,
      98.3) .. (116.9, 95.8).. controls (115.9, 90.5) and (111.9, 78.1)
      .. (109.6, 73.3).. controls (106.7, 67.0) and (102.5, 61.2) .. (96.0,
      54.5).. controls (88.8, 47.2) and (85.1, 44.6) .. (80.6,
      43.7).. controls (79.2, 43.5) and (66.4, 43.3) .. (40.0,
      43.1).. controls (5.8, 43.0) and (1.5, 42.9) .. (1.0, 42.4).. controls
      (0.1, 41.7) and (0.2, 39.7) .. (1.1, 38.8).. controls (1.8, 38.1) and
      (2.2, 38.1) .. (22.8, 38.1) -- (43.9, 38.1) -- (41.1, 36.9).. controls
      (37.5, 35.3) and (34.0, 32.6) .. (28.4, 27.0).. controls (21.9, 20.5)
      and (17.1, 14.2) .. (13.6, 7.2).. controls (11.7, 3.6) and (11.3, 1.3)
      .. (12.1, 0.4).. controls (13.4, -0.9) and (15.8, 1.5) .. (19.5,
      8.0).. controls (24.6, 16.9) and (31.6, 24.8) .. (38.6, 29.5).. controls
      (42.0, 31.7) and (48.4, 34.9) .. (52.3, 36.1).. controls (58.3, 38.1)
      and (58.1, 38.1) .. (97.5, 38.1).. controls (133.2, 38.1) and (133.7,
      38.1) .. (134.4, 38.8).. controls (135.3, 39.7) and (135.3, 41.7)
      .. (134.5, 42.4).. controls (133.9, 43.0) and (131.7, 43.0) .. (111.7,
      43.1) -- (89.5, 43.2) -- (91.6, 44.8).. controls (94.8, 47.2) and (97.1,
      49.5) .. (101.7, 54.8).. controls (108.4, 62.7) and (112.9, 70.1)
      .. (115.5, 78.2).. controls (116.2, 80.3) and (117.0, 82.4) .. (117.3,
      82.8).. controls (117.8, 83.4) and (117.9, 83.4) .. (118.3,
      82.8).. controls (118.6, 82.4) and (119.3, 80.7) .. (119.9,
      79.0).. controls (121.5, 74.4) and (123.1, 71.4) .. (125.5,
      68.4).. controls (128.8, 64.3) and (131.6, 62.1) .. (133.4,
      62.1).. controls (135.6, 62.1) and (135.0, 63.7) .. (131.2,
      67.9).. controls (124.2, 75.7) and (121.9, 80.8) .. (120.5,
      91.7).. controls (119.7, 97.9) and (119.2, 100.4) .. (118.7,
      101.1).. controls (118.3, 101.7) and (118.3, 101.7) .. (118.1, 101.0) --
      cycle;
    \end{scope}
  \end{tikzpicture}
}

\def\SngRedAthreeAthree{%
  \begin{tikzpicture}[y=1cm, x=1cm,
    yscale=\globalscale,xscale=1.1*\globalscale,%
    baseline={([yshift=-1mm]current bounding box.center)}]
    \begin{scope}[fill=black,cm={1,0,0,1,(0, 0)}]
      \path[fill=black] (46.7, 87.8).. controls (27.7, 85.0) and (12.9, 71.4)
      .. (9.0, 52.9).. controls (8.1, 48.8) and (8.1, 40.5) .. (9.0,
      36.4).. controls (11.8, 23.1) and (20.4, 12.0) .. (32.6, 5.9).. controls
      (48.2, -1.8) and (67.3, 0.0) .. (81.0, 10.6).. controls (83.3, 12.3) and
      (87.7, 16.8) .. (89.4, 19.1).. controls (95.8, 27.6) and (99.0, 39.6)
      .. (97.5, 50.1).. controls (94.8, 69.6) and (79.7, 84.7) .. (59.8,
      87.7).. controls (56.0, 88.2) and (50.1, 88.3) .. (46.7, 87.8) --
      cycle(62.4, 82.9).. controls (72.8, 80.4) and (81.3, 74.6) .. (87.1,
      66.1).. controls (88.7, 63.7) and (91.5, 58.1) .. (92.0, 56.1) -- (92.3,
      54.8) -- (89.7, 56.3).. controls (77.0, 63.5) and (52.1, 65.9) .. (32.1,
      61.7).. controls (25.4, 60.4) and (18.7, 57.9) .. (14.9, 55.4) -- (13.7,
      54.6) -- (14.0, 55.6).. controls (14.1, 56.2) and (14.6, 57.7) .. (15.2,
      58.9).. controls (20.3, 71.7) and (32.1, 81.0) .. (46.1,
      83.4).. controls (50.3, 84.2) and (58.2, 83.9) .. (62.4, 82.9) --
      cycle(66.5, 58.8).. controls (77.6, 57.3) and (86.5, 54.2) .. (90.8,
      50.4).. controls (92.8, 48.6) and (93.5, 47.4) .. (93.5,
      45.4).. controls (93.5, 43.5) and (92.6, 42.1) .. (90.3,
      40.2).. controls (79.4, 31.6) and (49.9, 28.9) .. (28.9,
      34.4).. controls (22.0, 36.3) and (15.9, 39.5) .. (13.8,
      42.4).. controls (8.6, 49.5) and (22.5, 57.3) .. (44.1, 59.3).. controls
      (48.8, 59.7) and (61.9, 59.4) .. (66.5, 58.8) -- cycle(19.1,
      33.5).. controls (25.6, 30.4) and (35.1, 28.3) .. (46.2,
      27.5).. controls (62.7, 26.3) and (81.8, 29.7) .. (90.8,
      35.4).. controls (91.9, 36.1) and (92.8, 36.5) .. (92.8,
      36.3).. controls (92.8, 35.5) and (90.6, 29.5) .. (89.5,
      27.3).. controls (85.8, 19.9) and (79.1, 13.4) .. (71.3, 9.6).. controls
      (65.6, 6.7) and (57.6, 5.0) .. (51.2, 5.4).. controls (44.5, 5.8) and
      (39.7, 7.1) .. (33.8, 10.1).. controls (27.2, 13.5) and (21.7, 18.7)
      .. (17.8, 25.1).. controls (16.1, 27.9) and (13.9, 33.4) .. (13.5, 35.6)
      -- (13.3, 36.8) -- (14.4, 36.0).. controls (15.0, 35.6) and (17.1, 34.5)
      .. (19.1, 33.5) -- cycle;
    \end{scope}
  \end{tikzpicture}
}

\def\SngRedAonepthree{%
  \begin{tikzpicture}[y=1cm, x=1cm,
    yscale=1.2*\globalscale,xscale=\globalscale,%
    baseline={([yshift=-1mm]current bounding box.center)}]
    \begin{scope}[fill=black,cm={1,0,0,1,(0, 0)}]
      \path[fill=black] (32.3, 87.1).. controls (31.7, 86.2) and (32.7, 84.1)
      .. (35.3, 80.4).. controls (39.4, 74.6) and (40.6, 73.3) .. (45.0,
      69.6).. controls (49.4, 65.8) and (52.1, 64.2) .. (56.4,
      62.7).. controls (59.1, 61.7) and (59.9, 61.6) .. (63.5,
      61.6).. controls (67.8, 61.6) and (70.8, 62.2) .. (75.0,
      64.0).. controls (80.4, 66.3) and (86.7, 65.8) .. (91.1,
      62.9).. controls (91.9, 62.3) and (93.7, 61.2) .. (95.0,
      60.5).. controls (96.4, 59.8) and (98.4, 58.3) .. (99.4,
      57.4).. controls (101.7, 55.3) and (104.6, 54.2) .. (105.2,
      55.1).. controls (106.2, 56.7) and (104.8, 58.2) .. (99.6,
      61.5).. controls (91.7, 66.5) and (87.7, 68.1) .. (82.5,
      68.3).. controls (78.5, 68.6) and (76.6, 68.2) .. (71.6,
      66.3).. controls (68.5, 65.1) and (67.4, 64.8) .. (64.5,
      64.6).. controls (60.1, 64.4) and (57.3, 65.2) .. (53.5,
      67.8).. controls (50.4, 69.8) and (43.7, 76.7) .. (39.2,
      82.4).. controls (37.7, 84.3) and (35.9, 86.3) .. (35.3,
      86.7).. controls (34.2, 87.5) and (32.6, 87.7) .. (32.3, 87.1) -- cycle;
      \path[fill=black] (82.7, 49.0).. controls (76.6, 46.9) and (68.9, 41.2)
      .. (64.1, 35.4) -- (63.2, 34.2) -- (45.9, 34.2).. controls (29.7, 34.2)
      and (28.6, 34.3) .. (28.4, 34.9).. controls (28.4, 35.2) and (28.2,
      36.9) .. (28.0, 38.6).. controls (27.6, 45.6) and (26.8, 49.4) .. (25.8,
      49.4).. controls (25.5, 49.4) and (25.1, 49.1) .. (24.7,
      48.7).. controls (23.7, 47.7) and (23.7, 44.1) .. (24.9, 35.6) -- (25.0,
      34.2) -- (13.2, 34.2).. controls (1.9, 34.2) and (1.4, 34.2) .. (0.7,
      33.5).. controls (-0.2, 32.6) and (-0.2, 31.6) .. (0.7, 30.7).. controls
      (1.4, 30.0) and (1.9, 30.0) .. (13.5, 30.0) -- (25.7, 30.0) -- (25.9,
      28.8).. controls (26.8, 24.2) and (31.2, 19.8) .. (37.4,
      17.4).. controls (39.3, 16.6) and (40.3, 16.5) .. (43.6,
      16.4).. controls (52.0, 16.4) and (54.9, 17.8) .. (61.0, 25.2) -- (65.1,
      30.0) -- (80.4, 29.9) -- (95.8, 29.8) -- (96.2, 27.7).. controls (96.8,
      24.6) and (97.4, 18.1) .. (97.7, 10.2).. controls (98.0, 2.5) and (98.5,
      0.4) .. (100.0, 0.4).. controls (101.1, 0.4) and (101.6, 1.7) .. (102.0,
      6.3).. controls (102.5, 11.9) and (101.8, 23.0) .. (100.5,
      28.0).. controls (100.4, 28.8) and (100.2, 29.6) .. (100.2,
      29.7).. controls (100.2, 29.9) and (105.4, 30.0) .. (111.8,
      30.0).. controls (123.0, 30.0) and (123.5, 30.0) .. (124.2,
      30.7).. controls (125.0, 31.5) and (125.1, 32.4) .. (124.4,
      33.4).. controls (123.9, 34.0) and (123.2, 34.0) .. (111.5, 34.2) --
      (99.1, 34.4) -- (97.9, 37.6).. controls (96.3, 41.9) and (95.0, 44.3)
      .. (93.1, 46.2).. controls (89.7, 49.5) and (86.5, 50.4) .. (82.7, 49.0)
      -- cycle(88.9, 44.3).. controls (89.6, 43.8) and (90.7, 42.6) .. (91.4,
      41.7).. controls (92.5, 39.9) and (94.7, 34.6) .. (94.4,
      34.4).. controls (94.4, 34.3) and (88.6, 34.3) .. (81.6, 34.3) -- (69.0,
      34.4) -- (71.8, 37.1).. controls (78.7, 43.9) and (85.4, 46.7) .. (88.9,
      44.3) -- cycle(56.5, 26.7).. controls (51.6, 21.0) and (50.2, 20.4)
      .. (42.7, 20.3).. controls (39.2, 20.3) and (38.4, 20.4) .. (37.0,
      21.1).. controls (34.0, 22.4) and (31.2, 25.5) .. (29.8, 28.7) -- (29.3,
      30.0) -- (44.3, 30.0) -- (59.3, 30.0) -- (56.5, 26.7) -- cycle;
    \end{scope}
  \end{tikzpicture}
}

\def\SngRedAoneptwoAthreeCubic{%
  \begin{tikzpicture}[y=1cm, x=1cm,
    yscale=1.2*\globalscale,xscale=\globalscale,%
    baseline={([yshift=-1mm]current bounding box.center)}]
    \begin{scope}[fill=black,cm={1,0,0,1,(0, 0)}]
      \path[fill=black] (114.1, 92.5).. controls (112.5, 91.7) and (111.3,
      90.4) .. (109.8, 87.8).. controls (107.6, 83.9) and (107.0, 79.3)
      .. (108.1, 73.9).. controls (108.8, 70.7) and (109.5, 69.0) .. (111.6,
      65.6).. controls (114.8, 60.3) and (114.8, 60.1) .. (113.2,
      57.1).. controls (110.8, 52.4) and (103.6, 42.5) .. (100.1,
      39.3).. controls (96.9, 36.3) and (91.9, 33.0) .. (89.6,
      32.5).. controls (84.7, 31.3) and (82.6, 33.5) .. (78.2,
      44.4).. controls (76.1, 49.4) and (75.5, 50.7) .. (73.3,
      53.9).. controls (68.1, 61.7) and (63.2, 65.2) .. (57.7,
      65.3).. controls (53.3, 65.4) and (49.6, 63.5) .. (47.2,
      59.9).. controls (45.6, 57.6) and (40.3, 44.9) .. (37.7, 37.1) -- (35.9,
      31.8) -- (21.4, 31.8).. controls (6.8, 31.8) and (3.0, 31.5) .. (1.0,
      30.5).. controls (-0.2, 29.8) and (-0.2, 29.4) .. (1.4, 28.4).. controls
      (2.6, 27.7) and (3.4, 27.7) .. (18.8, 27.6).. controls (33.8, 27.5) and
      (35.0, 27.4) .. (34.8, 26.9).. controls (34.3, 25.4) and (32.0, 14.9)
      .. (31.2, 10.7).. controls (30.1, 4.7) and (30.1, 1.7) .. (31.0,
      0.7).. controls (31.4, 0.3) and (31.9, 0.0) .. (32.0, 0.0).. controls
      (32.5, 0.0) and (33.7, 1.4) .. (34.0, 2.2).. controls (34.3, 3.0) and
      (37.1, 16.2) .. (37.7, 19.2).. controls (38.0, 20.4) and (38.6, 22.7)
      .. (39.2, 24.4) -- (40.3, 27.5) -- (92.7, 27.5).. controls (127.5, 27.5)
      and (145.7, 27.7) .. (146.8, 27.9).. controls (149.8, 28.6) and (150.6,
      29.8) .. (148.6, 30.7).. controls (146.4, 31.5) and (141.3, 31.7)
      .. (120.0, 31.8) -- (97.7, 31.8) -- (100.1, 33.7).. controls (106.6,
      39.2) and (112.3, 45.9) .. (116.1, 53.0).. controls (117.3, 55.1) and
      (118.7, 57.2) .. (119.2, 57.5).. controls (119.6, 57.9) and (122.1,
      59.0) .. (124.6, 60.0).. controls (130.3, 62.2) and (132.7, 63.4)
      .. (136.5, 66.0).. controls (140.5, 68.7) and (141.9, 70.0) .. (141.7,
      71.0).. controls (141.7, 71.6) and (141.4, 71.8) .. (140.6,
      71.9).. controls (139.5, 72.0) and (139.2, 71.8) .. (130.6,
      66.5).. controls (126.7, 64.1) and (123.8, 63.0) .. (122.8,
      63.5).. controls (121.8, 64.0) and (121.8, 65.9) .. (122.8,
      69.9).. controls (123.8, 74.3) and (124.1, 80.4) .. (123.3,
      83.7).. controls (122.6, 87.1) and (120.6, 91.1) .. (119.2,
      92.1).. controls (117.7, 93.2) and (115.9, 93.3) .. (114.1, 92.5) --
      cycle(117.6, 88.6).. controls (119.1, 87.0) and (119.7, 84.3) .. (119.9,
      79.0).. controls (120.1, 71.5) and (118.8, 64.9) .. (117.1,
      64.6).. controls (115.7, 64.4) and (113.4, 68.0) .. (111.8,
      72.7).. controls (110.5, 76.7) and (111.0, 81.6) .. (113.2,
      86.1).. controls (114.8, 89.1) and (116.3, 90.0) .. (117.6, 88.6) --
      cycle(61.3, 60.2).. controls (64.8, 58.4) and (69.1, 53.7) .. (71.4,
      49.0).. controls (72.1, 47.8) and (73.4, 44.7) .. (74.4,
      42.2).. controls (76.7, 36.5) and (78.1, 33.6) .. (79.1, 32.6) -- (79.8,
      31.8) -- (60.7, 31.8).. controls (48.1, 31.8) and (41.6, 31.9) .. (41.6,
      32.1).. controls (41.6, 32.9) and (48.5, 51.1) .. (49.7,
      53.8).. controls (51.3, 56.9) and (53.9, 60.1) .. (55.5,
      60.9).. controls (57.0, 61.6) and (59.2, 61.3) .. (61.3, 60.2) -- cycle;
    \end{scope}
  \end{tikzpicture}
}

\def\SngRedAoneptwoAthreeConic{%
  \begin{tikzpicture}[y=1cm, x=1cm,
    yscale=0.9*\globalscale,xscale=0.8*\globalscale,%
    baseline={([yshift=-1mm]current bounding box.center)}]
    \begin{scope}[fill=black,cm={1,0,0,1,(0, 0)}]
      \path[fill=black] (80.4, 112.5).. controls (57.2, 110.3) and (29.7,
      94.5) .. (13.3, 74.2).. controls (6.7, 66.0) and (2.2, 57.0) .. (0.7,
      49.0).. controls (0.1, 45.8) and (0.1, 40.1) .. (0.8, 37.4).. controls
      (3.7, 25.9) and (14.4, 19.3) .. (29.8, 19.3).. controls (35.3, 19.3) and
      (38.3, 19.6) .. (44.7, 21.0) -- (48.9, 21.9) -- (51.4, 20.1).. controls
      (69.3, 6.6) and (91.7, -1.2) .. (107.7, 0.3).. controls (120.8, 1.6) and
      (129.4, 8.0) .. (131.8, 18.2).. controls (132.5, 21.4) and (132.5, 27.6)
      .. (131.8, 31.2).. controls (129.6, 41.6) and (123.2, 52.4) .. (113.2,
      62.7) -- (109.0, 66.9) -- (110.8, 70.4).. controls (114.2, 77.1) and
      (116.0, 84.5) .. (115.6, 90.3).. controls (114.8, 102.8) and (105.6,
      111.0) .. (90.6, 112.5).. controls (86.6, 112.9) and (84.8, 112.9)
      .. (80.4, 112.5) -- cycle(95.8, 106.9).. controls (105.2, 104.4) and
      (110.4, 98.0) .. (110.4, 88.9).. controls (110.4, 84.2) and (108.8,
      78.2) .. (106.2, 73.0) -- (105.0, 70.6) -- (100.1, 74.2).. controls
      (83.4, 86.7) and (64.1, 94.2) .. (48.3, 94.2).. controls (43.7, 94.2)
      and (43.7, 94.1) .. (48.9, 97.0).. controls (58.5, 102.2) and (70.6,
      106.4) .. (79.8, 107.6).. controls (84.1, 108.1) and (92.5, 107.8)
      .. (95.8, 106.9) -- cycle(51.6, 88.9).. controls (66.3, 87.7) and (83.2,
      80.7) .. (97.5, 69.8).. controls (102.1, 66.3) and (102.0, 66.5)
      .. (102.0, 66.0).. controls (102.0, 65.1) and (96.1, 58.0) .. (91.7,
      53.6).. controls (82.9, 44.7) and (72.3, 37.3) .. (61.2,
      32.1).. controls (57.0, 30.1) and (51.1, 27.9) .. (50.2,
      27.9).. controls (49.4, 27.9) and (44.0, 32.9) .. (39.8,
      37.4).. controls (32.7, 44.8) and (27.2, 54.1) .. (24.9,
      61.9).. controls (23.8, 65.7) and (23.4, 70.8) .. (24.0,
      73.9).. controls (25.5, 82.1) and (33.1, 87.9) .. (43.5,
      88.9).. controls (45.3, 89.0) and (46.9, 89.2) .. (47.1,
      89.2).. controls (47.3, 89.2) and (49.3, 89.1) .. (51.6, 88.9) --
      cycle(19.9, 62.1).. controls (22.4, 51.7) and (29.3, 40.5) .. (39.8,
      30.1) -- (44.1, 25.8) -- (41.2, 25.3).. controls (36.0, 24.3) and (31.7,
      23.9) .. (27.3, 24.1).. controls (20.0, 24.4) and (14.7, 26.3) .. (10.8,
      30.0).. controls (7.0, 33.5) and (5.5, 37.4) .. (5.5, 43.2).. controls
      (5.6, 51.5) and (9.6, 61.0) .. (17.1, 70.6) -- (18.9, 72.8) -- (19.1,
      68.6).. controls (19.2, 66.3) and (19.5, 63.4) .. (19.9, 62.1) --
      cycle(109.9, 58.8).. controls (124.5, 43.9) and (130.8, 27.4) .. (126.2,
      16.6).. controls (124.1, 11.7) and (119.1, 8.0) .. (112.3,
      6.3).. controls (109.7, 5.6) and (108.2, 5.5) .. (103.2, 5.5).. controls
      (98.1, 5.5) and (96.5, 5.7) .. (92.6, 6.5).. controls (81.3, 8.8) and
      (67.9, 14.7) .. (57.7, 21.8).. controls (56.2, 22.9) and (55.1, 23.8)
      .. (55.3, 24.0).. controls (55.4, 24.2) and (56.8, 24.8) .. (58.3,
      25.4).. controls (75.7, 32.2) and (93.9, 45.9) .. (104.0,
      59.7).. controls (105.1, 61.2) and (106.1, 62.4) .. (106.2,
      62.4).. controls (106.3, 62.4) and (108.0, 60.8) .. (109.9, 58.8) --
      cycle;
    \end{scope}
  \end{tikzpicture}
}

\def\SngRedAoneptwoDfour{%
  \begin{tikzpicture}[y=1cm, x=1cm,
    yscale=1.5*\globalscale,xscale=1.4*\globalscale,%
    baseline={([yshift=-1mm]current bounding box.center)}]
    \begin{scope}[fill=black,cm={1,0,0,1,(0, 0)}]
      \path[fill=black] (40.3, 68.8).. controls (34.3, 68.1) and (27.8, 65.7)
      .. (22.0, 62.3).. controls (20.2, 61.2) and (18.7, 60.3) .. (18.6,
      60.3).. controls (18.5, 60.3) and (18.2, 61.3) .. (17.9,
      62.6).. controls (17.2, 64.9) and (16.2, 66.5) .. (15.4,
      66.2).. controls (14.8, 66.0) and (14.7, 63.6) .. (15.3,
      61.1).. controls (15.9, 58.5) and (16.4, 58.7) .. (10.7,
      59.3).. controls (6.7, 59.7) and (4.9, 59.6) .. (4.9, 58.9).. controls
      (4.9, 58.5) and (7.2, 57.9) .. (10.2, 57.7).. controls (15.0, 57.3) and
      (14.9, 57.5) .. (12.0, 54.5).. controls (2.9, 44.9) and (-0.2, 34.8)
      .. (3.5, 26.9).. controls (6.2, 21.1) and (12.9, 17.7) .. (21.7, 17.5)
      -- (25.6, 17.4) -- (27.2, 10.5).. controls (28.0, 6.8) and (29.0, 3.2)
      .. (29.2, 2.5).. controls (30.1, 0.5) and (31.8, 0.4) .. (31.8,
      2.4).. controls (31.8, 3.0) and (31.0, 6.6) .. (30.1, 10.5).. controls
      (29.2, 14.4) and (28.5, 17.7) .. (28.6, 17.8).. controls (28.7, 17.9)
      and (30.1, 18.3) .. (31.6, 18.7).. controls (44.8, 22.0) and (57.2,
      31.8) .. (62.3, 42.9).. controls (63.6, 45.7) and (64.6, 49.2) .. (64.6,
      51.3).. controls (64.6, 52.0) and (64.7, 52.6) .. (64.8,
      52.6).. controls (65.0, 52.6) and (68.5, 52.2) .. (72.7,
      51.9).. controls (79.8, 51.2) and (82.2, 51.2) .. (82.2,
      51.9).. controls (82.2, 52.5) and (79.7, 52.9) .. (72.4,
      53.6).. controls (68.2, 54.0) and (64.8, 54.4) .. (64.7,
      54.4).. controls (64.7, 54.5) and (64.4, 55.6) .. (64.0,
      56.9).. controls (62.9, 61.2) and (60.1, 64.5) .. (55.9,
      66.6).. controls (52.1, 68.5) and (45.5, 69.4) .. (40.3, 68.8) --
      cycle(49.4, 66.5).. controls (53.9, 65.6) and (57.3, 63.2) .. (59.2,
      59.7).. controls (59.8, 58.6) and (60.4, 57.0) .. (60.5, 56.1) -- (60.8,
      54.6) -- (58.9, 54.8).. controls (56.3, 55.1) and (22.8, 58.2) .. (22.1,
      58.2).. controls (21.0, 58.2) and (21.5, 58.8) .. (23.9,
      60.2).. controls (27.3, 62.3) and (30.8, 64.0) .. (33.9,
      65.0).. controls (36.2, 65.8) and (37.5, 66.1) .. (41.1,
      66.8).. controls (42.8, 67.1) and (47.3, 67.0) .. (49.4, 66.5) --
      cycle(41.9, 54.7).. controls (52.3, 53.7) and (60.9, 52.8) .. (61.0,
      52.7).. controls (61.1, 52.7) and (60.9, 51.2) .. (60.5,
      49.5).. controls (58.5, 39.3) and (49.4, 28.9) .. (37.4,
      23.3).. controls (34.4, 21.9) and (31.2, 20.8) .. (29.0,
      20.3).. controls (28.2, 20.1) and (27.9, 20.2) .. (27.7,
      20.7).. controls (27.4, 21.9) and (19.4, 56.1) .. (19.4,
      56.5).. controls (19.4, 56.7) and (20.1, 56.8) .. (21.3,
      56.6).. controls (22.3, 56.5) and (31.6, 55.6) .. (41.9, 54.7) --
      cycle(21.0, 37.0).. controls (23.2, 27.6) and (24.9, 19.8) .. (24.8,
      19.7).. controls (24.5, 19.2) and (19.3, 19.4) .. (17.0,
      19.9).. controls (5.2, 22.7) and (2.2, 34.8) .. (10.3, 47.1).. controls
      (12.1, 49.9) and (16.7, 55.0) .. (16.8, 54.5).. controls (16.9, 54.3)
      and (18.7, 46.5) .. (21.0, 37.0) -- cycle;
    \end{scope}
  \end{tikzpicture}
}

\def\SngRedAoneAtwoAthree{%
  \begin{tikzpicture}[y=1cm, x=1cm,
    yscale=0.9*\globalscale,xscale=0.9*\globalscale,%
    baseline={([yshift=-1mm]current bounding box.center)}]
    \begin{scope}[fill=black,cm={1,0,0,1,(0, 0)}]
      \path[fill=black] (109.0, 96.5).. controls (109.0, 96.3) and (108.9,
      95.2) .. (108.7, 94.1).. controls (108.5, 92.9) and (108.1, 88.7)
      .. (107.8, 84.7).. controls (107.2, 76.6) and (106.5, 71.4) .. (105.5,
      67.8).. controls (104.6, 64.5) and (100.6, 55.7) .. (98.3,
      51.8).. controls (95.7, 47.4) and (93.9, 45.3) .. (88.2,
      40.2).. controls (82.0, 34.6) and (78.6, 32.0) .. (74.5, 30.0) -- (71.2,
      28.4) -- (69.2, 29.5).. controls (66.3, 31.0) and (64.3, 33.8) .. (60.9,
      40.7).. controls (59.2, 44.1) and (57.6, 47.4) .. (57.3,
      48.3).. controls (55.6, 52.6) and (51.2, 57.6) .. (46.6,
      60.7).. controls (44.6, 61.9) and (44.1, 62.1) .. (42.9,
      61.9).. controls (38.9, 61.3) and (34.6, 57.6) .. (29.6,
      50.6).. controls (24.3, 43.3) and (22.4, 39.2) .. (20.3, 31.6) -- (19.1,
      27.3) -- (11.6, 27.1).. controls (7.4, 27.0) and (3.6, 26.8) .. (3.0,
      26.6).. controls (-0.1, 25.9) and (-0.9, 24.5) .. (1.1, 23.4).. controls
      (2.5, 22.7) and (7.3, 22.3) .. (13.3, 22.2).. controls (15.9, 22.2) and
      (18.0, 22.1) .. (18.0, 21.9).. controls (18.0, 21.8) and (17.6, 19.1)
      .. (17.1, 16.0).. controls (15.4, 5.3) and (15.6, 0.2) .. (17.6,
      0.2).. controls (19.5, 0.2) and (21.1, 5.0) .. (22.0, 13.9).. controls
      (22.6, 18.8) and (23.0, 20.8) .. (23.6, 21.9).. controls (23.9, 22.3)
      and (31.4, 22.4) .. (82.4, 22.3).. controls (116.8, 22.3) and (142.3,
      22.4) .. (144.4, 22.6).. controls (148.2, 23.0) and (149.9, 23.6)
      .. (149.9, 24.7).. controls (149.9, 25.4) and (149.1, 26.0) .. (147.1,
      26.6).. controls (146.3, 26.8) and (134.1, 27.0) .. (114.2, 27.1) --
      (82.6, 27.3) -- (84.8, 29.0).. controls (87.2, 30.8) and (92.0, 35.8)
      .. (96.1, 40.7).. controls (101.3, 47.0) and (106.1, 54.7) .. (107.4,
      58.8).. controls (109.3, 65.1) and (109.7, 65.3) .. (112.8,
      62.4).. controls (118.9, 56.7) and (127.2, 51.4) .. (131.8,
      50.4).. controls (133.4, 50.1) and (133.8, 50.1) .. (134.5,
      50.6).. controls (136.7, 52.0) and (135.5, 53.1) .. (128.4,
      56.6).. controls (125.8, 57.9) and (122.7, 59.7) .. (121.5,
      60.5).. controls (118.5, 62.7) and (114.7, 66.9) .. (113.1,
      69.8).. controls (111.4, 73.0) and (110.9, 75.7) .. (110.6,
      82.2).. controls (110.3, 89.8) and (109.7, 96.1) .. (109.3,
      96.5).. controls (109.1, 96.7) and (109.0, 96.7) .. (109.0, 96.5) --
      cycle(45.7, 55.4).. controls (47.4, 54.6) and (49.8, 52.1) .. (50.9,
      49.8).. controls (51.4, 48.9) and (52.7, 45.9) .. (53.8,
      43.2).. controls (56.7, 36.1) and (58.7, 32.4) .. (61.3, 29.6) -- (63.6,
      27.2) -- (44.1, 27.2).. controls (27.7, 27.2) and (24.7, 27.2) .. (24.7,
      27.7).. controls (24.7, 28.9) and (27.1, 35.2) .. (29.0,
      39.1).. controls (30.2, 41.4) and (32.8, 45.9) .. (34.9,
      48.9).. controls (38.1, 53.7) and (38.8, 54.6) .. (40.2,
      55.3).. controls (42.1, 56.3) and (43.8, 56.3) .. (45.7, 55.4) -- cycle;
    \end{scope}
  \end{tikzpicture}
}

\def\SngAoneAthreeptwo{%
  \begin{tikzpicture}[y=1cm, x=1cm,
    yscale=\globalscale,xscale=\globalscale,%
    baseline={([yshift=-1mm]current bounding box.center)}]
    \begin{scope}[fill=black,cm={1,0,0,1,(0, 0)}]
      \path[fill=black] (157.44, 101.88).. controls (156.95, 100.93) and
      (156.81, 99.62) .. (156.7, 94.47).. controls (156.63, 91.02) and
      (156.49, 88.19) .. (156.39, 88.19).. controls (156.28, 88.19) and
      (143.93, 85.76) .. (128.98, 82.76).. controls (114.02, 79.76) and
      (89.32, 74.82) .. (74.08, 71.83).. controls (58.84, 68.79) and (45.58,
      65.97) .. (44.56, 65.58).. controls (40.64, 63.99) and (40.99, 62.58)
      .. (45.3, 62.69).. controls (46.85, 62.72) and (53.94, 63.96) .. (63.78,
      65.93).. controls (72.57, 67.7) and (79.9, 69.14) .. (80.12,
      69.14).. controls (80.33, 69.14) and (79.62, 68.19) .. (78.56,
      67.03).. controls (75.85, 64.1) and (75.0, 61.84) .. (75.0,
      57.5).. controls (75.0, 54.61) and (75.14, 53.8) .. (76.09,
      51.47).. controls (79.41, 43.07) and (87.84, 35.28) .. (100.01,
      29.28).. controls (112.18, 23.28) and (122.45, 20.92) .. (134.58,
      21.27).. controls (143.51, 21.55) and (148.63, 23.0) .. (153.35, 26.63)
      -- (155.89, 28.58) -- (156.0, 16.79).. controls (156.1, 5.86) and
      (156.14, 4.97) .. (156.77, 4.27).. controls (157.66, 3.32) and (158.61,
      3.81) .. (159.07, 5.54).. controls (159.28, 6.31) and (159.6, 21.98)
      .. (159.84, 45.83).. controls (160.06, 67.27) and (160.34, 84.95)
      .. (160.48, 85.05).. controls (160.58, 85.2) and (162.88, 85.73)
      .. (165.56, 86.25).. controls (172.58, 87.59) and (175.68, 88.76)
      .. (175.68, 89.99).. controls (175.68, 91.3) and (172.37, 91.37)
      .. (165.98, 90.13).. controls (163.48, 89.64) and (161.22, 89.25)
      .. (161.01, 89.25).. controls (160.69, 89.25) and (160.55, 90.88)
      .. (160.44, 95.0).. controls (160.37, 99.13) and (160.16, 101.0)
      .. (159.81, 101.78).. controls (159.07, 103.26) and (158.22, 103.29)
      .. (157.44, 101.88) -- cycle(156.53, 78.95).. controls (156.39, 76.02)
      and (156.28, 68.02) .. (156.28, 61.14).. controls (156.28, 54.29) and
      (156.14, 48.68) .. (155.96, 48.68).. controls (155.82, 48.68) and
      (155.26, 49.32) .. (154.73, 50.09).. controls (152.22, 53.9) and
      (145.91, 59.55) .. (140.19, 63.08).. controls (132.04, 68.12) and
      (121.21, 72.18) .. (112.15, 73.59).. controls (109.93, 73.94) and
      (107.99, 74.3) .. (107.84, 74.44).. controls (107.53, 74.72) and
      (106.54, 74.51) .. (134.23, 80.05).. controls (145.7, 82.34) and
      (155.43, 84.24) .. (155.89, 84.28) -- (156.74, 84.31) -- (156.53, 78.95)
      -- cycle(106.01, 71.12).. controls (120.09, 70.27) and (137.23, 63.01)
      .. (147.11, 53.73).. controls (150.74, 50.31) and (152.93, 47.41)
      .. (154.62, 43.85).. controls (155.65, 41.7) and (155.75, 41.13)
      .. (155.75, 38.1).. controls (155.75, 35.28) and (155.61, 34.5)
      .. (154.94, 33.2).. controls (153.32, 30.2) and (149.65, 27.45)
      .. (145.17, 25.93).. controls (138.39, 23.64) and (126.29, 23.85)
      .. (116.52, 26.46).. controls (97.08, 31.64) and (81.84, 42.86)
      .. (78.49, 54.5).. controls (76.38, 61.67) and (81.63, 68.16) .. (91.37,
      70.38).. controls (93.98, 70.98) and (100.22, 71.58) .. (101.95,
      71.4).. controls (102.45, 71.33) and (104.28, 71.23) .. (106.01, 71.12)
      -- cycle;
    \end{scope}
  \end{tikzpicture}
}

\def\SngAonepfourCubic{%
  \begin{tikzpicture}[y=1cm, x=1cm,
    yscale=1.2*\globalscale,xscale=\globalscale,%
    baseline={([yshift=-1mm]current bounding box.center)}]
    \begin{scope}[fill=black,cm={1,0,0,1,(0, 0)}]
      \path[fill=black] (116.21, 91.76).. controls (112.43, 89.46) and
      (109.68, 82.73) .. (110.24, 77.08).. controls (110.74, 71.79) and
      (111.51, 69.64) .. (114.83, 64.03).. controls (116.66, 61.0) and
      (116.84, 60.54) .. (116.7, 59.13).. controls (116.49, 57.33) and
      (114.58, 53.98) .. (110.49, 48.26) -- (107.77, 44.45) -- (94.19, 44.45)
      -- (80.61, 44.45) -- (79.55, 46.99).. controls (76.76, 53.66) and
      (70.45, 61.24) .. (65.79, 63.54).. controls (63.54, 64.63) and (63.08,
      64.73) .. (60.15, 64.73).. controls (57.26, 64.73) and (56.73, 64.63)
      .. (54.61, 63.57).. controls (50.69, 61.67) and (48.97, 59.06)
      .. (45.12, 49.32) -- (43.22, 44.49) -- (24.87, 44.34).. controls (6.03,
      44.24) and (4.66, 44.13) .. (3.32, 42.79).. controls (2.79, 42.23) and
      (2.79, 42.16) .. (3.49, 41.59).. controls (4.97, 40.36) and (7.44,
      40.22) .. (24.69, 40.22) -- (41.59, 40.22) -- (40.01, 35.56).. controls
      (38.42, 30.8) and (35.17, 17.39) .. (33.83, 9.95).. controls (32.95,
      5.08) and (32.81, 1.38) .. (33.48, 0.56).. controls (34.18, -0.28) and
      (35.52, -0.14) .. (36.16, 0.85).. controls (36.48, 1.31) and (37.57,
      5.64) .. (38.63, 10.44).. controls (39.69, 15.28) and (40.78, 20.11)
      .. (41.06, 21.17).. controls (41.91, 24.48) and (45.05, 34.25)
      .. (46.21, 37.32) -- (47.34, 40.22) -- (62.48, 40.22) -- (77.65, 40.22)
      -- (78.18, 38.91).. controls (79.9, 34.68) and (81.32, 32.35) .. (83.22,
      30.66).. controls (86.82, 27.45) and (90.1, 26.53) .. (93.91,
      27.66).. controls (96.77, 28.47) and (102.02, 32.35) .. (108.13, 38.17)
      -- (110.28, 40.18) -- (130.0, 40.29).. controls (147.71, 40.39) and
      (149.9, 40.46) .. (151.09, 40.99).. controls (152.68, 41.73) and
      (152.86, 42.76) .. (151.45, 43.36).. controls (149.3, 44.2) and (144.82,
      44.45) .. (129.58, 44.45) -- (113.63, 44.45) -- (115.85,
      47.73).. controls (117.05, 49.53) and (118.75, 52.32) .. (119.59,
      53.9).. controls (121.32, 57.26) and (121.39, 57.29) .. (127.92,
      59.83).. controls (130.14, 60.68) and (132.93, 61.88) .. (134.13,
      62.48).. controls (138.47, 64.66) and (144.64, 69.39) .. (144.64,
      70.56).. controls (144.64, 71.19) and (143.55, 71.65) .. (142.56,
      71.44).. controls (141.96, 71.3) and (139.45, 69.92) .. (136.98,
      68.37).. controls (131.13, 64.66) and (127.95, 63.04) .. (126.4,
      63.01).. controls (124.35, 62.94) and (124.25, 63.46) .. (125.34,
      69.22).. controls (126.65, 76.31) and (126.75, 77.75) .. (126.29,
      81.1).. controls (125.59, 86.22) and (123.51, 90.81) .. (121.32,
      91.93).. controls (119.94, 92.64) and (117.51, 92.57) .. (116.21, 91.76)
      -- cycle(119.8, 88.37).. controls (122.52, 86.5) and (123.51, 76.94)
      .. (121.85, 68.79).. controls (121.18, 65.62) and (120.51, 64.21)
      .. (119.63, 64.21).. controls (117.44, 64.21) and (113.59, 72.39)
      .. (113.59, 76.98).. controls (113.59, 80.61) and (115.82, 86.68)
      .. (117.69, 88.16).. controls (118.82, 89.04) and (118.89, 89.04)
      .. (119.8, 88.37) -- cycle(63.96, 59.73).. controls (66.18, 58.63) and
      (69.74, 55.28) .. (71.83, 52.32).. controls (73.13, 50.52) and (75.85,
      45.33) .. (75.85, 44.66).. controls (75.85, 44.56) and (69.81, 44.45)
      .. (62.44, 44.45).. controls (55.07, 44.45) and (49.04, 44.52)
      .. (49.04, 44.63).. controls (49.04, 45.26) and (52.11, 53.06)
      .. (52.95, 54.57).. controls (54.29, 56.94) and (57.01, 59.94)
      .. (58.38, 60.54).. controls (59.76, 61.17) and (61.67, 60.89)
      .. (63.96, 59.73) -- cycle(103.51, 39.44).. controls (102.48, 38.28) and
      (96.1, 33.66) .. (94.37, 32.81).. controls (91.19, 31.29) and (88.76,
      31.43) .. (86.54, 33.3).. controls (85.27, 34.43) and (83.33, 37.54)
      .. (82.73, 39.44) -- (82.51, 40.22) -- (93.35, 40.22) -- (104.18, 40.22)
      -- (103.51, 39.44) -- cycle;
    \end{scope}
  \end{tikzpicture}
}

\def\SngAonepfourConic{%
  \begin{tikzpicture}[y=1cm, x=1cm,
    yscale=1.2*\globalscale,xscale=\globalscale,%
    baseline={([yshift=-1mm]current bounding box.center)}]
    \begin{scope}[fill=black,cm={1,0,0,1,(0, 0)}]
      \path[fill=black] (75.04, 114.23).. controls (65.62, 113.17) and (56.55,
      103.33) .. (51.89, 89.15).. controls (51.22, 87.14) and (50.94, 86.75)
      .. (50.17, 86.57).. controls (34.85, 83.19) and (23.04, 76.59)
      .. (19.16, 69.32).. controls (17.0, 65.23) and (16.93, 61.28) .. (18.94,
      57.19).. controls (20.14, 54.72) and (24.02, 50.66) .. (27.16,
      48.54).. controls (32.1, 45.23) and (39.79, 42.09) .. (47.7, 40.18) --
      (51.15, 39.37) -- (52.07, 36.55).. controls (52.6, 34.96) and (53.76,
      32.1) .. (54.65, 30.16).. controls (63.75, 10.76) and (79.8, 5.79)
      .. (92.43, 18.42).. controls (96.66, 22.65) and (100.29, 28.86)
      .. (102.8, 36.23).. controls (104.1, 40.15) and (104.14, 40.15)
      .. (105.41, 40.39).. controls (107.88, 40.85) and (116.95, 43.99)
      .. (120.05, 45.47).. controls (127.11, 48.86) and (132.26, 53.27)
      .. (134.51, 57.86).. controls (135.54, 59.94) and (135.64, 60.43)
      .. (135.64, 63.5).. controls (135.64, 66.68) and (135.57, 67.06)
      .. (134.44, 69.25).. controls (132.4, 73.24) and (127.25, 77.82)
      .. (121.5, 80.72).. controls (118.53, 82.23) and (112.96, 84.28)
      .. (108.83, 85.37).. controls (106.5, 86.01) and (104.39, 86.61)
      .. (104.1, 86.75).. controls (103.82, 86.89) and (103.15, 88.27)
      .. (102.62, 89.85).. controls (100.26, 96.91) and (96.41, 103.4)
      .. (91.93, 107.77).. controls (86.82, 112.85) and (81.35, 114.94)
      .. (75.04, 114.23) -- cycle(83.36, 109.11).. controls (89.15, 106.4) and
      (94.58, 99.55) .. (98.04, 90.56).. controls (99.24, 87.38) and (99.66,
      87.63) .. (93.84, 88.37).. controls (84.31, 89.64) and (70.1, 89.64)
      .. (60.71, 88.37).. controls (58.49, 88.09) and (56.48, 87.84)
      .. (56.23, 87.84).. controls (55.21, 87.84) and (58.6, 95.57) .. (61.52,
      100.01).. controls (67.7, 109.26) and (75.81, 112.64) .. (83.36, 109.11)
      -- cycle(88.02, 85.37).. controls (93.8, 84.84) and (99.59, 83.93)
      .. (100.08, 83.47).. controls (100.97, 82.62) and (102.45, 73.24)
      .. (102.83, 65.79).. controls (103.15, 59.94) and (102.2, 49.42)
      .. (100.86, 44.38).. controls (100.54, 43.18) and (100.4, 43.07)
      .. (98.67, 42.72).. controls (90.73, 41.06) and (78.35, 40.29)
      .. (69.43, 40.85).. controls (62.79, 41.28) and (55.03, 42.19)
      .. (54.61, 42.62).. controls (54.26, 42.97) and (53.02, 49.28)
      .. (52.53, 53.27).. controls (51.89, 58.21) and (52.0, 69.18) .. (52.71,
      74.26).. controls (53.38, 78.85) and (54.26, 83.29) .. (54.65,
      83.64).. controls (54.89, 83.93) and (61.63, 84.91) .. (66.15,
      85.34).. controls (70.73, 85.76) and (83.47, 85.8) .. (88.02, 85.37) --
      cycle(49.49, 80.19).. controls (47.52, 70.63) and (47.45, 56.9)
      .. (49.39, 46.92).. controls (49.74, 45.09) and (49.99, 43.53)
      .. (49.92, 43.46).. controls (49.88, 43.43) and (47.94, 43.92)
      .. (45.65, 44.59).. controls (35.98, 47.38) and (29.32, 50.91)
      .. (25.33, 55.35).. controls (22.9, 58.07) and (22.05, 60.01) .. (22.05,
      62.97).. controls (22.05, 65.12) and (22.23, 65.72) .. (23.25,
      67.66).. controls (25.54, 71.9) and (30.9, 75.95) .. (38.28,
      79.02).. controls (41.95, 80.54) and (47.98, 82.48) .. (49.25,
      82.51).. controls (49.99, 82.55) and (49.99, 82.51) .. (49.49, 80.19) --
      cycle(111.16, 80.79).. controls (118.0, 78.53) and (123.37, 75.6)
      .. (126.96, 72.18).. controls (129.36, 69.85) and (130.42, 68.23)
      .. (131.09, 65.65).. controls (132.61, 59.62) and (127.6, 53.34)
      .. (117.37, 48.51).. controls (114.09, 46.92) and (105.52, 43.92)
      .. (105.23, 44.24).. controls (105.16, 44.31) and (105.3, 45.61)
      .. (105.62, 47.13).. controls (106.68, 52.85) and (107.0, 57.79)
      .. (106.79, 65.37).. controls (106.61, 72.11) and (105.83, 79.02)
      .. (104.92, 81.99).. controls (104.74, 82.55) and (104.85, 82.59)
      .. (105.9, 82.37).. controls (106.54, 82.23) and (108.9, 81.53)
      .. (111.16, 80.79) -- cycle(98.64, 36.76).. controls (93.84, 22.72) and
      (84.24, 13.93) .. (75.42, 15.52).. controls (68.05, 16.86) and (61.31,
      24.06) .. (56.8, 35.42).. controls (56.2, 36.94) and (55.77, 38.28)
      .. (55.88, 38.35).. controls (55.95, 38.42) and (58.07, 38.24)
      .. (60.57, 37.92).. controls (72.11, 36.48) and (87.77, 36.97) .. (99.2,
      39.12).. controls (99.34, 39.12) and (99.06, 38.1) .. (98.64, 36.76) --
      cycle;
    \end{scope}
  \end{tikzpicture}
}

\def\SngAonepthreeAtwo{%
  \begin{tikzpicture}[y=1cm, x=1cm,
    yscale=1.2*\globalscale,xscale=\globalscale,%
    baseline={([yshift=-1mm]current bounding box.center)}]
    \begin{scope}[fill=black,cm={1,0,0,1,(0, 0)}]
      \path[fill=black] (110.7, 100.89).. controls (110.31, 99.48) and
      (109.78, 94.44) .. (109.36, 88.55).. controls (108.55, 76.73) and
      (107.6, 72.32) .. (104.28, 64.98).. controls (99.87, 55.1) and (97.75,
      51.82) .. (92.53, 46.67) -- (88.51, 42.69) -- (81.35, 42.69).. controls
      (77.43, 42.69) and (70.59, 42.58) .. (66.22, 42.44) -- (58.21, 42.23) --
      (56.09, 46.78).. controls (52.67, 54.01) and (52.14, 54.96) .. (50.2,
      57.33).. controls (48.08, 59.87) and (45.05, 62.44) .. (41.91,
      64.31).. controls (39.37, 65.83) and (38.77, 65.79) .. (34.93,
      63.92).. controls (32.91, 62.97) and (31.64, 62.02) .. (29.63,
      59.97).. controls (26.46, 56.76) and (20.67, 48.86) .. (18.7, 45.09) --
      (17.29, 42.4) -- (10.12, 42.26).. controls (4.45, 42.16) and (2.72,
      42.02) .. (1.66, 41.56).. controls (0.07, 40.85) and (-0.07, 39.83)
      .. (1.31, 39.19).. controls (2.65, 38.59) and (7.23, 38.1) .. (11.75,
      38.1).. controls (15.56, 38.1) and (15.56, 38.1) .. (15.31,
      37.32).. controls (12.38, 28.19) and (9.53, 12.38) .. (9.53,
      5.29).. controls (9.53, 0.6) and (10.44, -0.95) .. (12.1,
      0.85).. controls (13.58, 2.43) and (14.99, 8.47) .. (15.84,
      16.93).. controls (16.09, 19.05) and (16.51, 21.2) .. (16.9,
      22.05).. controls (17.29, 22.82) and (17.85, 25.05) .. (18.17,
      26.99).. controls (18.7, 30.3) and (19.72, 33.8) .. (20.96, 36.72) --
      (21.48, 37.99) -- (29.49, 38.21).. controls (33.9, 38.35) and (41.17,
      38.45) .. (45.65, 38.45) -- (53.83, 38.45) -- (55.17, 36.09).. controls
      (56.87, 33.13) and (62.19, 27.76) .. (64.88, 26.32).. controls (70.98,
      23.04) and (79.55, 26.14) .. (88.69, 34.96) -- (92.32, 38.42) --
      (119.45, 38.52).. controls (143.23, 38.63) and (146.76, 38.7)
      .. (147.99, 39.19).. controls (149.05, 39.62) and (149.4, 39.93)
      .. (149.4, 40.57).. controls (149.4, 41.2) and (149.05, 41.52)
      .. (147.99, 41.95).. controls (146.76, 42.44) and (143.37, 42.51)
      .. (121.29, 42.62) -- (95.96, 42.72) -- (98.57, 45.97).. controls
      (103.89, 52.6) and (108.06, 59.65) .. (109.04, 63.68).. controls
      (109.29, 64.66) and (109.75, 66.04) .. (110.14, 66.78).. controls
      (111.2, 68.9) and (112.22, 68.65) .. (115.96, 65.23).. controls (121.29,
      60.33) and (130.42, 54.79) .. (135.18, 53.55).. controls (137.37, 52.99)
      and (138.78, 53.45) .. (138.92, 54.75).. controls (139.06, 55.95) and
      (138.15, 56.59) .. (132.12, 59.44).. controls (125.73, 62.48) and
      (122.8, 64.42) .. (119.42, 67.87).. controls (116.31, 71.05) and
      (114.65, 73.59) .. (113.35, 77.15).. controls (112.61, 79.23) and
      (112.43, 80.65) .. (111.97, 88.9).. controls (111.37, 99.31) and
      (111.05, 102.31) .. (110.7, 100.89) -- cycle(42.51, 58.42).. controls
      (44.24, 57.4) and (45.76, 55.74) .. (47.1, 53.48).. controls (47.98,
      52.0) and (51.12, 44.87) .. (51.65, 43.11) -- (51.89, 42.33) -- (37.96,
      42.33).. controls (27.59, 42.33) and (23.99, 42.44) .. (23.99,
      42.76).. controls (23.99, 43.71) and (32.6, 56.59) .. (34.22,
      58.03).. controls (36.16, 59.8) and (39.86, 59.94) .. (42.51, 58.42) --
      cycle(83.26, 38.31).. controls (83.26, 38.03) and (78.11, 34.33)
      .. (76.02, 33.09).. controls (74.82, 32.42) and (72.81, 31.43)
      .. (71.51, 30.9) -- (69.14, 29.95) -- (67.66, 30.73).. controls (65.3,
      31.93) and (63.61, 33.41) .. (61.91, 35.81).. controls (61.03, 37.01)
      and (60.33, 38.1) .. (60.33, 38.24).. controls (60.33, 38.35) and
      (65.48, 38.45) .. (71.79, 38.45).. controls (78.11, 38.45) and (83.26,
      38.38) .. (83.26, 38.31) -- cycle;
    \end{scope}
  \end{tikzpicture}
}

\def\SngRedAonepthreeAthree{%
  \begin{tikzpicture}[y=1cm, x=1cm,
    yscale=1.2*\globalscale,xscale=\globalscale,%
    baseline={([yshift=-1mm]current bounding box.center)}]
    \begin{scope}[fill=black,cm={1,0,0,1,(0, 0)}]
      \path[fill=black] (116.8, 118.9).. controls (115.7, 118.3) and (115.3,
      115.9) .. (115.1, 110.1) -- (114.9, 104.2) -- (100.9, 96.4) -- (86.9,
      88.7) -- (83.6, 90.8).. controls (75.9, 95.6) and (65.7, 99.9) .. (57.7,
      101.6).. controls (53.1, 102.6) and (44.9, 102.9) .. (40.5,
      102.3).. controls (26.9, 100.5) and (17.6, 92.3) .. (15.0,
      79.9).. controls (14.2, 76.2) and (14.2, 69.6) .. (15.0,
      65.4).. controls (15.6, 61.9) and (17.2, 56.8) .. (18.6,
      53.8).. controls (19.1, 52.7) and (19.4, 51.6) .. (19.3,
      51.4).. controls (19.3, 51.3) and (16.0, 49.3) .. (12.1,
      47.2).. controls (4.2, 42.8) and (1.4, 41.0) .. (0.5, 39.4).. controls
      (-0.2, 38.2) and (0.1, 37.7) .. (1.4, 37.7).. controls (3.2, 37.7) and
      (6.6, 39.3) .. (14.3, 43.5).. controls (18.3, 45.8) and (21.7, 47.6)
      .. (21.7, 47.6).. controls (21.9, 46.9) and (26.2, 40.6) .. (27.5,
      39.0).. controls (29.8, 35.8) and (37.1, 28.5) .. (40.7,
      25.6).. controls (58.7, 11.1) and (80.7, 4.8) .. (96.9, 9.4).. controls
      (104.3, 11.5) and (109.4, 15.1) .. (113.7, 21.5).. controls (114.3,
      22.4) and (114.3, 21.9) .. (114.2, 12.9).. controls (114.1, 2.9) and
      (114.3, 1.1) .. (115.7, 0.4).. controls (116.5, -0.1) and (117.3, 0.8)
      .. (117.9, 2.5).. controls (118.1, 3.2) and (118.3, 17.0) .. (118.5,
      37.9).. controls (118.7, 56.8) and (119.0, 78.8) .. (119.1, 87.0) --
      (119.3, 101.8) -- (123.4, 104.1).. controls (127.6, 106.4) and (131.4,
      109.6) .. (131.2, 110.5).. controls (131.0, 110.8) and (130.4, 110.9)
      .. (129.2, 110.9).. controls (127.7, 110.8) and (126.5, 110.5)
      .. (123.6, 109.0).. controls (121.5, 107.9) and (119.7, 107.0)
      .. (119.6, 107.0).. controls (119.4, 106.9) and (119.2, 109.1)
      .. (119.2, 111.9).. controls (119.0, 116.8) and (118.6, 118.6)
      .. (117.7, 119.0).. controls (117.4, 119.1) and (117.1, 119.1)
      .. (116.8, 118.9) -- cycle(114.8, 77.1) -- (114.5, 55.0) -- (112.5,
      59.2).. controls (108.2, 67.9) and (101.8, 76.1) .. (93.2, 83.6) --
      (90.6, 85.9) -- (102.7, 92.6).. controls (109.4, 96.3) and (114.9, 99.3)
      .. (114.9, 99.3).. controls (115.0, 99.2) and (114.9, 89.3) .. (114.8,
      77.1) -- cycle(54.5, 96.1).. controls (59.0, 95.4) and (63.7, 94.1)
      .. (68.4, 92.2).. controls (72.5, 90.6) and (79.0, 87.3) .. (80.4,
      86.1).. controls (81.2, 85.5) and (81.2, 85.5) .. (54.3,
      70.6).. controls (39.5, 62.3) and (26.5, 55.2) .. (25.5, 54.6) -- (23.6,
      53.6) -- (22.3, 56.5).. controls (16.8, 68.3) and (17.1, 79.7) .. (23.2,
      87.4).. controls (26.6, 91.7) and (33.2, 95.2) .. (39.8,
      96.3).. controls (42.7, 96.8) and (51.1, 96.7) .. (54.5, 96.1) --
      cycle(88.5, 80.5).. controls (101.8, 70.1) and (111.2, 56.0) .. (113.4,
      43.3).. controls (114.5, 37.1) and (113.8, 32.2) .. (111.4,
      27.2).. controls (108.8, 21.8) and (103.7, 17.6) .. (97.2,
      15.6).. controls (84.6, 11.5) and (67.6, 14.8) .. (52.0,
      24.3).. controls (42.4, 30.1) and (33.1, 38.8) .. (27.1,
      47.7).. controls (26.3, 48.9) and (25.8, 50.0) .. (26.0,
      50.0).. controls (26.1, 50.1) and (39.5, 57.5) .. (55.6,
      66.5).. controls (71.7, 75.5) and (85.0, 82.9) .. (85.2,
      82.9).. controls (85.4, 82.9) and (86.9, 81.8) .. (88.5, 80.5) -- cycle;
    \end{scope}
  \end{tikzpicture}
}

\def\SngRedAonepthreeDfour{%
  \begin{tikzpicture}[y=1cm, x=1cm,
    yscale=\globalscale,xscale=\globalscale,%
    baseline={([yshift=-1mm]current bounding box.center)}]
    \begin{scope}[fill=black,cm={1,0,0,1,(0, 0)}]
      \path[fill=black] (97.7, 117.6).. controls (96.7, 116.4) and (97.2, 112.7) .. (100.3, 99.2).. controls (102.0, 91.9) and (103.3, 85.9) .. (103.3, 85.9).. controls (103.1, 85.8) and (70.5, 65.0) .. (51.8, 53.1).. controls (42.0, 46.8) and (33.7, 41.5) .. (33.3, 41.3).. controls (32.7, 41.0) and (30.8, 41.5) .. (23.8, 43.4).. controls (17.0, 45.3) and (14.4, 45.9) .. (12.5, 46.0).. controls (9.8, 46.0) and (9.0, 45.7) .. (9.4, 44.7).. controls (9.8, 43.4) and (12.8, 42.2) .. (20.8, 39.8) -- (23.5, 39.1) -- (14.5, 37.2).. controls (4.5, 35.2) and (2.2, 34.5) .. (0.8, 33.3).. controls (-0.1, 32.5) and (-0.1, 32.5) .. (0.5, 31.9).. controls (1.5, 30.8) and (5.5, 31.1) .. (13.9, 32.8).. controls (18.0, 33.6) and (21.7, 34.3) .. (22.0, 34.4).. controls (22.4, 34.5) and (21.2, 33.7) .. (19.4, 32.5).. controls (15.2, 29.7) and (13.4, 28.2) .. (12.7, 26.9).. controls (11.9, 25.6) and (12.2, 25.0) .. (13.5, 25.0).. controls (15.2, 25.0) and (18.2, 26.6) .. (26.1, 31.7) -- (33.4, 36.3) -- (54.7, 30.3).. controls (112.0, 14.1) and (119.8, 11.9) .. (120.1, 11.7).. controls (120.2, 11.5) and (120.8, 9.6) .. (121.3, 7.3).. controls (122.8, 1.6) and (124.2, -0.8) .. (125.5, 0.5).. controls (126.3, 1.3) and (126.4, 3.8) .. (125.7, 7.1).. controls (125.2, 9.1) and (125.2, 10.2) .. (125.4, 10.2).. controls (125.6, 10.2) and (130.2, 9.0) .. (135.6, 7.4).. controls (146.9, 4.3) and (148.6, 3.9) .. (151.0, 3.9).. controls (153.3, 3.9) and (153.8, 4.6) .. (152.6, 5.8).. controls (151.2, 7.2) and (149.2, 7.9) .. (136.2, 11.6).. controls (129.6, 13.5) and (124.1, 15.2) .. (124.0, 15.2).. controls (123.8, 15.5) and (115.4, 52.0) .. (115.4, 52.7).. controls (115.4, 53.0) and (120.1, 54.0) .. (128.3, 55.7).. controls (144.2, 58.9) and (146.8, 59.7) .. (146.8, 61.3).. controls (146.8, 63.2) and (142.8, 62.9) .. (127.6, 59.8).. controls (120.4, 58.3) and (114.4, 57.3) .. (114.3, 57.4).. controls (114.1, 57.7) and (108.3, 82.9) .. (108.3, 83.6).. controls (108.3, 83.9) and (114.0, 87.7) .. (120.9, 92.1).. controls (134.5, 100.8) and (137.3, 102.8) .. (138.1, 104.4).. controls (138.9, 105.9) and (138.5, 106.3) .. (136.8, 106.1).. controls (134.7, 105.8) and (132.1, 104.3) .. (119.6, 96.3).. controls (113.2, 92.2) and (107.7, 88.8) .. (107.5, 88.7).. controls (107.2, 88.6) and (106.2, 92.2) .. (104.3, 100.8).. controls (101.4, 113.7) and (100.6, 116.1) .. (99.5, 117.4).. controls (98.6, 118.3) and (98.3, 118.4) .. (97.7, 117.6) -- cycle(104.7, 80.0).. controls (106.0, 74.8) and (110.1, 56.4) .. (110.0, 56.3).. controls (109.5, 56.0) and (44.5, 43.2) .. (44.9, 43.6).. controls (45.1, 43.8) and (57.3, 51.6) .. (71.8, 60.9).. controls (86.4, 70.1) and (99.5, 78.5) .. (101.1, 79.6).. controls (102.6, 80.6) and (104.0, 81.5) .. (104.1, 81.5).. controls (104.2, 81.5) and (104.5, 80.8) .. (104.7, 80.0) -- cycle(112.0, 48.3).. controls (116.2, 30.0) and (119.2, 16.7) .. (119.1, 16.7).. controls (119.0, 16.5) and (104.5, 20.6) .. (55.5, 34.4).. controls (48.0, 36.5) and (42.2, 38.3) .. (42.4, 38.4).. controls (42.9, 38.6) and (110.0, 52.1) .. (110.7, 52.2).. controls (110.9, 52.2) and (111.4, 50.6) .. (112.0, 48.3) -- cycle;
    \end{scope}
  \end{tikzpicture}
}

\def\SngRedAonepfive{%
  \begin{tikzpicture}[y=1cm, x=1cm,
    yscale=0.8*\globalscale,xscale=\globalscale,%
    baseline={([yshift=-1mm]current bounding box.center)}]
    \begin{scope}[fill=black,cm={1,0,0,1,(0, 0)}]
      \path[fill=black] (126.4, 118.4).. controls (125.7, 117.5) and (123.6,
      114.2) .. (121.9, 110.9) -- (118.7, 105.0) -- (104.4, 96.7).. controls
      (93.6, 90.5) and (89.9, 88.5) .. (89.3, 88.7).. controls (88.8, 88.7)
      and (86.5, 89.3) .. (84.1, 89.9).. controls (77.7, 91.5) and (73.8,
      92.1) .. (67.0, 92.3).. controls (62.3, 92.5) and (60.0, 92.4) .. (57.2,
      91.9).. controls (46.2, 90.2) and (39.2, 85.5) .. (37.0,
      78.4).. controls (36.7, 77.2) and (36.3, 75.2) .. (36.3,
      74.1).. controls (36.3, 70.4) and (38.0, 65.4) .. (40.6, 61.5) -- (41.3,
      60.4) -- (28.7, 53.2).. controls (19.0, 47.7) and (15.6, 45.6) .. (14.0,
      44.1).. controls (12.5, 42.6) and (12.0, 41.9) .. (12.2,
      41.5).. controls (12.4, 41.0) and (12.7, 40.9) .. (14.0,
      41.0).. controls (16.2, 41.3) and (19.1, 42.8) .. (32.3, 50.4) -- (43.9,
      57.1) -- (47.1, 54.0).. controls (55.0, 46.5) and (66.3, 40.4) .. (78.7,
      36.9).. controls (80.5, 36.3) and (82.1, 35.8) .. (82.1,
      35.8).. controls (82.2, 35.7) and (79.3, 30.1) .. (75.6,
      23.2).. controls (67.8, 8.4) and (66.8, 5.8) .. (68.2, 4.6).. controls
      (68.9, 4.1) and (69.6, 4.5) .. (70.9, 5.9).. controls (72.2, 7.2) and
      (72.4, 7.6) .. (79.9, 21.9).. controls (83.5, 28.8) and (86.6, 34.5)
      .. (86.7, 34.6).. controls (86.9, 34.7) and (88.3, 34.6) .. (90.0,
      34.4).. controls (97.6, 33.2) and (105.3, 33.4) .. (111.7,
      34.9).. controls (128.9, 39.0) and (133.7, 52.5) .. (123.1,
      66.6).. controls (120.7, 69.9) and (116.1, 74.2) .. (112.6,
      76.7).. controls (111.2, 77.7) and (110.1, 78.7) .. (110.1,
      78.8).. controls (110.1, 79.1) and (112.8, 84.3) .. (116.2, 90.7) --
      (122.2, 102.1) -- (129.3, 106.2).. controls (138.3, 111.3) and (139.8,
      112.3) .. (141.1, 113.8).. controls (142.6, 115.5) and (142.4, 116.2)
      .. (140.5, 116.2).. controls (138.9, 116.2) and (135.1, 114.5)
      .. (129.8, 111.4).. controls (128.0, 110.3) and (126.4, 109.4)
      .. (126.4, 109.5).. controls (126.3, 109.6) and (127.0, 111.0)
      .. (127.8, 112.7).. controls (129.0, 115.3) and (129.3, 116.3)
      .. (129.3, 117.8).. controls (129.3, 119.6) and (129.2, 119.8)
      .. (128.6, 119.9).. controls (128.0, 119.9) and (127.4, 119.5)
      .. (126.4, 118.4) -- cycle(113.8, 95.5).. controls (113.1, 94.3) and
      (111.2, 90.7) .. (109.5, 87.5).. controls (107.9, 84.3) and (106.4,
      81.6) .. (106.3, 81.4).. controls (106.1, 81.3) and (103.8, 82.3)
      .. (101.1, 83.6).. controls (98.4, 85.0) and (96.0, 86.1) .. (95.7,
      86.1).. controls (95.5, 86.1) and (95.3, 86.2) .. (95.3,
      86.4).. controls (95.3, 86.6) and (114.2, 97.7) .. (114.7,
      97.7).. controls (114.9, 97.7) and (114.4, 96.7) .. (113.8, 95.5) --
      cycle(71.5, 88.5).. controls (76.7, 88.0) and (85.0, 86.3) .. (85.0,
      85.8).. controls (85.0, 85.5) and (45.4, 62.7) .. (44.7,
      62.5).. controls (44.2, 62.4) and (42.2, 65.5) .. (41.0,
      68.5).. controls (39.7, 71.8) and (39.7, 76.0) .. (41.0,
      78.7).. controls (44.7, 86.3) and (56.7, 90.1) .. (71.5, 88.5) --
      cycle(94.9, 82.5).. controls (97.8, 81.2) and (102.7, 78.9) .. (104.0,
      78.0).. controls (104.3, 77.8) and (101.9, 72.9) .. (94.5,
      58.9).. controls (89.0, 48.5) and (84.4, 39.8) .. (84.2,
      39.6).. controls (83.9, 39.3) and (83.0, 39.4) .. (80.3,
      40.2).. controls (70.3, 43.1) and (62.8, 46.7) .. (55.1,
      52.1).. controls (51.5, 54.7) and (46.9, 58.9) .. (47.4,
      59.1).. controls (47.7, 59.2) and (57.4, 64.8) .. (69.1,
      71.6).. controls (80.9, 78.4) and (90.7, 83.9) .. (91.0,
      84.0).. controls (91.2, 84.0) and (93.0, 83.3) .. (94.9, 82.5) --
      cycle(110.8, 73.5).. controls (113.7, 71.4) and (118.1, 67.1) .. (120.1,
      64.5).. controls (127.7, 54.8) and (126.6, 45.7) .. (117.3,
      40.9).. controls (111.8, 38.0) and (107.3, 37.3) .. (98.1,
      37.5).. controls (94.3, 37.6) and (90.6, 37.9) .. (89.9, 38.0) -- (88.5,
      38.3) -- (96.0, 52.4).. controls (107.5, 74.3) and (108.0, 75.1)
      .. (108.3, 75.1).. controls (108.4, 75.1) and (109.6, 74.4) .. (110.8,
      73.5) -- cycle;
    \end{scope}
  \end{tikzpicture}
}

\def\SngRedAonepsix{%
  \begin{tikzpicture}[y=1cm, x=1cm,
    yscale=0.8*\globalscale,xscale=\globalscale,%
    baseline={([yshift=-1mm]current bounding box.center)}]
    \begin{scope}[fill=black,cm={1,0,0,1,(0, 0)}]
      \path[fill=black] (57.2, 120.7).. controls (56.4, 120.2) and (56.6,
      117.2) .. (57.6, 114.6).. controls (58.1, 113.4) and (59.4, 110.4)
      .. (60.5, 108.0).. controls (61.6, 105.5) and (62.5, 103.3) .. (62.5,
      103.0).. controls (62.5, 102.6) and (57.7, 90.8) .. (53.7, 81.4) --
      (53.0, 79.7) -- (37.9, 92.0).. controls (29.7, 98.7) and (22.0, 104.8)
      .. (20.9, 105.4).. controls (19.5, 106.2) and (18.5, 106.5) .. (17.6,
      106.5).. controls (16.5, 106.4) and (16.4, 106.3) .. (16.5,
      105.2).. controls (16.6, 102.8) and (18.8, 100.8) .. (35.1,
      87.5).. controls (44.0, 80.3) and (50.8, 74.5) .. (50.7,
      74.3).. controls (50.6, 74.0) and (48.6, 69.2) .. (46.2, 63.6) -- (41.9,
      53.3) -- (27.6, 43.9).. controls (19.7, 38.7) and (12.3, 33.7) .. (11.2,
      32.9).. controls (8.2, 30.6) and (6.8, 27.9) .. (8.7, 27.9).. controls
      (10.6, 27.9) and (13.8, 29.7) .. (25.7, 37.6).. controls (32.8, 42.3)
      and (38.7, 46.1) .. (38.7, 46.1).. controls (38.7, 46.0) and (35.8,
      38.9) .. (32.1, 30.1).. controls (28.4, 21.4) and (25.1, 13.4) .. (24.8,
      12.2).. controls (24.2, 9.7) and (24.4, 8.1) .. (25.5, 8.1).. controls
      (26.5, 8.1) and (27.5, 9.2) .. (28.7, 11.5).. controls (29.4, 12.7) and
      (43.4, 45.8) .. (45.0, 50.0).. controls (45.3, 50.7) and (64.2, 63.2)
      .. (64.7, 63.1).. controls (65.3, 63.0) and (88.0, 44.6) .. (89.1,
      43.3).. controls (89.5, 42.9) and (93.3, 34.7) .. (97.6,
      25.2).. controls (104.1, 10.7) and (105.6, 7.7) .. (106.8,
      6.4).. controls (108.4, 4.8) and (109.1, 4.7) .. (109.5, 6.4).. controls
      (109.9, 8.1) and (108.8, 11.1) .. (103.2, 23.6).. controls (100.1, 30.3)
      and (97.5, 36.1) .. (97.4, 36.3).. controls (97.2, 37.0) and (97.2,
      36.9) .. (113.2, 23.9).. controls (127.4, 12.3) and (128.1, 11.8)
      .. (130.2, 11.1).. controls (133.2, 10.0) and (133.8, 11.9) .. (131.4,
      15.1).. controls (130.1, 17.0) and (129.3, 17.6) .. (108.0,
      35.0).. controls (99.7, 41.7) and (92.8, 47.4) .. (92.5,
      47.7).. controls (92.0, 48.3) and (80.8, 73.3) .. (80.8,
      73.8).. controls (80.8, 73.9) and (91.9, 81.5) .. (105.4,
      90.4).. controls (130.1, 106.9) and (132.8, 108.8) .. (132.6,
      110.5).. controls (132.5, 111.2) and (132.3, 111.3) .. (131.2,
      111.2).. controls (130.6, 111.2) and (129.2, 110.7) .. (128.2,
      110.3).. controls (127.3, 109.8) and (115.8, 102.4) .. (102.7,
      93.7).. controls (89.7, 85.1) and (79.0, 78.0) .. (78.9,
      78.0).. controls (78.7, 78.0) and (67.6, 102.8) .. (67.6,
      103.2).. controls (67.6, 103.4) and (68.7, 106.1) .. (70.0,
      109.1).. controls (72.3, 114.6) and (73.3, 118.3) .. (72.9,
      119.4).. controls (72.0, 121.7) and (69.6, 119.2) .. (67.2,
      113.5).. controls (66.3, 111.3) and (65.4, 109.3) .. (65.3,
      109.1).. controls (65.1, 108.8) and (64.2, 110.3) .. (63.0,
      113.2).. controls (60.4, 119.0) and (58.5, 121.5) .. (57.2, 120.7) --
      cycle(70.4, 85.9).. controls (73.0, 80.1) and (75.1, 75.4) .. (75.0,
      75.3).. controls (73.3, 74.0) and (66.2, 69.5) .. (65.9,
      69.5).. controls (65.5, 69.5) and (57.6, 75.8) .. (56.8,
      76.7).. controls (56.6, 77.0) and (57.6, 79.8) .. (60.7,
      87.3).. controls (63.7, 94.2) and (65.2, 97.3) .. (65.4,
      97.0).. controls (65.5, 96.7) and (67.8, 91.7) .. (70.4, 85.9) --
      cycle(58.0, 68.9).. controls (59.6, 67.6) and (61.0, 66.4) .. (61.0,
      66.3).. controls (61.0, 66.0) and (48.6, 57.7) .. (48.5,
      57.9).. controls (48.5, 58.0) and (49.7, 61.1) .. (51.3,
      64.8).. controls (53.1, 69.2) and (54.3, 71.6) .. (54.5,
      71.5).. controls (54.8, 71.4) and (56.3, 70.2) .. (58.0, 68.9) --
      cycle(80.6, 63.1).. controls (85.1, 53.1) and (84.9, 53.7) .. (83.6,
      54.8).. controls (76.2, 60.7) and (69.5, 66.2) .. (69.5,
      66.4).. controls (69.5, 66.8) and (76.3, 71.2) .. (76.7,
      71.1).. controls (76.9, 71.1) and (78.7, 67.5) .. (80.6, 63.1) -- cycle;
    \end{scope}
  \end{tikzpicture}
}

\def\SngNonLtwoC{%
  \begin{tikzpicture}[y=1cm, x=1cm,
    yscale=\globalscale,xscale=\globalscale,%
    baseline={([yshift=-1mm]current bounding box.center)}]
    \begin{scope}[fill=black,cm={1,0,0,1,(0, 0)}]
      \path[fill=black] (62.1, 82.0).. controls (53.4, 80.8) and (45.3, 78.0)
      .. (39.5, 74.1).. controls (36.0, 71.8) and (32.3, 67.8) .. (31.0, 64.9)
      -- (29.8, 62.5) -- (16.5, 62.5).. controls (9.2, 62.4) and (3.0, 62.3)
      .. (2.7, 62.2).. controls (1.8, 61.9) and (1.3, 60.8) .. (1.8,
      60.0).. controls (2.1, 59.3) and (2.5, 59.3) .. (15.5, 59.3) -- (28.9,
      59.3) -- (28.9, 57.2) -- (28.9, 55.0) -- (16.2, 55.0).. controls (8.1,
      55.0) and (3.2, 54.9) .. (2.8, 54.7).. controls (2.0, 54.3) and (1.9,
      52.9) .. (2.5, 52.3).. controls (2.9, 52.0) and (6.4, 51.9) .. (16.2,
      51.9) -- (29.5, 51.9) -- (30.6, 49.4).. controls (32.0, 46.5) and (35.6,
      42.4) .. (39.0, 40.0).. controls (49.8, 32.3) and (68.2, 29.3) .. (84.1,
      32.5).. controls (96.9, 35.1) and (107.1, 41.4) .. (111.0, 49.2) --
      (112.3, 51.9) -- (125.4, 51.9).. controls (135.1, 51.9) and (138.6,
      52.0) .. (138.9, 52.3).. controls (139.5, 52.9) and (139.5, 53.8)
      .. (138.8, 54.5).. controls (138.3, 55.0) and (136.6, 55.0) .. (125.6,
      55.0) -- (112.9, 55.0) -- (112.9, 57.2) -- (112.9, 59.3) -- (125.6,
      59.3) -- (138.3, 59.3) -- (138.7, 60.1).. controls (139.0, 60.9) and
      (139.0, 61.1) .. (138.5, 61.7).. controls (137.9, 62.4) and (137.4,
      62.4) .. (124.9, 62.4) -- (112.0, 62.4) -- (111.1, 64.5).. controls
      (106.9, 73.4) and (93.7, 80.5) .. (78.2, 82.2).. controls (73.8, 82.7)
      and (66.3, 82.6) .. (62.1, 82.0) -- cycle(83.0, 78.4).. controls (94.6,
      76.0) and (104.2, 70.4) .. (107.9, 63.8) -- (108.7, 62.4) -- (70.9,
      62.4) -- (33.2, 62.4) -- (33.9, 63.8).. controls (34.8, 65.8) and (38.2,
      69.3) .. (40.7, 71.1).. controls (45.5, 74.4) and (51.9, 77.0) .. (58.8,
      78.3).. controls (64.2, 79.3) and (64.4, 79.3) .. (72.1,
      79.2).. controls (77.6, 79.1) and (80.0, 79.0) .. (83.0, 78.4) --
      cycle(109.7, 57.2) -- (109.7, 55.0) -- (70.9, 55.0).. controls (32.2,
      55.0) and (32.1, 55.0) .. (31.9, 55.7).. controls (31.8, 56.2) and
      (31.8, 57.1) .. (31.9, 57.9) -- (32.2, 59.3) -- (70.9, 59.3) -- (109.7,
      59.3) -- (109.7, 57.2) -- cycle(108.7, 51.6).. controls (108.7, 50.8)
      and (106.3, 47.4) .. (104.7, 45.8).. controls (90.2, 31.0) and (54.2,
      30.4) .. (38.2, 44.7).. controls (36.1, 46.6) and (33.2, 50.4) .. (33.2,
      51.4).. controls (33.2, 51.8) and (39.0, 51.9) .. (70.9,
      51.9).. controls (91.7, 51.9) and (108.7, 51.8) .. (108.7, 51.6) --
      cycle;
    \end{scope}
  \end{tikzpicture}
}

\def\SngNonLtwoCBis{%
  \begin{tikzpicture}[y=1cm, x=1cm,
    yscale=\globalscale,xscale=\globalscale,%
    baseline={([yshift=-1mm]current bounding box.center)}]
    \begin{scope}[fill=black,cm={1,0,0,1,(0, 0)}]
      \path[fill=black] (2.6, 89.0).. controls (1.8, 88.7) and (2.0, 86.9)
      .. (2.8, 86.4).. controls (3.2, 86.2) and (25.0, 86.1) .. (70.6,
      86.1).. controls (130.8, 86.1) and (137.9, 86.1) .. (138.6,
      86.6).. controls (139.4, 87.2) and (139.6, 88.2) .. (138.9,
      88.8).. controls (138.6, 89.1) and (123.0, 89.3) .. (70.7,
      89.2).. controls (33.5, 89.2) and (2.8, 89.1) .. (2.6, 89.0) -- cycle;
      \path[fill=black] (2.0, 82.1).. controls (1.3, 81.1) and (1.3, 81.0)
      .. (2.0, 80.3).. controls (2.5, 79.8) and (5.2, 79.7) .. (26.7, 79.7) --
      (51.0, 79.7) -- (48.2, 78.5).. controls (41.9, 76.0) and (38.3, 73.7)
      .. (34.7, 70.0).. controls (28.2, 63.4) and (27.0, 54.6) .. (31.7,
      47.4).. controls (33.2, 45.2) and (37.0, 41.3) .. (39.5,
      39.6).. controls (48.3, 33.7) and (63.0, 30.3) .. (75.7,
      31.4).. controls (102.3, 33.6) and (118.6, 48.8) .. (111.1,
      64.4).. controls (108.4, 70.1) and (101.2, 75.8) .. (93.1, 78.7) --
      (90.8, 79.6) -- (114.6, 79.7).. controls (135.1, 79.9) and (138.4, 80.0)
      .. (138.6, 80.4).. controls (139.1, 81.2) and (139.1, 81.7) .. (138.4,
      82.3).. controls (137.8, 82.9) and (132.5, 82.9) .. (70.1, 82.9) --
      (2.5, 82.9) -- (2.0, 82.1) -- cycle(83.0, 78.4).. controls (95.3, 75.9)
      and (104.8, 70.1) .. (108.4, 62.8).. controls (109.5, 60.7) and (109.5,
      60.3) .. (109.5, 57.0).. controls (109.5, 53.6) and (109.5, 53.3)
      .. (108.2, 50.7).. controls (101.5, 37.2) and (75.4, 30.4) .. (54.0,
      36.6).. controls (41.6, 40.1) and (33.5, 46.9) .. (32.1,
      54.8).. controls (30.2, 64.9) and (41.8, 75.1) .. (58.8,
      78.3).. controls (64.2, 79.3) and (64.4, 79.3) .. (72.1,
      79.2).. controls (77.6, 79.1) and (80.0, 79.0) .. (83.0, 78.4) -- cycle;
    \end{scope}
  \end{tikzpicture}
}

\def\SngNonLLLtwo{%
  \begin{tikzpicture}[y=1cm, x=1cm,
    yscale=\globalscale,xscale=\globalscale,%
    baseline={([yshift=-1mm]current bounding box.center)}]
    \begin{scope}[fill=black,cm={1,0,0,1,(0, 0)}]
      \path[fill=black] (43.8, 115.9).. controls (42.2, 114.7) and (42.5,
      114.1) .. (53.6, 97.2).. controls (60.5, 86.8) and (64.1, 80.9)
      .. (64.0, 80.6).. controls (63.8, 80.4) and (55.3, 68.7) .. (45.0, 54.7)
      -- (26.3, 29.3) -- (13.8, 29.3).. controls (0.6, 29.3) and (0.4, 29.2)
      .. (0.4, 27.7).. controls (0.4, 26.1) and (0.6, 26.1) .. (12.6, 26.1) --
      (24.0, 26.1) -- (22.4, 24.0) -- (20.8, 21.9) -- (11.1, 21.9).. controls
      (4.2, 21.9) and (1.4, 21.8) .. (1.1, 21.4).. controls (0.6, 20.9) and
      (0.6, 19.6) .. (1.2, 19.1).. controls (1.5, 18.8) and (4.6, 18.7)
      .. (10.0, 18.6) -- (18.4, 18.5) -- (13.6, 12.0).. controls (11.0, 8.4)
      and (8.8, 5.2) .. (8.8, 4.7).. controls (8.8, 3.7) and (9.9, 2.8)
      .. (11.0, 2.8).. controls (12.1, 2.8) and (12.5, 3.2) .. (19.2, 12.5) --
      (23.8, 18.7) -- (64.5, 18.7) -- (105.2, 18.7) -- (111.0, 9.8).. controls
      (114.2, 4.9) and (116.9, 0.8) .. (117.1, 0.6).. controls (117.2, 0.5)
      and (117.7, 0.3) .. (118.3, 0.2).. controls (119.3, -0.1) and (120.7,
      1.0) .. (120.7, 2.1).. controls (120.7, 2.5) and (118.4, 6.4) .. (115.5,
      10.6).. controls (112.7, 14.9) and (110.4, 18.5) .. (110.4,
      18.6).. controls (110.4, 18.6) and (116.3, 18.7) .. (123.6,
      18.7).. controls (137.4, 18.7) and (137.6, 18.7) .. (137.6,
      20.3).. controls (137.6, 21.9) and (137.5, 21.9) .. (122.4, 21.9) --
      (108.1, 21.9) -- (106.8, 23.9) -- (105.5, 25.9) -- (120.9,
      26.0).. controls (131.6, 26.1) and (136.4, 26.2) .. (136.8,
      26.5).. controls (137.4, 27.1) and (137.3, 28.5) .. (136.6,
      28.9).. controls (136.1, 29.1) and (130.2, 29.3) .. (119.7,
      29.3).. controls (104.7, 29.3) and (103.4, 29.3) .. (102.9,
      29.9).. controls (101.2, 31.9) and (69.5, 80.9) .. (69.7,
      81.2).. controls (69.9, 81.4) and (75.0, 88.5) .. (81.2,
      96.9).. controls (87.4, 105.3) and (92.4, 112.4) .. (92.4,
      112.6).. controls (92.4, 113.8) and (90.7, 115.2) .. (89.6,
      114.8).. controls (89.2, 114.7) and (84.0, 107.9) .. (77.9,
      99.6).. controls (71.9, 91.4) and (66.9, 84.7) .. (66.9,
      84.7).. controls (66.8, 84.8) and (62.4, 91.6) .. (57.0,
      99.8).. controls (51.6, 108.1) and (46.8, 115.2) .. (46.3,
      115.6).. controls (45.3, 116.6) and (44.9, 116.6) .. (43.8, 115.9) --
      cycle(82.6, 53.2) -- (98.1, 29.5) -- (81.5, 29.4).. controls (72.4,
      29.3) and (57.4, 29.3) .. (48.3, 29.4) -- (31.7, 29.5) -- (49.1,
      53.2).. controls (58.7, 66.2) and (66.6, 76.9) .. (66.7,
      76.9).. controls (66.9, 76.9) and (74.0, 66.2) .. (82.6, 53.2) --
      cycle(101.1, 24.9).. controls (101.5, 24.3) and (102.1, 23.4) .. (102.4,
      22.9) -- (103.1, 21.9) -- (64.6, 21.9) -- (26.1, 21.9) -- (27.7, 24.0)
      -- (29.3, 26.1) -- (64.8, 26.1) -- (100.4, 26.1) -- (101.1, 24.9) --
      cycle;
    \end{scope}
  \end{tikzpicture}
}

\def\SngNonLLLtwoBis{%
  \begin{tikzpicture}[y=1cm, x=1cm,
    yscale=\globalscale,xscale=\globalscale,%
    baseline={([yshift=-1mm]current bounding box.center)}]
    \begin{scope}[fill=black,cm={1,0,0,1,(0, 0)}]
      \path[fill=black] (89.4, 114.8).. controls (89.1, 114.7) and (74.6,
      95.1) .. (57.1, 71.2) -- (25.3, 27.9) -- (13.2, 27.9).. controls (4.3,
      27.9) and (1.1, 27.8) .. (0.8, 27.4).. controls (0.1, 26.8) and (0.3,
      25.5) .. (1.0, 25.0).. controls (1.4, 24.8) and (5.5, 24.7) .. (11.7,
      24.7) -- (21.7, 24.7) -- (18.9, 23.2) -- (16.1, 21.9) -- (8.8,
      21.9).. controls (1.2, 21.9) and (0.7, 21.8) .. (0.7, 20.3).. controls
      (0.7, 18.9) and (1.4, 18.7) .. (5.3, 18.7) -- (9.0, 18.7) -- (6.2,
      17.3).. controls (3.0, 15.7) and (2.0, 14.7) .. (2.3, 13.6).. controls
      (2.6, 12.7) and (3.5, 12.0) .. (4.3, 12.0).. controls (4.7, 12.0) and
      (7.8, 13.3) .. (11.4, 15.0).. controls (14.9, 16.7) and (17.8, 18.0)
      .. (17.8, 17.9).. controls (17.9, 17.9) and (15.9, 15.0) .. (13.4,
      11.6).. controls (10.9, 8.3) and (8.8, 5.2) .. (8.8, 4.7).. controls
      (8.8, 3.7) and (9.9, 2.8) .. (11.0, 2.8).. controls (12.1, 2.8) and
      (12.5, 3.2) .. (19.2, 12.5) -- (23.8, 18.7) -- (80.3, 18.7).. controls
      (123.7, 18.7) and (136.8, 18.8) .. (137.2, 19.1).. controls (137.7,
      19.6) and (137.7, 20.9) .. (137.2, 21.4).. controls (136.8, 21.8) and
      (124.0, 21.9) .. (81.5, 21.9) -- (26.3, 21.9) -- (28.9, 23.3) -- (31.6,
      24.7) -- (83.7, 24.7).. controls (119.0, 24.7) and (136.1, 24.8)
      .. (136.6, 25.0).. controls (137.3, 25.5) and (137.4, 26.9) .. (136.8,
      27.5).. controls (136.4, 27.8) and (122.9, 27.9) .. (87.6, 27.9) --
      (39.0, 28.0) -- (84.6, 49.6).. controls (115.4, 64.1) and (130.4, 71.3)
      .. (130.8, 71.9).. controls (131.6, 73.0) and (130.7, 74.6) .. (129.3,
      74.7).. controls (128.5, 74.8) and (116.6, 69.3) .. (82.6,
      53.2).. controls (28.9, 27.9) and (31.6, 29.2) .. (32.0,
      29.7).. controls (32.1, 30.0) and (45.3, 47.9) .. (61.2,
      69.7).. controls (95.7, 116.8) and (92.4, 112.2) .. (92.4,
      112.8).. controls (92.4, 113.4) and (90.8, 115.0) .. (90.2,
      115.0).. controls (90.0, 115.0) and (89.6, 114.9) .. (89.4, 114.8) --
      cycle;
    \end{scope}
  \end{tikzpicture}
}

\def\SngNonCtwo{%
  \begin{tikzpicture}[y=1cm, x=1cm,
    yscale=\globalscale,xscale=\globalscale,%
    baseline={([yshift=-1mm]current bounding box.center)}]
    \begin{scope}[fill=black,cm={1,0,0,1,(0, 0)}]
      \path[fill=black] (63.1, 108.4).. controls (56.3, 107.9) and (46.6,
      105.9) .. (40.1, 103.7).. controls (24.5, 98.3) and (12.3, 88.7)
      .. (7.4, 77.8).. controls (-1.2, 58.8) and (11.0, 38.3) .. (36.9,
      28.3).. controls (47.6, 24.1) and (58.6, 22.2) .. (71.6,
      22.2).. controls (81.4, 22.2) and (90.1, 23.4) .. (98.5,
      25.9).. controls (129.8, 35.2) and (145.1, 57.6) .. (134.6,
      78.8).. controls (132.5, 83.0) and (130.3, 85.9) .. (126.2,
      89.9).. controls (113.7, 101.8) and (92.8, 108.8) .. (70.7,
      108.6).. controls (67.4, 108.5) and (64.0, 108.5) .. (63.1, 108.4) --
      cycle(81.5, 106.4).. controls (99.4, 104.5) and (114.7, 98.2) .. (124.8,
      88.8).. controls (129.9, 83.9) and (132.8, 79.6) .. (134.8,
      73.8).. controls (140.5, 56.7) and (128.3, 38.9) .. (104.6,
      29.8).. controls (93.2, 25.4) and (82.9, 23.8) .. (69.0,
      24.1).. controls (60.2, 24.3) and (56.8, 24.7) .. (49.4,
      26.3).. controls (21.3, 32.6) and (3.2, 51.3) .. (6.8, 70.3).. controls
      (8.1, 76.8) and (11.1, 82.1) .. (16.9, 87.8).. controls (26.9, 97.8) and
      (42.4, 104.4) .. (60.7, 106.4).. controls (66.3, 107.0) and (75.8,
      107.0) .. (81.5, 106.4) -- cycle;
      \path[fill=black] (59.7, 102.8).. controls (33.3, 99.5) and (13.9, 87.2)
      .. (9.7, 71.2).. controls (8.9, 68.2) and (8.9, 62.6) .. (9.7,
      59.7).. controls (11.1, 54.4) and (13.8, 50.0) .. (18.3,
      45.5).. controls (27.8, 36.2) and (42.4, 30.1) .. (60.3,
      28.0).. controls (69.2, 27.1) and (81.4, 27.5) .. (90.1,
      29.1).. controls (104.4, 31.8) and (117.2, 37.8) .. (124.9,
      45.5).. controls (131.0, 51.6) and (134.1, 58.2) .. (134.1,
      65.5).. controls (134.1, 74.6) and (128.8, 83.4) .. (119.1,
      90.3).. controls (115.4, 92.8) and (107.5, 96.8) .. (102.8,
      98.4).. controls (98.4, 100.0) and (92.3, 101.5) .. (87.0,
      102.3).. controls (80.9, 103.3) and (65.7, 103.6) .. (59.7, 102.8) --
      cycle(78.8, 101.2).. controls (95.7, 99.9) and (110.3, 94.8) .. (120.5,
      86.6).. controls (127.0, 81.4) and (131.0, 75.1) .. (132.0,
      68.2).. controls (133.4, 58.9) and (127.7, 48.9) .. (117.0,
      41.7).. controls (90.5, 24.0) and (44.2, 25.7) .. (21.6,
      45.1).. controls (13.1, 52.5) and (9.4, 61.9) .. (11.6, 70.6).. controls
      (16.6, 90.1) and (46.5, 103.7) .. (78.8, 101.2) -- cycle;
    \end{scope}
  \end{tikzpicture}
}

\def\SngNonLtwoLtwo{%
  \begin{tikzpicture}[y=1cm, x=1cm,
    yscale=\globalscale,xscale=\globalscale,%
    baseline={([yshift=-1mm]current bounding box.center)}]
    \begin{scope}[fill=black,cm={1,0,0,1,(0, 0)}]
      \path[fill=black] (97.6, 111.0).. controls (97.1, 110.6) and (82.7,
      91.1) .. (65.6, 67.7).. controls (48.5, 44.3) and (34.4, 25.2) .. (34.4,
      25.1).. controls (34.3, 25.0) and (33.3, 25.0) .. (32.2, 25.1) -- (30.1,
      25.2) -- (60.9, 67.2).. controls (84.1, 98.8) and (91.7, 109.4)
      .. (91.7, 110.0).. controls (91.5, 111.2) and (90.5, 111.8) .. (89.7,
      111.3).. controls (89.4, 111.1) and (74.9, 91.6) .. (57.6, 68.0) --
      (26.0, 25.0) -- (14.1, 25.0).. controls (7.5, 25.0) and (1.8, 24.9)
      .. (1.3, 24.7).. controls (0.2, 24.3) and (-0.1, 23.5) .. (0.4,
      22.5).. controls (0.7, 21.9) and (1.2, 21.9) .. (12.2, 21.9) -- (23.7,
      21.9) -- (22.5, 20.3) -- (21.4, 18.7) -- (11.1, 18.6) -- (0.9, 18.5) --
      (0.8, 17.5).. controls (0.7, 16.9) and (0.8, 16.3) .. (1.1,
      16.1).. controls (1.3, 15.8) and (4.4, 15.6) .. (10.2, 15.5) -- (18.9,
      15.3) -- (14.9, 9.9).. controls (9.0, 1.9) and (9.0, 1.9) .. (9.5,
      1.1).. controls (9.8, 0.6) and (10.3, 0.4) .. (10.9, 0.4).. controls
      (11.7, 0.4) and (12.5, 1.2) .. (17.3, 7.8).. controls (20.3, 11.9) and
      (22.9, 15.3) .. (23.1, 15.4).. controls (23.4, 15.5) and (24.4, 15.5)
      .. (25.4, 15.5) -- (27.2, 15.3) -- (22.4, 8.9).. controls (19.8, 5.3)
      and (17.6, 2.2) .. (17.6, 1.9).. controls (17.6, 1.2) and (18.7, 0.4)
      .. (19.4, 0.4).. controls (19.8, 0.4) and (21.8, 2.8) .. (25.6, 7.9) --
      (31.2, 15.5) -- (83.6, 15.5).. controls (112.5, 15.5) and (136.4, 15.6)
      .. (136.7, 15.7).. controls (137.6, 16.1) and (138.0, 17.2) .. (137.6,
      18.0).. controls (137.2, 18.7) and (136.8, 18.7) .. (85.5,
      18.7).. controls (57.1, 18.7) and (33.9, 18.8) .. (33.9,
      19.0).. controls (33.9, 19.1) and (34.4, 19.8) .. (35.0, 20.6) -- (36.0,
      21.9) -- (86.5, 21.9) -- (137.1, 22.0) -- (137.2, 23.2).. controls
      (137.3, 24.2) and (137.1, 24.4) .. (136.3, 24.7).. controls (135.7,
      24.9) and (118.4, 25.0) .. (86.9, 25.0).. controls (60.3, 25.0) and
      (38.5, 25.1) .. (38.5, 25.2).. controls (38.5, 25.3) and (51.8, 43.6)
      .. (68.2, 65.9).. controls (84.5, 88.2) and (98.4, 107.1) .. (99.1,
      108.1).. controls (99.9, 109.3) and (100.2, 110.0) .. (100.0,
      110.6).. controls (99.8, 111.7) and (98.6, 111.9) .. (97.6, 111.0) --
      cycle(31.2, 20.8).. controls (30.1, 18.9) and (29.5, 18.7) .. (27.4,
      18.8) -- (25.4, 18.9) -- (26.5, 20.4).. controls (27.6, 21.8) and (27.6,
      21.9) .. (29.7, 21.9) -- (31.8, 21.9) -- (31.2, 20.8) -- cycle;
    \end{scope}
  \end{tikzpicture}
}

\def\SngNonLLthree{%
  \begin{tikzpicture}[y=1cm, x=1cm,
    yscale=\globalscale,xscale=\globalscale,%
    baseline={([yshift=-1mm]current bounding box.center)}]
    \begin{scope}[fill=black,cm={1,0,0,1,(0, 0)}]
      \path[fill=black] (89.0, 110.7).. controls (88.6, 110.2) and (75.5,
      92.5) .. (60.0, 71.3).. controls (44.4, 50.0) and (31.3, 32.3) .. (30.9,
      31.9) -- (30.2, 31.0) -- (15.7, 31.0).. controls (1.4, 31.0) and (1.1,
      31.0) .. (0.7, 30.3).. controls (0.4, 29.8) and (0.4, 29.4) .. (0.7,
      28.9).. controls (1.1, 28.3) and (1.5, 28.2) .. (14.7, 28.2) -- (28.3,
      28.0) -- (27.2, 26.6) -- (26.0, 25.0) -- (14.1, 25.0).. controls (7.5,
      25.0) and (1.8, 24.9) .. (1.3, 24.7).. controls (0.2, 24.3) and (-0.1,
      23.5) .. (0.4, 22.5).. controls (0.7, 21.9) and (1.2, 21.9) .. (12.2,
      21.9) -- (23.7, 21.9) -- (22.5, 20.3) -- (21.4, 18.7) -- (11.1, 18.6) --
      (0.9, 18.5) -- (0.8, 17.5).. controls (0.7, 16.9) and (0.8, 16.3)
      .. (1.1, 16.1).. controls (1.3, 15.8) and (4.4, 15.6) .. (10.2, 15.5) --
      (18.9, 15.3) -- (14.9, 9.9).. controls (9.0, 1.8) and (9.0, 1.9)
      .. (9.5, 1.1).. controls (9.8, 0.5) and (10.2, 0.4) .. (10.9,
      0.4).. controls (11.8, 0.5) and (12.7, 1.6) .. (17.4, 7.9) -- (22.8,
      15.3) -- (79.9, 15.5).. controls (111.3, 15.6) and (137.1, 15.7)
      .. (137.1, 15.8).. controls (137.7, 16.4) and (137.9, 17.4) .. (137.6,
      18.0).. controls (137.2, 18.7) and (136.8, 18.7) .. (81.3, 18.8) --
      (25.4, 18.9) -- (26.5, 20.4) -- (27.6, 21.9) -- (82.3, 21.9) -- (137.1,
      22.0) -- (137.2, 23.2).. controls (137.3, 24.2) and (137.1, 24.4)
      .. (136.3, 24.7).. controls (135.7, 24.9) and (116.9, 25.0) .. (82.7,
      25.0) -- (29.9, 25.0) -- (30.9, 26.4).. controls (31.4, 27.1) and (32.1,
      27.8) .. (32.2, 27.9).. controls (32.5, 28.1) and (56.2, 28.2) .. (84.9,
      28.2) -- (137.3, 28.2) -- (137.5, 29.0).. controls (137.7, 29.5) and
      (137.7, 30.1) .. (137.5, 30.4).. controls (137.2, 31.0) and (134.4,
      31.0) .. (85.9, 31.0).. controls (57.7, 31.0) and (34.6, 31.1) .. (34.6,
      31.2).. controls (34.6, 31.3) and (47.4, 48.9) .. (63.1,
      70.3).. controls (78.8, 91.7) and (91.7, 109.5) .. (91.7,
      109.9).. controls (91.7, 111.5) and (90.2, 112.0) .. (89.0, 110.7) --
      cycle;
    \end{scope}
  \end{tikzpicture}
}

\def\SngNonLfour{%
  \begin{tikzpicture}[y=1cm, x=1cm,
    yscale=\globalscale,xscale=\globalscale,%
    baseline={([yshift=-1mm]current bounding box.center)}]
    \begin{scope}[fill=black,cm={1,0,0,1,(0, 0)}]
      \path[fill=black] (1.3, 22.0).. controls (0.6, 21.4) and (0.6, 21.0)
      .. (1.1, 20.1).. controls (1.4, 19.4) and (2.6, 19.4) .. (69.5, 19.4) --
      (137.6, 19.4) -- (138.0, 20.3).. controls (138.3, 21.0) and (138.3,
      21.3) .. (137.8, 21.9).. controls (137.1, 22.6) and (136.7, 22.6)
      .. (69.5, 22.6).. controls (8.4, 22.6) and (1.8, 22.5) .. (1.3, 22.0) --
      cycle;
      \path[fill=black] (0.7, 14.8).. controls (0.4, 14.3) and (0.4, 13.9)
      .. (0.7, 13.4).. controls (1.1, 12.7) and (1.4, 12.7) .. (69.2,
      12.8).. controls (133.6, 12.9) and (137.3, 12.9) .. (137.7,
      13.5).. controls (138.0, 13.9) and (138.0, 14.3) .. (137.7,
      14.7).. controls (137.3, 15.3) and (133.6, 15.3) .. (69.2,
      15.5).. controls (1.4, 15.5) and (1.1, 15.5) .. (0.7, 14.8) -- cycle;
      \path[fill=black] (0.4, 8.5).. controls (0.1, 8.0) and (0.1, 7.5)
      .. (0.4, 7.1).. controls (0.7, 6.4) and (1.0, 6.4) .. (68.9,
      6.4).. controls (133.3, 6.5) and (137.0, 6.6) .. (137.4, 7.2).. controls
      (137.7, 7.6) and (137.7, 7.9) .. (137.4, 8.4).. controls (137.0, 9.0)
      and (133.3, 9.0) .. (68.9, 9.1).. controls (1.0, 9.2) and (0.7, 9.2)
      .. (0.4, 8.5) -- cycle;
      \path[fill=black] (1.1, 2.8).. controls (0.5, 2.1) and (0.6, 0.8)
      .. (1.4, 0.4).. controls (2.3, -0.1) and (136.3, -0.1) .. (137.3,
      0.4).. controls (138.0, 0.8) and (138.1, 2.1) .. (137.5, 2.8).. controls
      (136.9, 3.4) and (1.7, 3.4) .. (1.1, 2.8) -- cycle;
    \end{scope}
  \end{tikzpicture}
}

\def\SngSmooth{%
  \begin{tikzpicture}[y=1cm, x=1cm,
    yscale=0.8*\globalscale,xscale=0.8*\globalscale,%
    baseline={([yshift=-1mm]current bounding box.center)}]
    \begin{scope}[fill=black,cm={1,0,0,1,(0, 0)}]
      \path[fill=black] (65.0, 140.4).. controls (52.8, 138.9) and (40.6,
      132.2) .. (29.6, 121.0).. controls (15.0, 106.2) and (4.6, 85.8)
      .. (1.2, 65.4).. controls (0.3, 60.0) and (0.3, 48.1) .. (1.2,
      43.6).. controls (3.1, 33.9) and (7.1, 26.7) .. (14.0, 20.0).. controls
      (33.8, 1.3) and (74.6, -5.4) .. (105.2, 5.1).. controls (113.4, 8.0) and
      (120.5, 12.3) .. (126.0, 17.6).. controls (132.4, 24.0) and (135.9,
      30.8) .. (138.0, 41.3).. controls (138.7, 44.9) and (138.8, 46.3)
      .. (138.8, 53.8).. controls (138.8, 61.3) and (138.7, 62.8) .. (137.8,
      67.4).. controls (132.0, 99.8) and (110.1, 129.1) .. (85.3,
      137.9).. controls (79.0, 140.1) and (70.8, 141.1) .. (65.0, 140.4) --
      cycle(74.6, 136.1).. controls (93.6, 133.8) and (112.9, 117.4)
      .. (124.8, 93.3).. controls (135.3, 72.2) and (137.7, 49.6) .. (131.2,
      32.8).. controls (127.5, 23.2) and (118.6, 15.0) .. (106.6,
      10.2).. controls (100.7, 7.8) and (92.9, 5.9) .. (85.2, 4.9).. controls
      (80.0, 4.2) and (66.5, 4.3) .. (60.9, 5.1).. controls (42.2, 7.5) and
      (26.8, 13.8) .. (17.0, 23.0).. controls (13.3, 26.6) and (11.1, 29.4)
      .. (9.1, 33.4).. controls (5.6, 40.1) and (4.4, 46.6) .. (4.7,
      56.1).. controls (5.0, 68.2) and (8.0, 79.1) .. (14.3, 92.0).. controls
      (20.3, 104.0) and (29.0, 115.5) .. (37.7, 122.8).. controls (50.1,
      133.2) and (62.4, 137.6) .. (74.6, 136.1) -- cycle;
    \end{scope}
  \end{tikzpicture}
}

\def\SngAone{%
  \begin{tikzpicture}[y=1cm, x=1cm,
    yscale=0.8*\globalscale,xscale=0.8*\globalscale,%
    baseline={([yshift=-1mm]current bounding box.center)}]
    \begin{scope}[fill=black,cm={1,0,0,1,(0, 0)}]
      \path[fill=black] (66.0, 140.5).. controls (57.0, 139.8) and (52.5,
      138.0) .. (50.8, 134.7).. controls (48.5, 129.8) and (51.9, 123.5)
      .. (62.6, 113.5).. controls (64.8, 111.5) and (66.5, 109.7) .. (66.5,
      109.5).. controls (66.5, 109.3) and (61.0, 104.5) .. (54.3,
      98.7).. controls (39.7, 86.3) and (37.9, 84.7) .. (30.2,
      77.6).. controls (13.5, 62.1) and (5.3, 50.2) .. (4.0, 39.8).. controls
      (3.3, 33.7) and (6.7, 27.0) .. (14.0, 20.0).. controls (33.8, 1.3) and
      (74.6, -5.4) .. (105.2, 5.1).. controls (113.4, 8.0) and (120.5, 12.3)
      .. (126.0, 17.6).. controls (130.6, 22.2) and (134.2, 27.8) .. (135.5,
      32.5).. controls (136.2, 34.9) and (136.0, 40.6) .. (135.2,
      43.9).. controls (132.9, 52.1) and (126.2, 61.8) .. (114.7,
      73.4).. controls (107.7, 80.3) and (101.7, 85.7) .. (86.0, 98.9) --
      (73.4, 109.5) -- (78.2, 114.1).. controls (80.8, 116.6) and (84.0,
      120.0) .. (85.4, 121.7).. controls (93.2, 131.3) and (91.7, 137.7)
      .. (81.1, 139.8).. controls (78.5, 140.4) and (69.1, 140.8) .. (66.0,
      140.5) -- cycle(80.5, 135.6).. controls (86.7, 134.3) and (87.5, 132.2)
      .. (83.9, 126.8).. controls (82.1, 124.1) and (71.0, 112.5) .. (70.3,
      112.5).. controls (69.7, 112.5) and (62.4, 119.4) .. (59.5,
      122.7).. controls (56.2, 126.5) and (54.6, 129.0) .. (54.4,
      131.0).. controls (54.3, 132.4) and (54.4, 132.7) .. (55.2,
      133.5).. controls (56.5, 134.7) and (60.0, 135.7) .. (64.0,
      136.1).. controls (68.2, 136.6) and (77.5, 136.3) .. (80.5, 135.6) --
      cycle(83.2, 95.7).. controls (99.1, 82.4) and (104.6, 77.5) .. (111.7,
      70.4).. controls (122.2, 59.9) and (128.1, 51.7) .. (130.8,
      43.9).. controls (131.5, 41.6) and (131.7, 40.5) .. (131.7,
      37.6).. controls (131.7, 34.5) and (131.6, 33.7) .. (130.8,
      31.8).. controls (127.0, 22.7) and (118.2, 14.9) .. (106.6,
      10.2).. controls (100.7, 7.8) and (92.9, 5.9) .. (85.2, 4.9).. controls
      (80.0, 4.2) and (66.5, 4.3) .. (60.9, 5.1).. controls (38.9, 7.9) and
      (21.3, 16.3) .. (12.3, 28.2).. controls (7.4, 34.7) and (7.0, 39.7)
      .. (10.7, 47.5).. controls (16.3, 59.0) and (27.0, 70.1) .. (57.2,
      95.6).. controls (62.3, 99.9) and (67.3, 104.2) .. (68.2,
      105.0).. controls (69.2, 105.9) and (70.1, 106.5) .. (70.2,
      106.5).. controls (70.4, 106.4) and (76.2, 101.6) .. (83.2, 95.7) --
      cycle;
    \end{scope}
  \end{tikzpicture}
}

\def\SngAoneptwo{%
  \begin{tikzpicture}[y=1cm, x=1cm,
    yscale=0.8*\globalscale,xscale=0.8*\globalscale,%
    baseline={([yshift=-1mm]current bounding box.center)}]
    \begin{scope}[fill=black,cm={1,0,0,1,(0, 0)}]
      \path[fill=black] (62.1, 123.8).. controls (51.3, 122.5) and (44.1,
      116.3) .. (41.4, 105.9).. controls (40.6, 102.9) and (40.5, 101.8)
      .. (40.5, 96.5).. controls (40.4, 90.1) and (40.9, 86.4) .. (42.5, 79.8)
      -- (43.5, 76.1) -- (42.1, 75.0).. controls (32.5, 67.1) and (16.2, 51.3)
      .. (11.2, 45.2).. controls (5.8, 38.3) and (2.5, 32.6) .. (1.0,
      27.3).. controls (0.4, 25.1) and (0.3, 23.9) .. (0.4, 20.8).. controls
      (0.5, 17.4) and (0.6, 16.9) .. (1.8, 14.4).. controls (6.7, 4.2) and
      (19.3, 1.2) .. (40.4, 5.3).. controls (49.0, 7.0) and (60.6, 10.3)
      .. (76.6, 15.6) -- (79.0, 16.4) -- (81.6, 14.1).. controls (89.6, 6.9)
      and (96.6, 2.8) .. (103.4, 1.0).. controls (106.9, 0.1) and (112.0,
      -0.1) .. (115.0, 0.7).. controls (121.8, 2.3) and (127.4, 7.2)
      .. (132.3, 15.7).. controls (135.2, 20.8) and (136.3, 23.7) .. (136.4,
      27.0).. controls (136.6, 30.7) and (135.8, 32.2) .. (132.9,
      33.6).. controls (131.1, 34.5) and (130.6, 34.6) .. (127.5,
      34.6).. controls (120.9, 34.6) and (112.9, 32.5) .. (92.6,
      25.6).. controls (86.0, 23.4) and (80.5, 21.5) .. (80.2,
      21.5).. controls (80.0, 21.5) and (78.4, 23.1) .. (76.7,
      25.1).. controls (65.5, 37.9) and (56.1, 53.6) .. (50.1,
      69.6).. controls (49.1, 72.2) and (48.3, 74.5) .. (48.3,
      74.5).. controls (48.3, 74.6) and (51.9, 77.8) .. (56.4,
      81.5).. controls (74.8, 97.2) and (82.2, 104.6) .. (85.4,
      110.8).. controls (86.1, 112.0) and (86.3, 112.9) .. (86.3,
      115.2).. controls (86.3, 117.8) and (86.1, 118.2) .. (85.3,
      119.3).. controls (84.0, 121.0) and (80.7, 122.6) .. (77.2,
      123.3).. controls (73.8, 123.9) and (65.5, 124.2) .. (62.1, 123.8) --
      cycle(77.3, 118.9).. controls (83.0, 117.5) and (83.6, 115.9) .. (80.4,
      110.5).. controls (78.9, 108.1) and (72.0, 100.9) .. (66.0,
      95.5).. controls (58.8, 89.1) and (47.2, 79.3) .. (47.0,
      79.5).. controls (46.6, 79.9) and (45.2, 87.5) .. (44.8,
      91.4).. controls (44.1, 98.4) and (44.9, 104.8) .. (47.1,
      109.5).. controls (48.4, 112.2) and (51.7, 115.7) .. (54.2,
      117.1).. controls (56.5, 118.3) and (59.5, 119.2) .. (62.4,
      119.6).. controls (65.7, 120.1) and (74.3, 119.7) .. (77.3, 118.9) --
      cycle(46.3, 67.7).. controls (48.5, 61.7) and (52.5, 53.2) .. (56.4,
      46.4).. controls (60.4, 39.2) and (68.1, 28.4) .. (73.2,
      22.6).. controls (74.5, 21.2) and (75.4, 19.9) .. (75.2,
      19.7).. controls (74.4, 18.9) and (51.5, 12.0) .. (44.4,
      10.4).. controls (33.0, 7.9) and (23.1, 7.1) .. (17.3, 8.3).. controls
      (8.2, 10.1) and (3.6, 15.7) .. (4.6, 23.7).. controls (5.6, 31.4) and
      (11.6, 40.3) .. (24.7, 53.3).. controls (32.8, 61.4) and (44.4, 71.8)
      .. (44.8, 71.4).. controls (45.0, 71.2) and (45.6, 69.6) .. (46.3, 67.7)
      -- cycle(131.2, 29.9).. controls (132.0, 29.5) and (132.1, 29.2)
      .. (132.1, 27.2).. controls (132.1, 25.2) and (131.9, 24.4) .. (130.4,
      21.4).. controls (129.5, 19.5) and (127.8, 16.5) .. (126.8,
      14.9).. controls (117.6, 0.8) and (104.2, 0.8) .. (87.0,
      15.0).. controls (85.1, 16.5) and (83.7, 17.9) .. (83.9,
      18.0).. controls (84.2, 18.3) and (97.8, 23.0) .. (105.7,
      25.5).. controls (119.8, 30.1) and (127.9, 31.5) .. (131.2, 29.9) --
      cycle;
    \end{scope}
  \end{tikzpicture}
}

\def\SngAonepthree{%
  \begin{tikzpicture}[y=1cm, x=1cm,
    yscale=0.8*\globalscale,xscale=0.8*\globalscale,%
    baseline={([yshift=-1mm]current bounding box.center)}]
    \begin{scope}[fill=black,cm={1,0,0,1,(0, 0)}]
      \path[fill=black] (64.2, 124.5).. controls (57.5, 123.7) and (53.2,
      121.7) .. (49.2, 117.7).. controls (47.0, 115.5) and (46.0, 114.2)
      .. (45.0, 112.1).. controls (40.6, 102.8) and (41.1, 89.1) .. (46.6,
      72.7).. controls (47.7, 69.3) and (49.3, 65.0) .. (50.1,
      63.1).. controls (50.9, 61.3) and (51.6, 59.6) .. (51.7,
      59.4).. controls (51.9, 59.1) and (51.2, 58.9) .. (49.7,
      58.7).. controls (41.5, 57.9) and (30.8, 55.4) .. (24.4,
      53.1).. controls (1.8, 44.6) and (-5.4, 30.3) .. (4.4, 13.1).. controls
      (13.8, -3.3) and (29.6, -4.1) .. (48.2, 10.8).. controls (52.4, 14.2)
      and (60.5, 22.5) .. (64.4, 27.6).. controls (66.1, 29.8) and (67.7,
      31.7) .. (67.8, 31.7).. controls (67.9, 31.8) and (69.4, 30.0) .. (71.1,
      27.8).. controls (75.0, 22.8) and (82.8, 14.8) .. (87.0,
      11.5).. controls (92.7, 6.8) and (98.1, 3.7) .. (103.7, 1.9).. controls
      (106.4, 1.1) and (107.2, 0.9) .. (111.5, 0.9).. controls (116.9, 0.9)
      and (119.2, 1.4) .. (123.1, 3.8).. controls (128.8, 7.3) and (134.5,
      15.2) .. (136.4, 22.2).. controls (137.3, 25.6) and (137.2, 30.2)
      .. (136.2, 33.4).. controls (132.4, 45.7) and (114.6, 55.1) .. (88.9,
      58.4).. controls (86.3, 58.7) and (84.1, 59.1) .. (84.0,
      59.1).. controls (84.0, 59.2) and (84.4, 60.3) .. (84.9,
      61.6).. controls (94.5, 83.4) and (96.7, 101.6) .. (91.1,
      112.9).. controls (89.2, 116.7) and (85.8, 120.1) .. (81.8,
      122.0).. controls (80.2, 122.8) and (77.8, 123.6) .. (76.4,
      124.0).. controls (73.3, 124.6) and (67.2, 124.9) .. (64.2, 124.5) --
      cycle(75.8, 119.8).. controls (82.5, 118.1) and (87.1, 113.6) .. (88.9,
      106.6).. controls (90.9, 98.9) and (89.9, 88.3) .. (85.8,
      75.7).. controls (84.2, 70.8) and (79.9, 60.1) .. (79.4,
      59.8).. controls (79.2, 59.7) and (73.9, 59.6) .. (67.7, 59.6) -- (56.3,
      59.5) -- (54.4, 63.7).. controls (50.6, 72.5) and (47.5, 83.3) .. (46.5,
      91.2).. controls (46.0, 95.3) and (46.3, 102.2) .. (47.1,
      105.4).. controls (48.8, 112.4) and (53.6, 117.5) .. (60.0,
      119.4).. controls (64.6, 120.7) and (71.2, 120.9) .. (75.8, 119.8) --
      cycle(74.2, 55.4) -- (77.3, 55.4) -- (75.6, 52.3).. controls (73.7,
      48.8) and (70.7, 43.8) .. (69.0, 41.2) -- (67.8, 39.4) -- (66.4,
      41.5).. controls (63.9, 45.3) and (58.6, 54.5) .. (58.6,
      55.2).. controls (58.6, 55.4) and (67.3, 55.7) .. (69.5,
      55.5).. controls (70.3, 55.4) and (72.4, 55.4) .. (74.2, 55.4) --
      cycle(54.8, 53.0).. controls (56.4, 49.8) and (59.8, 43.9) .. (62.7,
      39.6) -- (65.2, 35.6) -- (63.2, 32.9).. controls (57.4, 25.3) and (49.5,
      17.0) .. (43.8, 12.7).. controls (33.0, 4.6) and (23.7, 2.4) .. (16.5,
      6.2).. controls (10.9, 9.2) and (5.2, 18.5) .. (4.4, 25.6).. controls
      (2.9, 39.5) and (19.6, 50.5) .. (48.3, 54.3).. controls (50.9, 54.7) and
      (53.3, 55.0) .. (53.4, 55.0).. controls (53.7, 55.0) and (54.3, 54.1)
      .. (54.8, 53.0) -- cycle(87.6, 54.3).. controls (114.0, 50.9) and
      (130.7, 41.9) .. (132.8, 29.8).. controls (134.1, 22.2) and (127.4,
      10.6) .. (119.3, 6.6).. controls (116.7, 5.3) and (116.2, 5.2)
      .. (112.9, 5.0).. controls (106.1, 4.7) and (100.2, 7.0) .. (91.7,
      13.2).. controls (86.3, 17.1) and (76.4, 27.4) .. (71.3, 34.4) -- (70.5,
      35.6) -- (73.8, 40.8).. controls (75.6, 43.7) and (78.1, 48.0) .. (79.5,
      50.5).. controls (80.8, 53.0) and (82.0, 55.0) .. (82.1,
      55.0).. controls (82.1, 55.0) and (84.6, 54.7) .. (87.6, 54.3) -- cycle;
    \end{scope}
  \end{tikzpicture}
}

\def\SngAtwo{%
  \begin{tikzpicture}[y=1cm, x=1cm,
    yscale=0.8*\globalscale,xscale=0.8*\globalscale,%
    baseline={([yshift=-1mm]current bounding box.center)}]
    \begin{scope}[fill=black,cm={1,0,0,1,(0, 0)}]
      \path[fill=black] (68.2, 144.0).. controls (67.6, 143.5) and (67.5,
      142.6) .. (67.3, 139.8).. controls (66.4, 126.2) and (59.4, 113.4)
      .. (44.8, 98.8).. controls (42.5, 96.6) and (37.0, 91.5) .. (32.5,
      87.4).. controls (19.9, 76.1) and (16.1, 72.5) .. (12.6,
      68.7).. controls (6.9, 62.2) and (3.8, 56.8) .. (2.4, 51.3).. controls
      (-0.3, 39.8) and (6.6, 27.3) .. (20.8, 17.9).. controls (31.2, 11.1) and
      (44.0, 6.8) .. (59.8, 4.7).. controls (66.2, 3.9) and (79.7, 3.9)
      .. (85.9, 4.7).. controls (109.5, 7.9) and (125.8, 17.4) .. (133.0,
      32.3).. controls (135.3, 36.9) and (136.2, 40.7) .. (136.2,
      45.5).. controls (136.2, 50.4) and (135.3, 53.7) .. (132.9,
      58.7).. controls (129.1, 66.3) and (123.8, 72.2) .. (108.7,
      85.7).. controls (97.3, 95.8) and (88.4, 104.6) .. (84.5,
      109.5).. controls (76.3, 119.9) and (72.4, 129.0) .. (71.8,
      139.7).. controls (71.5, 142.6) and (71.4, 143.5) .. (70.8,
      144.0).. controls (70.0, 144.8) and (69.0, 144.8) .. (68.2, 144.0) --
      cycle(70.7, 124.5).. controls (72.0, 121.2) and (74.6, 116.1) .. (77.0,
      112.6).. controls (82.2, 104.8) and (90.4, 96.2) .. (105.7,
      82.7).. controls (115.7, 73.7) and (120.9, 68.6) .. (124.5,
      64.1).. controls (133.8, 52.3) and (134.5, 41.2) .. (126.6,
      29.6).. controls (122.6, 23.7) and (115.6, 18.2) .. (107.5,
      14.7).. controls (76.4, 1.3) and (29.6, 9.4) .. (12.2, 31.1).. controls
      (5.5, 39.5) and (4.3, 47.2) .. (8.5, 55.6).. controls (11.6, 61.9) and
      (17.8, 68.7) .. (32.5, 81.7).. controls (42.8, 90.8) and (50.8, 98.6)
      .. (55.3, 103.9).. controls (61.1, 110.7) and (66.1, 118.9) .. (68.5,
      125.4).. controls (68.9, 126.6) and (69.4, 127.6) .. (69.5,
      127.4).. controls (69.6, 127.3) and (70.2, 126.0) .. (70.7, 124.5) --
      cycle;
    \end{scope}
  \end{tikzpicture}
}

\def\SngAtwoptwo{%
  \begin{tikzpicture}[y=1cm, x=1cm,
    yscale=0.8*\globalscale,xscale=0.8*\globalscale,%
    baseline={([yshift=-1mm]current bounding box.center)}]
    \begin{scope}[fill=black,cm={1,0,0,1,(0, 0)}]
      \path[fill=black] (72.8, 128.5).. controls (72.0, 128.0) and (71.6,
      126.6) .. (71.6, 124.8).. controls (71.6, 117.2) and (67.9, 106.8)
      .. (62.2, 98.4).. controls (56.6, 90.1) and (50.2, 83.4) .. (34.9,
      69.9).. controls (15.6, 52.6) and (9.9, 45.8) .. (7.0, 36.7).. controls
      (5.8, 32.7) and (5.8, 27.3) .. (7.2, 23.2).. controls (9.4, 16.5) and
      (14.0, 10.6) .. (19.6, 7.5).. controls (23.3, 5.5) and (26.6, 4.8)
      .. (32.5, 4.8).. controls (40.4, 4.8) and (45.5, 6.0) .. (66.6,
      12.9).. controls (86.8, 19.4) and (95.5, 21.3) .. (106.9,
      21.8).. controls (116.0, 22.1) and (122.1, 21.2) .. (130.0,
      18.5).. controls (136.1, 16.4) and (137.0, 16.2) .. (137.8,
      17.4).. controls (138.9, 19.1) and (138.2, 19.8) .. (132.2,
      22.7).. controls (103.3, 36.6) and (87.1, 55.0) .. (80.3,
      81.7).. controls (77.4, 93.0) and (75.8, 107.5) .. (75.8,
      122.8).. controls (75.8, 127.4) and (75.8, 127.7) .. (75.1,
      128.2).. controls (74.2, 128.8) and (73.4, 128.9) .. (72.8, 128.5) --
      cycle(73.6, 94.0).. controls (76.2, 75.5) and (82.1, 60.8) .. (91.7,
      48.3).. controls (96.8, 41.7) and (105.5, 33.7) .. (113.2,
      28.6).. controls (115.1, 27.4) and (116.7, 26.2) .. (116.9,
      26.1).. controls (117.1, 26.0) and (115.0, 25.9) .. (112.3,
      26.0).. controls (106.5, 26.2) and (99.7, 25.7) .. (93.1,
      24.5).. controls (85.1, 23.0) and (78.4, 21.1) .. (63.3,
      16.2).. controls (46.7, 10.8) and (40.4, 9.2) .. (33.7, 8.9).. controls
      (24.2, 8.5) and (17.7, 12.0) .. (13.2, 20.0).. controls (9.5, 26.6) and
      (9.3, 32.8) .. (12.7, 39.7).. controls (15.9, 46.0) and (22.0, 52.6)
      .. (36.2, 65.3).. controls (48.2, 76.0) and (51.9, 79.5) .. (56.4,
      84.5).. controls (63.0, 91.8) and (66.9, 97.2) .. (70.6, 104.4) --
      (72.1, 107.4) -- (72.6, 102.7).. controls (72.8, 100.0) and (73.2, 96.2)
      .. (73.6, 94.0) -- cycle;
    \end{scope}
  \end{tikzpicture}
}

\def\SngAtwopthree{%
  \begin{tikzpicture}[y=1cm, x=1cm,
    yscale=0.8*\globalscale,xscale=0.8*\globalscale,%
    baseline={([yshift=-1mm]current bounding box.center)}]
    \begin{scope}[fill=black,cm={1,0,0,1,(0, 0)}]
      \path[fill=black] (65.9, 114.4).. controls (65.5, 113.9) and (65.3,
      112.5) .. (65.1, 108.4).. controls (63.4, 67.3) and (43.3, 32.8)
      .. (7.6, 9.3).. controls (5.4, 7.9) and (3.4, 6.3) .. (3.1,
      6.0).. controls (2.9, 5.5) and (2.9, 5.0) .. (3.2, 4.3).. controls (4.0,
      2.8) and (5.3, 2.9) .. (8.4, 4.7).. controls (31.5, 18.0) and (55.6,
      22.0) .. (84.1, 17.3).. controls (96.3, 15.2) and (112.2, 10.3)
      .. (126.3, 4.1).. controls (129.3, 2.8) and (129.6, 2.7) .. (130.5,
      3.2).. controls (131.5, 3.8) and (131.9, 5.4) .. (131.2, 6.2).. controls
      (131.0, 6.4) and (128.0, 8.0) .. (124.4, 9.8).. controls (103.5, 20.2)
      and (90.3, 32.0) .. (81.3, 48.5).. controls (73.6, 62.6) and (69.5,
      83.0) .. (69.5, 107.2).. controls (69.5, 112.3) and (69.4, 113.6)
      .. (68.9, 114.2).. controls (68.3, 115.2) and (66.7, 115.3) .. (65.9,
      114.4) -- cycle(69.0, 70.4).. controls (73.9, 48.5) and (84.7, 31.4)
      .. (101.7, 18.6).. controls (102.7, 17.9) and (103.3, 17.3) .. (103.2,
      17.3).. controls (103.0, 17.3) and (100.1, 18.0) .. (96.7,
      18.8).. controls (84.1, 22.0) and (74.2, 23.3) .. (62.3,
      23.3).. controls (54.2, 23.3) and (49.9, 23.0) .. (43.2,
      21.8).. controls (38.8, 21.1) and (31.3, 19.3) .. (27.8,
      18.0).. controls (26.6, 17.6) and (25.8, 17.4) .. (25.8,
      17.6).. controls (25.9, 17.7) and (27.4, 19.3) .. (29.3,
      21.0).. controls (41.5, 32.3) and (51.6, 46.2) .. (58.3,
      60.9).. controls (61.5, 67.8) and (63.4, 73.4) .. (66.0, 83.3) -- (66.4,
      84.8) -- (67.2, 79.7).. controls (67.7, 76.8) and (68.5, 72.7) .. (69.0,
      70.4) -- cycle;
    \end{scope}
  \end{tikzpicture}
}

\def\SngAoneAtwo{%
  \begin{tikzpicture}[y=1cm, x=1cm,
    yscale=0.8*\globalscale,xscale=0.8*\globalscale,%
    baseline={([yshift=-1mm]current bounding box.center)}]
    \begin{scope}[fill=black,cm={1,0,0,1,(0, 0)}]
      \path[fill=black] (66.0, 140.5).. controls (51.2, 139.2) and (43.7,
      134.3) .. (42.3, 125.0).. controls (41.3, 118.7) and (44.1, 110.1)
      .. (50.3, 100.1).. controls (51.4, 98.4) and (52.2, 96.9) .. (52.2,
      96.9).. controls (52.1, 96.8) and (49.8, 94.9) .. (47.1,
      92.6).. controls (23.6, 72.7) and (12.0, 60.1) .. (6.8, 49.2).. controls
      (2.6, 40.2) and (3.0, 33.5) .. (8.4, 26.4).. controls (14.6, 18.3) and
      (23.2, 15.5) .. (36.2, 17.3).. controls (42.2, 18.2) and (47.8, 19.8)
      .. (62.6, 24.5).. controls (78.0, 29.5) and (82.2, 30.6) .. (90.3,
      32.3).. controls (96.6, 33.5) and (100.9, 34.0) .. (107.2,
      34.0).. controls (115.6, 34.0) and (122.2, 32.7) .. (128.9,
      29.8).. controls (132.4, 28.3) and (133.5, 28.3) .. (134.3,
      29.5).. controls (135.2, 30.9) and (134.5, 31.8) .. (131.0,
      33.7).. controls (108.1, 46.1) and (83.4, 66.2) .. (65.4,
      87.1).. controls (62.0, 91.1) and (58.2, 95.9) .. (58.2,
      96.3).. controls (58.2, 96.4) and (60.5, 98.5) .. (63.4,
      100.9).. controls (75.1, 110.8) and (82.7, 118.1) .. (86.4,
      122.9).. controls (93.1, 131.8) and (91.2, 137.8) .. (81.1,
      139.8).. controls (78.4, 140.4) and (69.0, 140.8) .. (66.0, 140.5) --
      cycle(80.5, 135.6).. controls (83.5, 135.0) and (85.7, 133.8) .. (85.9,
      132.8).. controls (86.6, 130.0) and (83.9, 125.8) .. (76.6,
      118.5).. controls (72.3, 114.3) and (60.2, 103.5) .. (56.1,
      100.3).. controls (55.5, 99.9) and (55.2, 100.3) .. (52.9,
      104.1).. controls (48.4, 111.7) and (45.9, 119.2) .. (46.4,
      123.8).. controls (46.6, 126.4) and (48.4, 129.9) .. (50.0,
      131.2).. controls (53.0, 133.7) and (58.1, 135.5) .. (64.2,
      136.1).. controls (68.2, 136.6) and (77.5, 136.3) .. (80.5, 135.6) --
      cycle(57.4, 90.1).. controls (70.0, 73.9) and (92.1, 53.7) .. (112.1,
      40.3).. controls (113.9, 39.0) and (115.4, 38.0) .. (115.3,
      37.9).. controls (115.2, 37.9) and (113.3, 38.0) .. (111.1,
      38.1).. controls (107.1, 38.5) and (99.6, 38.1) .. (94.9,
      37.4).. controls (85.5, 35.9) and (78.1, 34.0) .. (57.8,
      27.4).. controls (42.8, 22.6) and (36.0, 21.1) .. (29.1,
      21.1).. controls (24.2, 21.1) and (21.5, 21.6) .. (18.2,
      23.4).. controls (14.5, 25.4) and (11.0, 29.2) .. (9.1, 33.3).. controls
      (5.6, 41.3) and (11.9, 53.5) .. (27.7, 69.3).. controls (34.6, 76.2) and
      (54.1, 93.6) .. (54.7, 93.4).. controls (54.9, 93.3) and (56.1, 91.9)
      .. (57.4, 90.1) -- cycle;
    \end{scope}
  \end{tikzpicture}
}

\def\SngAoneptwoAtwo{%
  \begin{tikzpicture}[y=1cm, x=1cm,
    yscale=0.8*\globalscale,xscale=0.8*\globalscale,%
    baseline={([yshift=-1mm]current bounding box.center)}]
    \begin{scope}[fill=black,cm={1,0,0,1,(0, 0)}]
      \path[fill=black] (65.1, 129.1).. controls (53.0, 128.2) and (45.0,
      121.5) .. (42.7, 110.1).. controls (41.9, 106.2) and (41.7, 98.6)
      .. (42.3, 94.0).. controls (43.6, 84.9) and (46.7, 74.8) .. (51.5,
      64.1).. controls (53.5, 59.7) and (53.6, 59.5) .. (53.0,
      58.8).. controls (52.0, 57.6) and (41.7, 48.7) .. (36.7,
      44.6).. controls (27.1, 36.8) and (16.6, 29.0) .. (8.5, 23.7).. controls
      (6.2, 22.3) and (4.3, 20.9) .. (4.2, 20.7).. controls (3.4, 19.5) and
      (4.4, 17.3) .. (5.7, 17.3).. controls (6.0, 17.3) and (8.1, 18.4)
      .. (10.4, 19.8).. controls (22.8, 27.2) and (38.3, 34.1) .. (55.3,
      39.7).. controls (59.7, 41.2) and (63.3, 42.3) .. (63.4,
      42.3).. controls (63.6, 42.3) and (64.8, 40.7) .. (66.2,
      38.7).. controls (82.0, 16.3) and (100.5, 3.6) .. (114.4,
      5.6).. controls (118.1, 6.2) and (120.5, 7.1) .. (123.6, 9.2).. controls
      (128.2, 12.3) and (132.0, 17.1) .. (135.7, 24.8).. controls (138.3,
      30.2) and (139.1, 33.7) .. (138.9, 38.1).. controls (138.7, 43.5) and
      (136.9, 46.7) .. (132.4, 49.7).. controls (127.3, 53.1) and (121.1,
      54.5) .. (110.2, 54.5).. controls (99.0, 54.5) and (87.3, 52.8)
      .. (72.5, 49.1).. controls (67.2, 47.7) and (65.4, 47.4) .. (65.2,
      47.7).. controls (65.1, 48.0) and (63.6, 50.4) .. (61.9, 53.2) -- (58.9,
      58.3) -- (65.5, 65.1).. controls (72.8, 72.4) and (76.3, 76.3) .. (80.6,
      81.7).. controls (88.4, 91.4) and (93.1, 99.4) .. (95.1,
      106.5).. controls (95.9, 109.3) and (96.0, 110.2) .. (95.9,
      113.3).. controls (95.7, 117.4) and (95.0, 119.3) .. (93.0,
      121.9).. controls (88.7, 127.2) and (77.8, 130.0) .. (65.1, 129.1) --
      cycle(73.6, 124.9).. controls (86.1, 123.9) and (91.9, 120.1) .. (91.8,
      112.7).. controls (91.8, 109.3) and (90.9, 106.4) .. (88.6,
      101.4).. controls (84.1, 92.1) and (76.2, 82.0) .. (63.4,
      68.9).. controls (57.9, 63.3) and (57.0, 62.5) .. (56.7,
      63.0).. controls (55.7, 64.7) and (52.7, 71.9) .. (51.3,
      75.7).. controls (47.3, 86.9) and (45.5, 97.4) .. (46.2,
      105.3).. controls (47.3, 117.2) and (53.5, 123.7) .. (64.7,
      124.9).. controls (68.4, 125.3) and (68.6, 125.3) .. (73.6, 124.9) --
      cycle(58.5, 50.8).. controls (59.9, 48.2) and (61.0, 46.1) .. (61.0,
      46.0).. controls (60.9, 46.0) and (59.5, 45.5) .. (57.9,
      45.0).. controls (52.5, 43.3) and (43.7, 40.2) .. (38.3,
      38.1).. controls (35.5, 36.9) and (33.1, 36.0) .. (32.9,
      36.0).. controls (32.8, 36.0) and (34.5, 37.4) .. (36.6,
      39.1).. controls (40.3, 42.0) and (50.1, 50.4) .. (53.8,
      53.8).. controls (54.8, 54.6) and (55.6, 55.4) .. (55.7,
      55.4).. controls (55.8, 55.4) and (57.1, 53.3) .. (58.5, 50.8) --
      cycle(121.9, 49.4).. controls (128.3, 48.0) and (133.0, 45.0) .. (134.2,
      41.4).. controls (134.9, 39.3) and (134.9, 35.3) .. (134.2,
      32.7).. controls (132.8, 27.1) and (128.5, 19.5) .. (124.5,
      15.5).. controls (117.7, 8.6) and (110.0, 7.8) .. (99.7,
      12.9).. controls (91.8, 16.7) and (84.3, 23.1) .. (75.6,
      33.3).. controls (73.5, 35.8) and (68.1, 43.1) .. (68.1,
      43.5).. controls (68.1, 43.7) and (77.9, 46.1) .. (83.3,
      47.2).. controls (89.7, 48.5) and (97.2, 49.6) .. (104.1,
      50.2).. controls (108.2, 50.6) and (118.4, 50.1) .. (121.9, 49.4) --
      cycle;
    \end{scope}
  \end{tikzpicture}
}

\def\SngAoneAtwoptwo{%
  \begin{tikzpicture}[y=1cm, x=1cm,
    yscale=0.8*\globalscale,xscale=0.8*\globalscale,%
    baseline={([yshift=-1mm]current bounding box.center)}]
    \begin{scope}[fill=black,cm={1,0,0,1,(0, 0)}]
      \path[fill=black] (60.8, 132.1).. controls (49.6, 130.4) and (43.2,
      125.3) .. (41.9, 117.2).. controls (40.4, 107.2) and (48.0, 92.1)
      .. (63.6, 74.5) -- (66.8, 70.8) -- (63.3, 67.3).. controls (49.7, 53.6)
      and (27.9, 36.7) .. (12.3, 27.9).. controls (5.8, 24.2) and (5.8, 24.3)
      .. (5.7, 23.1).. controls (5.6, 22.0) and (6.5, 20.8) .. (7.4,
      20.8).. controls (7.7, 20.8) and (10.8, 22.2) .. (14.1, 23.7).. controls
      (33.1, 32.8) and (47.2, 37.4) .. (61.0, 38.8).. controls (66.0, 39.4)
      and (77.5, 39.1) .. (82.4, 38.3).. controls (87.7, 37.5) and (96.8,
      35.1) .. (102.8, 33.0).. controls (109.9, 30.6) and (114.6, 28.6)
      .. (124.3, 24.1).. controls (132.0, 20.6) and (132.4, 20.4) .. (133.2,
      21.0).. controls (134.2, 21.6) and (134.5, 22.7) .. (134.1,
      23.6).. controls (133.9, 24.0) and (132.1, 25.2) .. (130.1,
      26.2).. controls (114.1, 34.7) and (94.9, 49.1) .. (78.0, 65.4) --
      (72.7, 70.6) -- (74.4, 72.6).. controls (89.5, 89.5) and (97.2, 102.2)
      .. (98.5, 112.6).. controls (99.3, 119.5) and (96.5, 125.5) .. (90.7,
      128.8).. controls (88.3, 130.2) and (84.1, 131.5) .. (80.2,
      132.1).. controls (76.0, 132.8) and (64.9, 132.8) .. (60.8, 132.1) --
      cycle(79.2, 128.0).. controls (84.5, 127.3) and (88.3, 125.8) .. (91.0,
      123.3).. controls (93.5, 121.0) and (94.3, 118.8) .. (94.3,
      114.8).. controls (94.3, 112.1) and (94.1, 111.1) .. (93.0,
      107.8).. controls (90.8, 101.1) and (86.3, 93.6) .. (79.6,
      85.3).. controls (75.7, 80.4) and (70.2, 74.1) .. (69.8,
      74.1).. controls (69.7, 74.1) and (68.1, 75.8) .. (66.2,
      77.9).. controls (56.6, 88.9) and (50.2, 98.7) .. (47.4,
      107.1).. controls (46.3, 110.4) and (46.1, 111.4) .. (46.1,
      114.3).. controls (46.0, 117.2) and (46.1, 117.9) .. (47.0,
      119.7).. controls (48.8, 123.5) and (52.5, 126.0) .. (58.4,
      127.4).. controls (63.7, 128.6) and (72.4, 128.9) .. (79.2, 128.0) --
      cycle(82.4, 55.5).. controls (87.3, 51.0) and (95.3, 44.5) .. (100.8,
      40.3).. controls (103.2, 38.5) and (105.1, 36.9) .. (105.1,
      36.9).. controls (105.1, 36.8) and (102.7, 37.5) .. (99.7,
      38.5).. controls (89.3, 41.7) and (81.9, 43.0) .. (72.2,
      43.3).. controls (65.1, 43.5) and (59.1, 43.1) .. (52.9,
      41.9).. controls (48.5, 41.1) and (41.1, 39.2) .. (37.4,
      37.9).. controls (36.0, 37.4) and (34.6, 37.0) .. (34.5,
      37.0).. controls (34.3, 37.0) and (36.4, 38.8) .. (39.2,
      40.9).. controls (48.5, 47.9) and (59.2, 57.2) .. (66.1, 64.2) -- (69.6,
      67.7) -- (74.1, 63.4).. controls (76.5, 61.0) and (80.3, 57.5) .. (82.4,
      55.5) -- cycle;
    \end{scope}
  \end{tikzpicture}
}

\def\SngAthree{%
  \begin{tikzpicture}[y=1cm, x=1cm,
    yscale=\globalscale,xscale=\globalscale,%
    baseline={([yshift=-1mm]current bounding box.center)}]
    \begin{scope}[fill=black,cm={1,0,0,1,(0, 0)}]
      \path[fill=black] (2.2, 72.2).. controls (-0.1, 70.9) and (-0.2, 67.3)
      .. (1.7, 59.5).. controls (6.5, 41.0) and (13.1, 28.7) .. (23.7,
      18.5).. controls (41.3, 1.4) and (66.7, -4.2) .. (90.1, 3.6).. controls
      (111.3, 10.6) and (126.7, 27.1) .. (133.4, 50.0).. controls (135.5,
      57.3) and (136.5, 62.4) .. (136.5, 65.5).. controls (136.5, 68.1) and
      (136.4, 68.5) .. (135.6, 69.5).. controls (134.9, 70.4) and (134.4,
      70.6) .. (133.2, 70.6).. controls (129.7, 70.6) and (126.3, 67.0)
      .. (110.9, 47.3).. controls (96.8, 29.1) and (92.0, 23.3) .. (85.7,
      16.8).. controls (78.4, 9.3) and (72.7, 5.3) .. (69.4, 5.3).. controls
      (66.2, 5.3) and (61.5, 8.4) .. (55.1, 14.6).. controls (48.3, 21.3) and
      (42.7, 27.9) .. (26.5, 48.5).. controls (12.5, 66.3) and (8.7, 70.5)
      .. (5.5, 72.0).. controls (3.8, 72.8) and (3.3, 72.8) .. (2.2, 72.2) --
      cycle(7.5, 64.9).. controls (10.9, 61.3) and (13.9, 57.7) .. (24.7,
      43.9).. controls (41.7, 22.3) and (49.1, 13.8) .. (57.3, 6.7) -- (59.1,
      5.2) -- (57.3, 5.5).. controls (56.4, 5.6) and (54.2, 6.1) .. (52.6,
      6.6).. controls (42.5, 9.5) and (34.4, 14.1) .. (27.0, 21.3).. controls
      (17.7, 30.1) and (12.3, 39.3) .. (7.7, 54.0).. controls (5.2, 61.8) and
      (4.1, 67.4) .. (4.9, 67.4).. controls (5.1, 67.4) and (6.2, 66.3)
      .. (7.5, 64.9) -- cycle(132.5, 64.8).. controls (131.7, 60.3) and
      (131.4, 58.7) .. (130.1, 54.1).. controls (127.5, 44.6) and (124.9,
      38.5) .. (120.4, 31.8).. controls (112.7, 20.1) and (101.3, 11.5)
      .. (88.0, 7.3).. controls (84.0, 6.0) and (80.2, 5.1) .. (79.9,
      5.4).. controls (79.8, 5.5) and (80.8, 6.4) .. (82.0, 7.3).. controls
      (88.3, 12.5) and (97.1, 22.6) .. (112.4, 42.3).. controls (117.8, 49.2)
      and (123.3, 56.2) .. (124.7, 57.9).. controls (127.6, 61.4) and (131.9,
      66.0) .. (132.4, 66.0).. controls (132.5, 66.0) and (132.6, 65.4)
      .. (132.5, 64.8) -- cycle;
    \end{scope}
  \end{tikzpicture}
}

\def\SngAoneAthree{%
  \begin{tikzpicture}[y=1cm, x=1cm,
    yscale=\globalscale,xscale=\globalscale,%
    baseline={([yshift=-1mm]current bounding box.center)}]
    \begin{scope}[fill=black,cm={1,0,0,1,(0, 0)}]
      \path[fill=black] (165.6, 131.7).. controls (162.2, 129.5) and (159.4, 124.1) .. (157.2, 115.4).. controls (156.0, 110.6) and (155.1, 105.5) .. (153.8, 95.3).. controls (152.5, 84.9) and (151.8, 80.1) .. (150.3, 72.9) -- (149.2, 67.5) -- (146.6, 65.8).. controls (142.7, 63.2) and (133.9, 59.1) .. (128.8, 57.4).. controls (116.9, 53.4) and (105.2, 51.9) .. (88.1, 51.9).. controls (57.3, 51.9) and (28.3, 47.6) .. (13.9, 40.9).. controls (4.0, 36.3) and (-0.9, 30.0) .. (0.3, 23.0).. controls (1.1, 18.2) and (5.2, 11.9) .. (10.8, 6.8).. controls (17.7, 0.6) and (23.0, -1.1) .. (28.0, 1.2).. controls (31.3, 2.8) and (33.9, 5.8) .. (40.1, 15.0).. controls (46.7, 25.0) and (50.0, 29.2) .. (54.5, 33.5).. controls (61.5, 40.0) and (69.9, 44.4) .. (79.4, 46.4).. controls (83.3, 47.2) and (90.1, 47.7) .. (93.0, 47.5).. controls (96.6, 47.2) and (99.5, 46.2) .. (108.8, 41.8).. controls (118.1, 37.4) and (121.5, 36.2) .. (125.7, 35.6).. controls (134.7, 34.2) and (142.0, 39.0) .. (147.3, 49.7).. controls (148.9, 53.0) and (150.8, 58.2) .. (152.0, 62.3).. controls (152.8, 64.8) and (152.9, 64.9) .. (155.7, 67.3).. controls (166.3, 76.3) and (173.5, 88.2) .. (176.4, 101.4).. controls (179.4, 115.1) and (177.0, 128.7) .. (171.0, 131.8).. controls (169.7, 132.5) and (166.9, 132.4) .. (165.6, 131.7) -- cycle(170.1, 127.1).. controls (173.2, 123.6) and (174.3, 112.9) .. (172.6, 103.9).. controls (170.4, 92.8) and (165.2, 82.9) .. (157.6, 75.0).. controls (154.6, 71.9) and (154.5, 71.8) .. (154.7, 73.0).. controls (156.1, 81.2) and (157.2, 87.9) .. (158.0, 95.0).. controls (159.1, 103.8) and (160.2, 110.1) .. (161.1, 113.8).. controls (162.7, 120.3) and (165.0, 125.5) .. (167.0, 127.3).. controls (168.4, 128.5) and (168.9, 128.4) .. (170.1, 127.1) -- cycle(146.1, 57.6).. controls (142.7, 48.4) and (138.7, 42.9) .. (133.9, 40.7).. controls (132.0, 39.8) and (131.4, 39.7) .. (128.2, 39.7).. controls (125.4, 39.8) and (124.1, 40.0) .. (122.0, 40.7).. controls (117.5, 42.2) and (104.7, 48.3) .. (105.8, 48.3).. controls (107.1, 48.4) and (116.1, 49.7) .. (118.5, 50.3).. controls (127.1, 52.0) and (136.2, 55.2) .. (143.1, 58.9).. controls (145.2, 60.0) and (147.0, 60.8) .. (147.0, 60.7).. controls (147.1, 60.5) and (146.6, 59.2) .. (146.1, 57.6) -- cycle(63.6, 45.2).. controls (58.1, 42.3) and (53.8, 39.0) .. (49.4, 34.3).. controls (45.4, 30.2) and (42.8, 26.8) .. (36.8, 17.8).. controls (29.5, 6.7) and (27.2, 4.4) .. (23.5, 4.4).. controls (20.0, 4.4) and (13.8, 8.9) .. (9.0, 14.9).. controls (4.9, 20.0) and (3.7, 23.6) .. (4.7, 27.3).. controls (5.3, 29.1) and (5.7, 29.8) .. (7.9, 31.9).. controls (13.3, 37.2) and (24.3, 41.2) .. (41.1, 44.1).. controls (48.9, 45.4) and (60.9, 46.7) .. (65.3, 46.8) -- (66.9, 46.9) -- (63.6, 45.2) -- cycle;
    \end{scope}
  \end{tikzpicture}
}

\def\SngAtwoAthree{%
  \begin{tikzpicture}[y=1cm, x=1cm,
    yscale=\globalscale,xscale=\globalscale,%
    baseline={([yshift=-1mm]current bounding box.center)}]
    \begin{scope}[fill=black,cm={1,0,0,1,(0, 0)}]
      \path[fill=black] (152.3, 150.5).. controls (151.9, 150.1) and (151.7,
      148.8) .. (151.6, 146.5).. controls (151.4, 144.6) and (151.0, 139.5)
      .. (150.6, 135.1).. controls (147.8, 103.3) and (140.2, 80.9) .. (127.9,
      68.2).. controls (122.8, 62.9) and (117.8, 59.6) .. (111.7,
      57.5).. controls (106.0, 55.6) and (103.9, 55.3) .. (91.2,
      55.2).. controls (77.6, 55.0) and (72.1, 54.8) .. (63.3,
      54.1).. controls (38.1, 52.1) and (19.1, 47.4) .. (9.7, 41.0).. controls
      (6.8, 39.1) and (4.0, 36.0) .. (2.9, 33.6).. controls (2.0, 31.8) and
      (1.9, 31.2) .. (1.9, 28.2).. controls (2.0, 25.6) and (2.2, 24.4)
      .. (2.8, 22.8).. controls (4.1, 19.6) and (6.9, 15.7) .. (10.8,
      11.8).. controls (19.5, 2.9) and (26.1, 1.3) .. (32.6, 6.5).. controls
      (34.6, 8.1) and (36.7, 10.9) .. (42.3, 19.3).. controls (48.3, 28.3) and
      (51.8, 32.7) .. (56.2, 36.9).. controls (64.4, 44.6) and (74.3, 49.1)
      .. (86.1, 50.4).. controls (87.9, 50.6) and (97.2, 50.9) .. (106.7,
      51.2).. controls (126.2, 51.6) and (135.7, 52.0) .. (144.5,
      52.7).. controls (161.4, 54.3) and (168.6, 56.8) .. (171.3,
      62.3).. controls (173.2, 66.1) and (172.8, 70.3) .. (169.6,
      79.4).. controls (161.8, 101.8) and (160.1, 107.7) .. (158.3,
      117.8).. controls (156.8, 126.0) and (156.2, 131.6) .. (156.0,
      140.8).. controls (155.8, 150.3) and (155.7, 151.0) .. (153.8,
      151.0).. controls (153.3, 151.0) and (152.6, 150.7) .. (152.3, 150.5) --
      cycle(154.5, 115.0).. controls (156.5, 105.1) and (158.5, 98.3)
      .. (163.6, 84.0).. controls (167.1, 74.1) and (168.3, 70.0) .. (168.3,
      67.4).. controls (168.3, 65.2) and (167.0, 62.9) .. (165.3,
      61.7).. controls (160.2, 58.4) and (148.1, 56.6) .. (125.6, 55.9) --
      (118.7, 55.7) -- (121.3, 57.3).. controls (130.6, 63.1) and (137.7,
      71.8) .. (143.1, 83.6).. controls (145.6, 89.1) and (148.2, 96.8)
      .. (149.9, 103.9).. controls (151.3, 109.4) and (153.1, 118.6)
      .. (153.1, 120.0).. controls (153.1, 121.6) and (153.5, 120.3)
      .. (154.5, 115.0) -- cycle(68.1, 50.0).. controls (68.0, 50.0) and
      (67.0, 49.5) .. (65.8, 48.9).. controls (62.0, 47.0) and (56.7, 43.2)
      .. (53.1, 39.8).. controls (48.5, 35.5) and (45.3, 31.4) .. (38.5,
      21.2).. controls (30.9, 9.9) and (28.7, 7.8) .. (25.0, 7.8).. controls
      (21.4, 7.8) and (15.4, 12.1) .. (10.5, 18.3).. controls (6.2, 23.6) and
      (5.0, 27.8) .. (6.7, 31.7).. controls (10.4, 40.4) and (31.4, 47.2)
      .. (62.6, 49.9).. controls (67.3, 50.3) and (68.5, 50.3) .. (68.1, 50.0)
      -- cycle;
    \end{scope}
  \end{tikzpicture}
}

\def\SngAfour{%
  \begin{tikzpicture}[y=1cm, x=1cm,
    yscale=\globalscale,xscale=\globalscale,%
    baseline={([yshift=-1mm]current bounding box.center)}]
    \begin{scope}[fill=black,cm={1,0,0,1,(0, 0)}]
      \path[fill=black] (59.7, 153.6).. controls (58.7, 152.3) and (58.8,
      151.9) .. (61.9, 147.8).. controls (70.5, 136.5) and (74.2, 129.2)
      .. (73.6, 124.9).. controls (72.9, 120.4) and (70.1, 117.7) .. (59.9,
      111.6).. controls (35.5, 97.1) and (17.0, 80.5) .. (8.6,
      65.5).. controls (4.6, 58.5) and (2.7, 50.5) .. (3.5, 44.3).. controls
      (5.0, 32.5) and (14.4, 21.6) .. (29.3, 14.1).. controls (42.6, 7.5) and
      (57.4, 4.2) .. (74.6, 4.2).. controls (96.6, 4.2) and (115.0, 10.2)
      .. (126.2, 21.0).. controls (130.2, 24.7) and (132.3, 27.6) .. (134.6,
      32.3).. controls (137.5, 38.2) and (138.1, 40.7) .. (138.1,
      47.1).. controls (138.1, 53.2) and (137.8, 54.7) .. (135.3,
      59.8).. controls (132.2, 65.9) and (127.8, 71.0) .. (116.5,
      81.5).. controls (103.4, 93.7) and (97.4, 100.2) .. (92.6,
      107.4).. controls (91.5, 109.1) and (88.3, 114.9) .. (85.5,
      120.4).. controls (76.1, 138.6) and (75.5, 139.6) .. (68.8,
      147.0).. controls (61.2, 155.4) and (61.1, 155.5) .. (59.7, 153.6) --
      cycle(82.8, 116.4).. controls (90.3, 101.6) and (94.7, 96.2) .. (111.1,
      80.8).. controls (123.3, 69.4) and (126.1, 66.5) .. (129.5,
      61.4).. controls (132.7, 56.6) and (134.1, 52.4) .. (134.1,
      47.1).. controls (134.1, 41.6) and (132.7, 37.1) .. (129.5,
      31.6).. controls (127.5, 28.1) and (122.8, 23.1) .. (118.9,
      20.3).. controls (111.0, 14.6) and (99.6, 10.6) .. (86.6,
      9.0).. controls (81.1, 8.3) and (69.2, 8.2) .. (63.9, 8.8).. controls
      (37.8, 11.7) and (17.1, 22.6) .. (10.1, 37.2).. controls (8.0, 41.4) and
      (7.6, 43.5) .. (7.7, 48.0).. controls (7.7, 53.9) and (9.8, 59.9)
      .. (14.1, 66.5).. controls (22.2, 79.0) and (38.5, 93.4) .. (57.9,
      105.4).. controls (69.3, 112.5) and (70.7, 113.4) .. (72.8,
      115.5).. controls (75.4, 118.0) and (76.7, 120.1) .. (77.3,
      122.6).. controls (78.0, 125.5) and (77.9, 125.4) .. (78.5,
      124.7).. controls (78.7, 124.3) and (80.6, 120.6) .. (82.8, 116.4) --
      cycle;
    \end{scope}
  \end{tikzpicture}
}

\def\SngAoneAfour{%
  \begin{tikzpicture}[y=1cm, x=1cm,
    yscale=0.8*\globalscale,xscale=0.8*\globalscale,%
    baseline={([yshift=-1mm]current bounding box.center)}]
    \begin{scope}[fill=black,cm={1,0,0,1,(0, 0)}]
      \path[fill=black] (69.1, 149.3).. controls (44.7, 147.7) and (31.1,
      144.6) .. (25.1, 139.5).. controls (20.4, 135.4) and (20.7, 129.7)
      .. (26.0, 123.4).. controls (30.7, 117.9) and (39.9, 110.9) .. (56.0,
      100.9).. controls (58.0, 99.6) and (59.6, 98.5) .. (59.6,
      98.4).. controls (59.6, 98.3) and (55.7, 94.7) .. (50.9,
      90.3).. controls (16.5, 59.0) and (3.4, 42.3) .. (2.3, 28.4).. controls
      (1.5, 19.1) and (9.6, 8.8) .. (20.5, 5.2).. controls (24.7, 3.8) and
      (27.8, 3.4) .. (34.2, 3.6).. controls (43.0, 3.9) and (50.4, 5.6)
      .. (61.3, 9.9).. controls (74.3, 15.1) and (88.8, 23.3) .. (105.0,
      34.7).. controls (113.7, 41.0) and (119.0, 41.8) .. (123.5,
      37.8).. controls (126.7, 35.0) and (130.2, 29.2) .. (135.3,
      18.3).. controls (137.1, 14.6) and (138.9, 11.3) .. (139.3,
      11.0).. controls (140.2, 10.4) and (141.3, 10.7) .. (142.0,
      11.8).. controls (142.5, 12.6) and (142.4, 12.9) .. (141.0,
      16.8).. controls (140.1, 19.1) and (138.7, 23.1) .. (137.8,
      25.8).. controls (135.1, 33.4) and (133.9, 35.5) .. (122.6,
      52.4).. controls (113.6, 65.8) and (110.4, 69.3) .. (99.0,
      78.0).. controls (94.5, 81.3) and (81.6, 89.9) .. (73.0,
      95.3).. controls (69.8, 97.2) and (67.1, 98.9) .. (67.1,
      99.0).. controls (67.0, 99.1) and (68.5, 100.6) .. (70.4,
      102.4).. controls (80.5, 111.7) and (89.4, 121.0) .. (93.0,
      126.0).. controls (100.5, 136.2) and (100.4, 143.4) .. (92.6,
      147.0).. controls (88.1, 149.0) and (78.6, 150.0) .. (69.1, 149.3) --
      cycle(86.4, 144.5).. controls (89.6, 143.8) and (92.7, 142.3) .. (93.7,
      141.0).. controls (95.8, 138.0) and (93.2, 132.4) .. (85.7,
      123.6).. controls (81.1, 118.4) and (63.9, 101.6) .. (63.0,
      101.6).. controls (62.7, 101.6) and (54.2, 107.0) .. (47.8,
      111.2).. controls (31.5, 122.1) and (24.4, 129.7) .. (26.2,
      134.2).. controls (27.7, 137.7) and (34.8, 140.7) .. (46.5,
      142.7).. controls (50.8, 143.4) and (62.1, 144.7) .. (68.1,
      145.1).. controls (72.2, 145.4) and (83.8, 145.0) .. (86.4, 144.5) --
      cycle(70.5, 91.9).. controls (79.6, 86.2) and (91.5, 78.2) .. (96.1,
      74.8).. controls (106.8, 66.8) and (110.6, 62.8) .. (117.9,
      51.8).. controls (121.0, 47.2) and (123.5, 43.4) .. (123.5,
      43.3).. controls (123.5, 43.2) and (122.6, 43.4) .. (121.6,
      43.8).. controls (116.1, 45.7) and (110.2, 43.8) .. (101.1,
      37.2).. controls (91.7, 30.4) and (82.3, 24.7) .. (72.1,
      19.6).. controls (57.6, 12.3) and (46.2, 8.5) .. (36.0, 7.8).. controls
      (24.5, 7.0) and (16.2, 10.0) .. (10.5, 17.1).. controls (5.3, 23.5) and
      (5.1, 29.1) .. (9.7, 38.3).. controls (15.2, 49.5) and (27.3, 63.0)
      .. (52.4, 85.9).. controls (56.8, 90.0) and (61.1, 93.9) .. (61.9,
      94.6).. controls (62.7, 95.4) and (63.5, 96.0) .. (63.6,
      96.0).. controls (63.8, 96.0) and (66.9, 94.1) .. (70.5, 91.9) -- cycle;
    \end{scope}
  \end{tikzpicture}
}

\def\SngAtwoAfour{%
  \begin{tikzpicture}[y=1cm, x=1cm,
    yscale=0.8*\globalscale,xscale=0.8*\globalscale,%
    baseline={([yshift=-1mm]current bounding box.center)}]
    \begin{scope}[fill=black,cm={1,0,0,1,(0, 0)}]
      \path[fill=black] (59.7, 153.6).. controls (58.7, 152.3) and (58.8,
      151.9) .. (61.9, 147.8).. controls (70.5, 136.5) and (74.2, 129.2)
      .. (73.6, 124.9).. controls (72.9, 120.4) and (70.1, 117.7) .. (59.9,
      111.6).. controls (38.1, 98.6) and (20.3, 83.3) .. (11.3,
      69.9).. controls (2.2, 56.4) and (0.9, 43.8) .. (7.4, 32.9).. controls
      (10.8, 27.4) and (15.2, 23.6) .. (20.2, 22.0).. controls (22.7, 21.2)
      and (23.7, 21.1) .. (27.5, 21.0).. controls (34.4, 21.0) and (38.8,
      22.1) .. (55.9, 28.2).. controls (71.5, 33.8) and (79.6, 36.1) .. (89.4,
      37.6).. controls (95.0, 38.5) and (108.7, 38.5) .. (114.1,
      37.6).. controls (119.0, 36.8) and (126.4, 34.9) .. (130.0,
      33.5).. controls (133.2, 32.3) and (134.1, 32.2) .. (134.9,
      33.1).. controls (135.6, 33.9) and (135.6, 35.4) .. (135.0,
      35.9).. controls (134.8, 36.1) and (132.1, 37.6) .. (129.1,
      39.2).. controls (110.7, 49.0) and (100.7, 58.0) .. (97.6,
      67.7).. controls (96.2, 71.8) and (96.0, 73.9) .. (95.6,
      82.9).. controls (95.2, 92.6) and (94.8, 95.9) .. (93.7,
      100.8).. controls (92.4, 106.2) and (90.8, 109.9) .. (85.2,
      120.8).. controls (76.3, 138.2) and (75.5, 139.6) .. (68.8,
      147.0).. controls (61.2, 155.4) and (61.1, 155.5) .. (59.7, 153.6) --
      cycle(82.9, 115.9).. controls (89.6, 102.6) and (90.7, 98.4) .. (91.3,
      83.4).. controls (91.7, 73.6) and (92.3, 69.9) .. (93.8,
      65.5).. controls (96.6, 58.0) and (103.5, 50.2) .. (113.3,
      43.6).. controls (114.8, 42.7) and (116.0, 41.8) .. (116.0,
      41.7).. controls (115.9, 41.7) and (113.9, 41.9) .. (111.7,
      42.2).. controls (103.8, 43.1) and (94.4, 42.8) .. (85.7,
      41.2).. controls (76.7, 39.7) and (69.5, 37.6) .. (52.0,
      31.3).. controls (38.5, 26.5) and (33.9, 25.3) .. (28.2,
      25.3).. controls (24.9, 25.2) and (23.3, 25.4) .. (21.8,
      25.8).. controls (17.2, 27.3) and (12.8, 31.5) .. (10.1,
      37.2).. controls (8.0, 41.4) and (7.6, 43.5) .. (7.7, 48.0).. controls
      (7.7, 53.9) and (9.8, 59.9) .. (14.1, 66.5).. controls (22.2, 79.0) and
      (38.5, 93.4) .. (57.9, 105.4).. controls (69.3, 112.5) and (70.7, 113.4)
      .. (72.8, 115.5).. controls (75.4, 118.0) and (76.7, 120.1) .. (77.3,
      122.6).. controls (78.0, 125.6) and (78.0, 125.5) .. (78.4,
      124.9).. controls (78.5, 124.6) and (80.6, 120.5) .. (82.9, 115.9) --
      cycle;
    \end{scope}
  \end{tikzpicture}
}

\def\SngDfour{%
  \begin{tikzpicture}[y=1cm, x=1cm,
    yscale=\globalscale,xscale=\globalscale,%
    baseline={([yshift=-1mm]current bounding box.center)}]
    \begin{scope}[fill=black,cm={1,0,0,1,(0, 0)}]
      \path[fill=black] (67.7, 118.6).. controls (58.7, 117.3) and (53.1, 112.3) .. (51.3, 103.6).. controls (50.6, 100.2) and (50.8, 92.9) .. (51.8, 87.8).. controls (53.7, 78.4) and (58.7, 65.8) .. (65.2, 54.2).. controls (66.4, 52.1) and (67.4, 50.3) .. (67.4, 50.2).. controls (67.4, 50.2) and (64.5, 50.1) .. (60.9, 50.1).. controls (45.1, 50.1) and (29.2, 48.1) .. (19.3, 44.9).. controls (10.4, 42.1) and (4.5, 37.9) .. (1.9, 32.6).. controls (0.7, 30.2) and (0.5, 29.5) .. (0.4, 26.7).. controls (0.2, 22.8) and (0.9, 19.8) .. (2.6, 15.9).. controls (4.4, 12.3) and (8.1, 6.7) .. (10.1, 5.0).. controls (15.9, -0.1) and (24.0, 0.0) .. (33.7, 5.4).. controls (44.3, 11.1) and (57.6, 24.3) .. (68.5, 39.7).. controls (70.1, 41.9) and (71.4, 43.7) .. (71.5, 43.7).. controls (71.5, 43.7) and (72.4, 42.5) .. (73.4, 41.0).. controls (78.9, 32.8) and (85.2, 25.1) .. (92.2, 18.0).. controls (109.5, 0.5) and (124.3, -4.4) .. (133.2, 4.5).. controls (136.5, 7.7) and (140.6, 15.5) .. (142.0, 20.8).. controls (142.7, 24.0) and (142.7, 28.3) .. (141.8, 31.1).. controls (138.7, 41.3) and (125.8, 46.9) .. (100.4, 49.0).. controls (98.0, 49.2) and (91.5, 49.5) .. (85.9, 49.7).. controls (80.3, 50.0) and (75.7, 50.2) .. (75.6, 50.2).. controls (75.6, 50.2) and (76.4, 51.8) .. (77.5, 53.6).. controls (78.6, 55.5) and (80.9, 59.9) .. (82.7, 63.5).. controls (92.5, 83.3) and (95.3, 99.1) .. (90.6, 109.0).. controls (88.5, 113.3) and (84.2, 116.6) .. (78.9, 118.0).. controls (76.4, 118.6) and (70.4, 119.0) .. (67.7, 118.6) -- cycle(78.2, 113.8).. controls (85.3, 111.9) and (88.5, 106.9) .. (88.5, 97.9).. controls (88.5, 89.4) and (85.4, 78.4) .. (79.2, 66.0).. controls (76.5, 60.3) and (71.9, 52.2) .. (71.5, 52.2).. controls (70.9, 52.2) and (64.7, 63.8) .. (62.1, 70.0).. controls (59.7, 75.7) and (57.4, 82.3) .. (56.2, 87.4).. controls (55.4, 91.1) and (55.2, 92.3) .. (55.2, 97.4).. controls (55.2, 102.7) and (55.3, 103.4) .. (56.1, 105.3).. controls (58.1, 110.5) and (61.5, 113.2) .. (67.5, 114.3).. controls (69.9, 114.7) and (75.9, 114.4) .. (78.2, 113.8) -- cycle(67.4, 45.6).. controls (67.4, 45.5) and (66.0, 43.6) .. (64.4, 41.3).. controls (48.5, 19.1) and (30.0, 4.4) .. (19.2, 5.4).. controls (15.0, 5.9) and (12.2, 8.0) .. (8.8, 13.5).. controls (5.9, 18.2) and (4.8, 21.2) .. (4.7, 24.9).. controls (4.5, 27.8) and (4.6, 28.3) .. (5.5, 30.3).. controls (9.2, 38.5) and (24.0, 43.8) .. (47.8, 45.5).. controls (52.8, 45.8) and (67.4, 46.0) .. (67.4, 45.6) -- cycle(87.3, 45.5).. controls (113.7, 44.5) and (127.4, 41.5) .. (134.2, 35.2).. controls (137.0, 32.7) and (138.0, 30.5) .. (138.2, 26.8).. controls (138.4, 22.2) and (137.2, 18.2) .. (133.4, 11.8).. controls (130.2, 6.4) and (127.4, 4.4) .. (122.6, 4.4).. controls (113.1, 4.5) and (99.1, 14.8) .. (85.1, 32.3).. controls (81.8, 36.4) and (75.8, 44.7) .. (75.6, 45.5).. controls (75.4, 45.9) and (76.3, 45.9) .. (87.3, 45.5) -- cycle;
    \end{scope}
  \end{tikzpicture}
}

\def\SngAfive{%
  \begin{tikzpicture}[y=1cm, x=1cm,
    yscale=\globalscale,xscale=\globalscale,%
    baseline={([yshift=-1mm]current bounding box.center)}]
    \begin{scope}[fill=black,cm={1,0,0,1,(0, 0)}]
      \path[fill=black] (7.2, 73.7).. controls (4.3, 73.3) and (2.4, 72.0)
      .. (1.2, 69.6).. controls (-1.1, 65.1) and (0.3, 57.9) .. (5.8,
      44.8).. controls (9.5, 35.9) and (23.0, 10.0) .. (26.7, 4.8).. controls
      (28.9, 1.7) and (32.7, 0.8) .. (36.6, 2.6).. controls (40.1, 4.2) and
      (43.3, 7.1) .. (55.2, 19.2).. controls (69.9, 34.0) and (75.0, 38.2)
      .. (82.1, 40.9).. controls (84.1, 41.6) and (85.1, 41.8) .. (88.2,
      41.8).. controls (92.6, 41.8) and (95.8, 40.9) .. (100.4,
      38.2).. controls (106.4, 34.7) and (109.5, 32.0) .. (123.5,
      18.0).. controls (136.8, 4.6) and (139.0, 2.6) .. (142.7,
      1.0).. controls (146.2, -0.5) and (150.0, 0.3) .. (152.4,
      3.1).. controls (154.5, 5.6) and (165.1, 27.9) .. (169.9,
      39.9).. controls (174.3, 50.9) and (176.1, 58.1) .. (176.2,
      63.9).. controls (176.2, 67.0) and (176.1, 67.5) .. (175.1,
      69.4).. controls (174.4, 70.8) and (173.6, 71.8) .. (172.6,
      72.5).. controls (171.1, 73.5) and (170.9, 73.6) .. (166.9,
      73.6).. controls (163.1, 73.5) and (162.6, 73.4) .. (158.6,
      72.1).. controls (153.4, 70.3) and (147.9, 67.9) .. (134.6,
      61.5).. controls (106.2, 47.8) and (95.3, 44.7) .. (82.7,
      46.8).. controls (73.2, 48.5) and (63.8, 52.0) .. (40.7,
      62.6).. controls (20.1, 72.1) and (13.0, 74.4) .. (7.2, 73.7) --
      cycle(14.5, 68.8).. controls (19.1, 67.5) and (24.5, 65.3) .. (40.2,
      58.2).. controls (59.8, 49.3) and (65.4, 47.1) .. (76.5,
      43.8).. controls (77.2, 43.6) and (77.0, 43.3) .. (74.3,
      41.7).. controls (68.1, 37.9) and (63.7, 33.9) .. (47.4,
      17.4).. controls (38.7, 8.5) and (35.5, 6.0) .. (32.6, 5.9).. controls
      (31.0, 5.9) and (30.5, 6.4) .. (27.5, 11.6).. controls (14.1, 35.1) and
      (6.0, 52.9) .. (4.6, 62.2).. controls (3.5, 69.2) and (6.1, 70.9)
      .. (14.5, 68.8) -- cycle(170.8, 68.4).. controls (174.3, 64.1) and
      (170.3, 49.7) .. (158.9, 25.0).. controls (154.1, 14.8) and (150.1, 7.0)
      .. (149.1, 5.7).. controls (147.7, 4.2) and (146.5, 4.0) .. (144.2,
      5.0).. controls (140.9, 6.6) and (138.2, 9.0) .. (127.3,
      20.1).. controls (113.9, 33.9) and (107.4, 39.6) .. (101.9,
      42.2).. controls (100.3, 43.0) and (100.2, 43.3) .. (101.7,
      43.6).. controls (103.8, 44.0) and (111.6, 46.6) .. (115.1,
      48.0).. controls (120.6, 50.2) and (126.6, 53.0) .. (137.7,
      58.2).. controls (154.1, 66.2) and (159.7, 68.5) .. (164.6,
      69.3).. controls (167.6, 69.8) and (170.0, 69.5) .. (170.8, 68.4) --
      cycle;
    \end{scope}
  \end{tikzpicture}
}

\def\SngDfive{%
  \begin{tikzpicture}[y=1cm, x=1cm,
    yscale=\globalscale,xscale=\globalscale,%
    baseline={([yshift=-1mm]current bounding box.center)}]
    \begin{scope}[fill=black,cm={1,0,0,1,(0, 0)}]
      \path[fill=black] (102.6, 73.0).. controls (99.3, 72.5) and (94.8, 70.8)
      .. (91.8, 69.0).. controls (88.1, 66.6) and (82.2, 60.7) .. (79.4,
      56.4).. controls (74.3, 48.6) and (70.1, 38.1) .. (67.5,
      26.2).. controls (67.1, 24.6) and (66.8, 23.3) .. (66.7,
      23.3).. controls (66.6, 23.3) and (66.3, 24.6) .. (66.0,
      26.3).. controls (64.8, 35.1) and (61.7, 45.4) .. (58.2,
      52.3).. controls (54.2, 60.3) and (47.7, 67.2) .. (41.2,
      70.2).. controls (37.1, 72.1) and (33.6, 72.8) .. (28.8,
      72.6).. controls (23.8, 72.4) and (21.1, 71.6) .. (16.5,
      69.3).. controls (3.7, 62.5) and (-2.8, 47.8) .. (1.6, 35.0).. controls
      (6.8, 19.6) and (24.9, 7.2) .. (49.4, 2.3).. controls (62.2, -0.2) and
      (76.8, -0.5) .. (89.3, 1.6).. controls (116.9, 6.4) and (134.8, 22.6)
      .. (134.8, 43.0).. controls (134.8, 47.1) and (134.2, 50.5) .. (132.9,
      54.1).. controls (129.5, 62.9) and (122.0, 69.9) .. (113.2,
      72.3).. controls (110.2, 73.1) and (105.4, 73.4) .. (102.6, 73.0) --
      cycle(113.2, 67.9).. controls (118.3, 66.3) and (124.0, 61.6) .. (127.0,
      56.7).. controls (132.6, 47.1) and (131.7, 34.5) .. (124.4,
      24.8).. controls (115.1, 12.5) and (97.5, 5.3) .. (74.3, 4.4) -- (68.3,
      4.1) -- (68.6, 6.6).. controls (69.3, 14.2) and (71.2, 24.3) .. (73.2,
      31.4).. controls (80.8, 59.0) and (96.3, 73.2) .. (113.2, 67.9) --
      cycle(34.6, 68.1).. controls (36.8, 67.6) and (41.6, 65.4) .. (43.4,
      64.1).. controls (55.1, 55.2) and (62.1, 36.3) .. (63.7, 8.7) -- (64.0,
      4.5) -- (61.3, 4.8).. controls (57.4, 5.2) and (49.2, 6.6) .. (45.3,
      7.5).. controls (26.8, 12.2) and (12.7, 21.7) .. (6.8, 33.4).. controls
      (3.6, 39.8) and (3.7, 47.2) .. (7.1, 54.0).. controls (10.6, 61.0) and
      (17.7, 66.4) .. (25.6, 68.2).. controls (27.0, 68.5) and (32.8, 68.4)
      .. (34.6, 68.1) -- cycle;
    \end{scope}
  \end{tikzpicture}
}

\def\T
{\begin{tikzpicture}[scale=1.5]
    \draw[line width = 6] (0,0) to (0.8,0);
  \end{tikzpicture}}

\def\Dn
{\begin{tikzpicture}[scale=.20]
		\draw[line width = 4] (0,0) to[quick curve through={(2,0.5) (3,1.5) (2,3) (1,1.5) (2,0.5)}]
		(4,0);
\end{tikzpicture}}

\def\Unn
{\begin{tikzpicture}[scale=.20,xscale=0.8]
		\draw[line width = 2] (0,0) to[quick curve through={(2,0.5) (3,1.5) (2,3) (1,1.5) (2,0.5) (4,0) (6,0.5) (7,1.5) (6,3) (5,1.5) (6,0.5)}]
		(8,0);
\end{tikzpicture}}

\def\Znnn
{\begin{tikzpicture}[scale=.20,xscale=0.7]
		\draw (0,0) to[quick curve through={(2,0.5) (3,1.5) (2,3) (1,1.5) (2,0.5) (4,0) (6,0.5) (7,1.5) (6,3) (5,1.5) (6,0.5) (8,0) (10,0.5) (11,1.5) (10,3) (9,1.5) (10,0.5)}]
		(12,0);
\end{tikzpicture}}

\def\DNA
{\begin{tikzpicture}[scale=.15]
		\draw (-0.5,0.5) to[quick curve through={(1,1) (2,2) (1,3) (0,4) (1,5) (2,6) (1,7)}]
		(-0.5,7.5);
		\draw (2.5,0.5) to[quick curve through={(1,1) (0,2) (1,3) (2,4) (1,5) (0,6) (1,7)}]
		(2.5,7.5);
		\clip (-2.5,0) rectangle (4.5,8);
\end{tikzpicture}}

\def\UeU
{\begin{tikzpicture}[scale=.15]
		\draw[line width = 2] (-0.5,0.5) to[quick curve through={(2,3)}]
		(-0.5,6.5);
		\draw[line width = 2] (2.5,0.5) to[quick curve through={(0,3)}]
		(2.5,6.5);
		\clip (-2.5,0) rectangle (4.5,8);
\end{tikzpicture}}

\def\CAVE{
\begin{tikzpicture}[scale= .15]
		\draw (0,0) to[quick curve through={(2,3)}]
		(6,4);
		\draw (-1,1) to[quick curve through={(2,0.5)}]
		(6,0);
		\draw (3,-1) to[out angle = 90, in angle = 90, curve through={(4,5)}]
		(5,-1);
\end{tikzpicture}}

\def\UeZn
{\begin{tikzpicture}[scale= .14]
		\draw (-0.5,0.5) to[quick curve through={(2,3) (-1.5,6.5) (-1.5, 5.5) (-1.5,6.5)}]
		(-3.5,6.5);
		\draw[line width = 2] (2.5,0.5) to[curve through={(0,3)}]
		(3.5,6.5);
\end{tikzpicture}}

\def\ZZeU
{\begin{tikzpicture}[scale=.15]
		\draw (-0.5,-0.5) to[quick curve through={(2,3)}]
		(-0.5,6.5);
		\draw (2.5,-0.5) to[quick curve through={(0,3)}]
		(2.5,6.5);
		\draw[line width = 2] (-2,0) to (4,0);
\end{tikzpicture}}

\def\ZZeZZ
{\begin{tikzpicture}[scale=.15]
		\draw (0,2) to[curve through={(3,1)}]
		(6,3);
		\draw (4,3) to[curve through={(7,1)}] (10,2);
		\draw (0,1) to[in angle = 120, curve through={(3,2)}]
		(6,0);
		\draw (4,0) to[out angle = 60, curve through={(7,2)}] (10,1);
\end{tikzpicture}}

\def\ZZeUn
{\begin{tikzpicture}[scale=.12]
		\draw (-0.5,-0.5) to[quick curve through={(2,3)}]
		(-0.5,6.5);
		\draw (2.5,-0.5) to[quick curve through={(0,3)}]
		(2.5,6.5);
		\draw (-2,0) to[curve through={(2,0)(4,0)(5,0.5) (5,2) (5,0.5)}] (8,0);
\end{tikzpicture}}

\def\BRAID
{\begin{tikzpicture}[scale=.13]
		\draw (0,5) to (5,0);
		\draw (0,4.5) to (7,1.5);
		\draw (1,2) to (9,5);
		\draw (3,-1) to (8,5);
\end{tikzpicture}}

\def\ZneZn
{\begin{tikzpicture}[scale =.13]
		\draw (-0.5,0.5) to[quick curve through={(2,3) (-1.5,6.5) (-1.5, 5.5) (-1.5,6.5)}]
		(-3.5,6.5);
		\draw (2.5,0.5) to[curve through={(0,3) (3.5,6.5) (3.5,5.5) (3.5,6.5)}]
		(5.5,6.5);
\end{tikzpicture}}

\def\DeU
{\begin{tikzpicture}[scale=.3]
		\draw[line width = 4] (0,0) to
		(3.5,0);
		\draw[line width = 2] (0.6,1) to (0.6,-0.5);
\end{tikzpicture}}

\def\DeZn
{\begin{tikzpicture}[scale=.7]
		\draw[line width = 4] (0,0) to (1.5,0);
		\draw (0.6,1) to[quick curve through={(0.7,0.4) (0.9,0.3) (1,0.4) (0.9,0.5) (0.7,0.4)}] (0.6,-0.5);
\end{tikzpicture}}

\def\UeUn
{\begin{tikzpicture}[scale=.5]
		\draw[line width = 2] (0,0) to[quick curve through={(1.2,0.2) (1.4, 0.4) (1.2, 0.8) (1,0.4) (1.2,0.2)}] (2,0);
		\draw[line width = 2] (0.4,1) to (0.4,-0.5);
\end{tikzpicture}}

\def\UeUeU
{\begin{tikzpicture}[scale=.5]
		\draw[line width = 2] (0,0) to (2,0);
		\draw[line width = 2] (0,1) to (0.5,-0.5);
		\draw[line width = 2] (2,1) to (1.5, -0.5);
\end{tikzpicture}}

\def\UeUeUeZ
{\begin{tikzpicture}[scale=1.2]
		\draw (0,0) to (1,0);
		\draw[line width = 2] (0,0.5) to (0.25,-0.25);
		\draw[line width = 2] (1,0.5) to (0.75, -0.25);
		\draw[line width = 2] (0.5,0.5) to (0.5,-0.25);
\end{tikzpicture}}

\def\UeeeZ
{\begin{tikzpicture}[scale=.20]
		\draw (0,0.5) to[quick curve through={(1,1) (2,2) (1,3) (0,4) (1,5)}] (2,5.5);
		\draw[line width = 2] (2,0.5) to[quick curve through={(1,1) (0,2) (1,3) (2,4) (1,5)}] (0,5.5);
\end{tikzpicture}}

\def\UeeeUeZ
{\begin{tikzpicture}[scale=.2]
		\draw (0,0.5) to[quick curve through={(1,1) (2,2) (1,3) (0,4) (1,5)}] (2,5.5);
		\draw (3,0.5) to[quick curve through={(1,1) (0,2) (1,3) (2,4) (1,5)}] (0,5.5);
		\draw[line width = 2] (2,0) to (3,3);
\end{tikzpicture}}

\def\UeeeUn
{\begin{tikzpicture}[scale=.17]
		\draw (0,0.5) to[quick curve through={(1,1) (2,2) (1,3) (0,4) (1,5)}] (2,5.5);
		\draw (2.3,0) to[quick curve through={(2.1, 0.5) (2.4,1) (2.7,0.5) (2,0.5) (1,1) (0,2) (1,3) (2,4) (1,5)}] (0,5.5);
\end{tikzpicture}}

\def\ZeUnn
{\begin{tikzpicture}[scale=.16]
		\draw (-1,0) to[quick curve through={(2,0.5) (3,1.5) (2,3) (1,1.5) (2,0.5) (4,0) (6,0.5) (7,1.5) (6,3) (5,1.5) (6,0.5)}]
		(8,0);
		\draw[line width = 2] (0,5) to (0,-1);
\end{tikzpicture}}

\def\ZneUn
{\begin{tikzpicture}[scale=.63]
		\draw (0.6,1) to[quick curve through={(0.7,0.4) (0.9,0.3) (1,0.4) (0.9,0.5) (0.7,0.4)}] (0.6,-0.6);
		\draw[line width = 2] (1.8,-0.4) to[quick curve through={(1.2,-0.3) (1.1,-0.1) (1.2,0) (1.3,-0.1) (1.2,-0.3)}] (0.3,-0.4);
\end{tikzpicture}}

\def\UeZeU
{\begin{tikzpicture}[scale=0.14]
		\draw[line width = 2] (-0.5,0.5) to[quick curve through={(2,3)}]
		(-0.5,6.5);
		\draw (2.5,0.5) to[quick curve through={(0,3)}]
		(4,6.5);
		\draw[line width = 2] (3,5) to (3,8);
\end{tikzpicture}}

\def\UneZeZ
{\begin{tikzpicture}[scale=.6]
		\draw[line width = 2] (0,0) to (2,0);
		\draw (0.6,1) to[quick curve through={(0.7,0.4) (0.9,0.3) (1,0.4) (0.9,0.5) (0.7,0.4)}] (0.6,-0.6);
		\draw[line width = 2] (1.5,1) to (1.5, -0.5);
\end{tikzpicture}}

\def\ZeUneZ
{\begin{tikzpicture}[scale=.45,xscale=0.8]
        \draw[] (-1,0) to[quick curve through={(1,0.2) (1.2,0.4) (1,0.8) (0.8,0.4) (1,0.2)}]
		(3,0);
		\draw[line width = 2] (-0.5,1.5) to (0,-0.5);
		\draw[line width = 2] (2.5,1.5) to (2, -0.5);
\end{tikzpicture}}

\def\UnneUn
{\begin{tikzpicture}[scale=.25]
		\draw (-0.8,3) to[quick curve through={(-0.6,1.8) (-0.2,1.6) (0,1.8) (-0.2,2) (-0.6,1.8)}] (-0.8,-0.7);
		\draw (-2,0.5) to[quick curve through={(1,0.25) (1.5,0.75) (1,1.5) (0.5,0.75) (1,0.25) (2,0) (3,0.25) (3.5,0.75) (3,1.5) (2.5,0.75) (3,0.25)}]
		(4,0);
\end{tikzpicture}}

\def\UeZeZn
{\begin{tikzpicture}[scale=0.15]
		\draw[line width = 2] (-0.5,0.5) to[quick curve through={(2,3)}]
		(-0.5,6.5);
		\draw (2.5,0.5) to[quick curve through={(0,3)}]
		(4,6.5);
		\draw (3.2,7.5) to[quick curve through={(3.4,4.8) (3.8,4.6) (4,4.8) (3.8,5) (3.4,4.8)}] (3.2,3);
\end{tikzpicture}}

\def\ZneZeU
{\begin{tikzpicture}[scale= .15]
		\draw (-0.5,0.5) to[quick curve through={(2,3) (-1.5,6.5) (-1.5, 5.5) (-1.5,6.5)}]
		(-3.5,6.5);
		\draw (2.5,0.5) to[curve through={(0,3)}]
		(3.5,6.5);
		\draw[line width = 2] (3,5) to (3,8);
\end{tikzpicture}}

\def\ZneUeZn
{\begin{tikzpicture}[scale=0.55]
	    \draw[line width = 2] (0,0) to (2.2,0);
		\draw (0.6,1) to[quick curve through={(0.7,0.4) (0.9,0.3) (1,0.4) (0.9,0.5) (0.7,0.4)}] (0.6,-0.6);
		\draw (1.6,1) to[quick curve through={(1.7,0.4) (1.9,0.3) (2,0.4) (1.9,0.5) (1.7,0.4)}] (1.6,-0.6);
\end{tikzpicture}}

\def\ZneZneU
{\begin{tikzpicture}[scale=0.4]
       \draw[] (-1,0) to[quick curve through={(1,0.2) (1.2,0.4) (1,0.8) (0.8,0.4) (1,0.2)}]
		(3,0);
		\draw (-0.6,1.5) to[quick curve through={(-0.4,0.75) (-0.15,0.65) (-0,0.75) (-0.25,1) (-0.4,0.75)}] (0,-0.5);
		\draw[line width = 2] (2.5,1.5) to (2, -0.5);
\end{tikzpicture}}

\def\UeZZeU
{\begin{tikzpicture}[scale=0.15]
		\draw (-0.5,0.5) to[quick curve through={(2,3)}]
		(-2,6.5);
		\draw (2.5,0.5) to[quick curve through={(0,3)}]
		(4,6.5);
		\draw[line width = 2] (3,5) to (3,8);
		\draw[line width = 2] (-1,5) to (-1,8);
\end{tikzpicture}}

\def\ZneUeUeZ
{\begin{tikzpicture}[scale=1.2]
		\draw (0,0) to (1,0);
		\draw (0,0.5) to[quick curve through={(0.1,0.125) (0.225,0.075) (0.3,0.125) (0.175,0.25) (0.1,0.125)}] (0.3,-0.25);;
		\draw[line width = 2] (1,0.5) to (0.75, -0.25);
		\draw[line width = 2] (0.5,0.5) to (0.5,-0.25);
\end{tikzpicture}}

\def\ZeeeZeZn
{\begin{tikzpicture}[scale=0.18]
		\draw (0,0.5) to[quick curve through={(1,1) (2,2) (1,3) (0,4) (1,5)}] (2,5.5);
		\draw (3,0.5) to[quick curve through={(1,1) (0,2) (1,3) (2,4) (1,5)}] (0,5.5);
		\draw (2,0) to[quick curve through={(2.5,1.5)(2.8,1.75)(3.5,1.5)(2.7,1.25)(2.5,1.5)}] (3,3);
\end{tikzpicture}}

\def\ZneZeZn
{\begin{tikzpicture}[scale= .13]
		\draw (-0.5,0.5) to[quick curve through={(2,3) (-1.5,6.5) (-1.5, 5.5) (-1.5,6.5)}]
		(-3.5,6.5);
		\draw (2.5,0.5) to[curve through={(0,3)}]
		(3.5,6.5);
		\draw (3.2,7.5) to[quick curve through={(3.4,4.8) (3.8,4.6) (4,4.8) (3.8,5) (3.4,4.8)}] (3.2,3);
\end{tikzpicture}}

\def\ZneZneZn
{\begin{tikzpicture}[scale=.4,xscale=0.8]
        \draw[] (-1,0) to[quick curve through={(1,0.2) (1.2,0.4) (1,0.8) (0.8,0.4) (1,0.2)}]
		(3,0);
		\draw (-0.6,1.5) to[quick curve through={(-0.4,0.75) (-0.15,0.65) (-0,0.75) (-0.25,1) (-0.4,0.75)}] (0,-0.5);
		\draw (2.6,1.5) to[quick curve through={(2.4,0.75) (2.15,0.65) (2,0.75) (2.25,1) (2.4,0.75)}] (2,-0.5);
\end{tikzpicture}}

\def\ZZeZeU
{\begin{tikzpicture}[scale=.12]
		\draw (-0.5,-0.5) to[quick curve through={(2,3)}]
		(-0.5,6.5);
		\draw (2.5,-0.5) to[quick curve through={(0,3)}]
		(2.5,6.5);
		\draw (-2,0) to (7,0);
		\draw[line width = 2] (5,-0.5) to (5,5.5);
\end{tikzpicture}}

\def\UeZZeZn
{\begin{tikzpicture}[scale=0.12]
		\draw (-0.5,0.5) to[quick curve through={(2,3)}]
		(-2,6.5);
		\draw (2.5,0.5) to[quick curve through={(0,3)}]
		(4,6.5);
		\draw (3.2,7.5) to[quick curve through={(3.4,4.8) (3.8,4.6) (4,4.8) (3.8,5) (3.4,4.8)}] (3.2,3);
		\draw[line width = 2] (-1.2,7.5) to (-1.2,3);
\end{tikzpicture}}

\def\ZneZneUeZ
{\begin{tikzpicture}[scale=1]
		\draw (-0.1,0) to (1,0);
		\draw (0.1,0.5) to[quick curve through={(0.1,0.2)(0.2,0.13)(0.3,0.2)(0.2,0.3)(0.1,0.2)}] (0.1,-0.25);
		\draw[line width = 2] (1,0.5) to (0.75, -0.25);
		\draw (0.5,0.5) to[quick curve through={(0.5,0.2)(0.6,0.13)(0.7,0.2)(0.6,0.3)(0.5,0.2)}] (0.5,-0.25);
\end{tikzpicture}}

\def\ZZeZeZn
{\begin{tikzpicture}[scale=0.12]
		\draw (-0.5,-0.5) to[quick curve through={(2,3)}]
		(-0.5,6.5);
		\draw (2.5,-0.5) to[quick curve through={(0,3)}]
		(2.5,6.5);
		\draw (-2,0) to (7,0);
		\draw (5,-0.5) to[quick curve through={(5.2,2)(5.7,2.5)(6.7,2)(5.7,1.5)(5.2,2)}] (5,4);
\end{tikzpicture}}

\def\ZneZZeZn
{\begin{tikzpicture}[scale=0.12]
		\draw (-0.5,0.5) to[quick curve through={(2,3)}]
		(-2,6.5);
		\draw (2.5,0.5) to[quick curve through={(0,3)}]
		(4,6.5);
		\draw (3.2,7.5) to[quick curve through={(3.4,4.8) (3.8,4.6) (4,4.8) (3.8,5) (3.4,4.8)}] (3.2,3);
		\draw (-2.2,7.5) to[quick curve through={(-1.4,4.8) (-1.8,4.6) (-2,4.8) (-1.8,5) (-1.4,4.8)}] (-1.2,3);
\end{tikzpicture}}

\def\ZneZneZneZ
{\begin{tikzpicture}[scale=1]
		\draw (-0.1,0) to (1.1,0);
		\draw (0.1,0.5) to[quick curve through={(0.1,0.2)(0.2,0.13)(0.3,0.2)(0.2,0.3)(0.1,0.2)}] (0.1,-0.25);
		\draw (0.5,0.5) to[quick curve through={(0.5,0.2)(0.6,0.13)(0.7,0.2)(0.6,0.3)(0.5,0.2)}] (0.5,-0.25);
		\draw (0.9,0.5) to[quick curve through={(0.9,0.2)(1,0.13)(1.1,0.2)(1,0.3)(0.9,0.2)}] (0.9,-0.25);
\end{tikzpicture}}

\RaymondContactDetails
\JordanContactDetails
\ReynaldContactDetails
\ElisaContactDetails

\begin{document}

\title{Reduction of Plane Quartics and Dixmier-Ohno invariants}

\begin{abstract}
  We characterise, in terms of Dixmier-Ohno invariants, the types of
  singularities that a plane quartic curve can have. We then use these results
  to obtain new criteria for determining the stable reduction types of
  non-hyperelliptic curves of genus 3.
\end{abstract}

\maketitle

\section{Introduction}

Let $C$ be a geometrically connected smooth projective curve of genus~3. If
its canonical divisor is very ample, so that the curve is non-hyperelliptic,
the embedding it defines produces a plane quartic, which is unique up to
projective transformations.
These simple facts give rise to a remarkable connection between the moduli of
curves of genus 3 and the algebra of Dixmier-Ohno invariants for the natural
action of $\SL_3$ on the vector space $\mathcal{F}_{3,4}$ of homogeneous
polynomials of degree 4 in 3 variables, that is, ternary quartics.

In a landmark paper~\cite{DM69}, Deligne and Mumford proved that the moduli
space $\mathcal{M}_g$ of smooth projective curves of genus $g \geq 2$ is
always irreducible, independently of the characteristic of the field of
definition. The first of the two proofs they give consists in compactifying this
space by adding the curves with mild singularities, the stable curves.
Although potentially complicated, blowing up singular loci and taking
normalisations makes it theoretically possible to stabilise a projective
curve~\cite{HM98}, from which it is then easier to obtain related quantities
of a geometric or arithmetic nature (genus, conductor, \textit{etc}).

At the same time, Mumford was developing geometric invariant theory (GIT),
with the aim of explicitly constructing moduli spaces as quotients of
parametrising schemes~\cite{GIT}. An important notion in GIT is that of
stability. In genus 3~\cite[Lemma~1.4]{Artebani09}, a quartic is GIT-stable if
it has at most ordinary nodes and cusps. In particular, quartics with non-zero
discriminants, \textit{i.e.}\ smooth, are GIT-stable. On the other hand, a
quartic is GIT-unstable, \textit{i.e.}\ it has a triple point or consists of a
cubic and an inﬂectional tangent, if its Dixmier invariants are all zero. If
at least one invariant is non-zero, a quartic is GIT-semi-stable. Quartics
with a tacnode are typical examples of semi-stable quartics that are not
stable.
In a sense, GIT-semi-stability can also be seen as another compactification of
the moduli space of projective curves.

A modest contribution to this area is made in this paper for the curves of
genus~3.
It first aims to classify GIT-semi-stable quartics according to their
invariants (see Thm.~\ref{thm:invtosingtypes}). Unlike the case of the
action of $\SL(2,\C)$ on binary forms of small degree (see for
instance~\cite{Mestre91} for the degree 6 case), little is known from this
viewpoint for quartics.
We start from the stratification over the complex numbers of singular
quartics.
The stratification which is based on the
  different types of singularities a plane quartic can have was studied by
  Arnol\cprime d in \cite{Arnold72,Arnold74} and has its origin in work by
  Du~Val on del Pezzo surfaces in \cite{DuVal34}.  This
  stratification has also been studied by Hui in~\cite{hui79} in great detail,
  and since we use the formulas from Hui's thesis, we will refer to this
  stratification as the Hui stratification.
This complete classification, that we extend to positive characteristic
(greater than 7), consists of 21 possibilities for irreducible quartics, and
34 possibilities for non-reducible ones (see Sec.~\ref{sec:singular-quartics},
Tab.~\ref{tab:inf} and Tab.~\ref{tab:rnf}). The subset of GIT-semi-stable
quartics, the quartics of interest here, reduces to 33 cases (see the
specialisation graph on Fig.~\ref{fig:specialize}).

This work allows us, in the second part, to efficiently determine the stable
reduction of a plane quartic curve based on its invariants.
In genus 2, for example, Mestre exhibits such a connection at the end
of~\cite{Mestre91}, and this question is completely resolved in Liu's
remarkable work~\cite{Liu93}. More generally, the canonical isomorphism between hyperelliptic curves and binary forms is extended as an holomorphic map from the Deligne-Mumford compactification to the compactification of binary forms by adding the singular ones~\cite{AvritzerLange}, and the determination of the stable reduction is determined in terms of the degenerations modulo powers of $p$ of the binary forms (``Cluster picture'') in~\cite{m2d2}. This characterisation suggests that the stable reduction type can be, in general, read from the valuations of the invariants of hyperelliptic curves (\textit{e.g.}~\cite[Thm. 6.5]{yo}): this is work in progress by Cowland Kellock \cite{Lilybelle}.

In this paper, we endeavour to generalise this approach to the case of plane
quartic curves using Dixmier-Ohno invariants.
In genus 3, however, the situation is more complicated. There are 42
possibilities for the type of stable reduction, see~\cite{BDDLL2023}.
In \textit{loc.\ cit.}\ there is a conjectural correspondence
between the stable reduction type of a plane quartic curve and the degeneration type
(``Octad picture'') of any of its Cayley octads.
We also mention here the work of Dokchitser \cite{Dokchitser21} that allows to determine the stable model from a fan associated to the Newton polygon of a plane curve, so long as the curve satisfies a condition known as ``$\Delta_v$-regularity''.
But invariant-based criteria are far more effective in
their field of application.
Partial results of this form can be found in~\cite{lllr21}, characterising good
hyperelliptic reduction and in~\cite{BCKLS20} for the special family of Ciani
plane quartics.
Our interpretation in terms of $p$-adic valuations of the invariant-based
classification for singularities allows us to give more precise results
(see Thm.~\ref{thm:GIT-MD} and Cor.~\ref{cor:stabsingtypes}).
They are summarised in Tab.~\ref{fig:labelledstablemodelgraph}. From
the reduction of the invariants we are able to determine all possibilities for
the stable reduction type of a curve with such invariants. We demonstrate their
usefulness at the end of the article, on a large database of genus 3
curves~\cite{Database3}.

Overall, the paper focuses thus on the relationship between GIT and
Deligne-Mumford compactifications, for plane quartics via the Hui
stratification (see Fig.~\ref{fig:commutativedg}).
The arrows in the diagram give
  relationships between the boundary components of the different moduli
  spaces, \textit{i.e.}\ if a curve after reduction ends up on one of the boundary
  components of one of the spaces, then the respective Theorem or Corollary
  would tell on which of the boundary components of the other moduli space the
  reduction of the curve could lie. For example, relations satisfied by Dixmier–Oh\-no invariants of a
plane quartic tell us to which stratum of the Hui stratification it
belongs. In turn, Hui normal forms for quartics in a given stratum are close
enough to the stable model that a significant amount of information
  about the stable reduction type can be obtained.
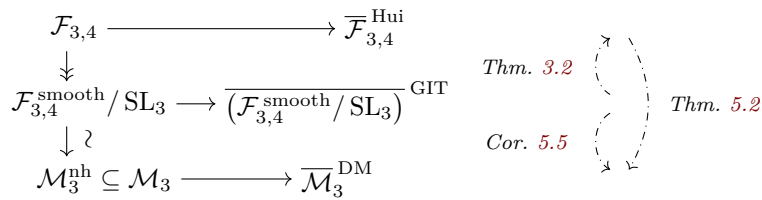
\begin{figure}[htbp]
  \begin{center}
    \begin{tikzpicture}
      \node at (-0.5,3) (A) {$\mathcal{F}_{3,4}$};%
      \node at (3.5,3) (B) {$ \overline{\mathcal{F}}_{3,4}^{\:\mathrm{Hui}}$};
      \node at (-0.3,2) (C)
      {$\mathcal{F}_{3,4}^{\:\mathrm{smooth}}/\operatorname{SL}_3$};%
      \node at (3,2) (D) {$\overline{\left(\mathcal{F}_{3,4}^{\:\mathrm{smooth}}/\operatorname{SL}_3\right)}^{\:\mathrm{GIT}}$};
      \node at (-0.1,1) (E) {$\mathcal{M}_3^{\mathrm{nh}}\subseteq\mathcal{M}_3$};%
      \node at (3,1) (F) {$\overline{\mathcal{M}}_3^{\:\mathrm{DM}}$};
      \draw[->] (A) to (B) ;%
      \draw[->>] (-0.6,2.7) to (-0.6,2.3) ;%
      \draw[->] (C) to (D) ;%
      \draw[->] (-0.6,1.7) -- ++  (0,-0.4) node[pos=.75, label={[rotate=90]below right:{$\sim$}}] {};%
      \draw[->] (E) to (F) ;%
      \node at (6.5,3) (Br) {}; \node at (6.5,2) (Dr) {}; \node at (6.5,1) (Fr) {};
      \node at (6.75,3) (Bs) {}; \node at (6.75,2) (Ds) {}; \node at (6.75,1) (Fs) {};
      \draw[dashdotted,<-] (Br.-45) to [bend right=60]  (Dr.45);
      \draw[dashdotted,->] (Dr.-45) to [bend right=60]  (Fr.45);
      \draw[dashdotted,->] (Bs.-45) to [bend left=30]  (Fs.45);
      \node at (5.5,2.5) {{\footnotesize \textit{Thm.~\ref{thm:invtosingtypes}}}};
      \node at (8,2) {{\footnotesize \textit{Thm.~\ref{thm:GIT-MD}}}};
      \node at (5.5,1.5) {{\footnotesize \textit{Cor.~\ref{cor:stabsingtypes}}}};
    \end{tikzpicture}
  \end{center}
  \caption{Relations among different quartic compactifications}
  \label{fig:commutativedg}
\end{figure}

\addtocontents{toc}{\SkipTocEntry}
\subsection*{Structure of the paper}
Sec.~\ref{sec:preliminaries} contains the necessary prerequisites about
Dixmier-Ohno invariants (Sec.~\ref{sec:invariants}), Hui's stratification of
singular quartics (Sec.~\ref{sec:singular-quartics}) and stable
reduction (Sec.~\ref{sec:stable-reduct-curv}).

In Sec.~\ref{sec:sing-invar-strata}, we then design an algorithm,
Alg.~\ref{algo:singularities}, which, given the invariants of a quartic as
input, returns which singularities it might have. The algorithm uniquely
determines the singularities for quartics all of whose singularities are nodes or
cusps, irreducible or not (15 cases). For less common quartics with more
complex singularities, the invariants no longer allow a precise
determination. For example, Alg.~\ref{algo:singularities} cannot tell whether
a quartic has just a single tacnode, or a tacnode together with an additional
node or cusp.

Its validity is based on a sequence of technical propositions given in
Sec.~\ref{sec:algebr-char}.
Our approach follows a ``guess and prove'' paradigm. Relations stated in these
propositions have been calculated heuristically over $\Q$, by interpolating
them on normalised quartics given by Hui (see tables in
Sec.~\ref{sec:singular-quartics}).  We independently prove that they are valid
in Sec.~\ref{sec:sketches-proof}. These make use of the computational algebra
system~\textsc{magma}~\cite{magma} and the interested reader can download the
corresponding script at~\cite[file \texttt{G3SingularProof.m}]{BGit23}.
Using \textsc{magma} version 2.28-8, these
computations need about 6~\textsc{gb} of memory and less than two hours on a standard
laptop (\textsc{intel i{\footnotesize 7}-{\footnotesize 8850}h cpu}).  \smallskip

In Sec.~\ref{sec:quart-stable-reduct} we relate the GIT compactification of
$\mathcal{M}_3$ with the Deligne-Mumford compactification of $\mathcal{M}_3$
by stable curves. This is done via the Hui stratification. In
Sec.~\ref{sec:stable-reduct-types} we establish Thm.~\ref{thm:GIT-MD} stating
a relation between stable reduction types and singularities of GIT-semi-stable
plane quartics. Its proof is postponed to Sec.~\ref{sec:proof-theorem}. We
translate these results into explicit reduction criteria in
Corollary~\ref{cor:stabsingtypes}.  These characterisations are tight,
\textit{i.e.}\ one-to-one, for quartics whose reduction has only nodal
singularities.  For more complex singularities, the stable reduction type is
not unique. We determine the different possibilities and we prove that all of
them may be achieved.
We conclude in Sec.~\ref{sec:modform} with an application to integer quartics
of small discriminant.

\begin{remark}
\label{rmk:characteristichypotheses}
In Sections~\ref{sec:sing-invar-strata} and~\ref{sec:sketches-proof}, we state
theorems, propositions, and lemmas only in fields of characteristic $0$. A
primary rationale for this limitation is that Dixmier-Ohno invariants do not
constitute a complete list of generators for the equivalence of ternary
quartic forms under the action of $\SL_3$ in fields of positive characteristic
$p=2$, $3$, $5$ and $7$.  It is also uncertain at the time of writing whether
they generate the invariant algebra for $p>7$ (see Rem.~\ref{rmk:DOcharp}).
A second reason for this limitation is that their proofs often rely on Gröbner
basis computations performed over $\Q$ (\textit{e.g.}~Prop.~\ref{prop:dim1} to
Prop.~\ref{prop:dim4}, see Rem.~\ref{rmk:gbcharp}).

However, we expect that these statements hold in positive characteristic,
excluding $p=2$, $3$, $5$ and $7$. In practice, one can verify a particular
statement for a particular value of $p$ by running the same computational
verification in characteristic $p$. This leads us to the conjectural result
Conj.~\ref{conj:poschar}.
\end{remark}

\addtocontents{toc}{\SkipTocEntry}
\subsection*{Acknowledgements}

We express our gratitude to I.~Dolgachev for bringing Du~Val's work to our
attention, and to V.~Dokchitser for preliminary discussions on the subject. We
also extend our thanks to the referees for their valuable suggestions, which
have enhanced the presentation of this paper.\medskip

The first author was supported by the Simons Foundation grant 550033.
The second author undertook this research whilst supported by the Engineering and Physical Sciences Research Council [EP/L015234/1], the EPSRC Centre for Doctoral Training in Geometry and Number Theory (The London School of Geometry and Number Theory) at University College London.
The third author benefits from the France 2030 framework program ``Centre
Henri Lebesgue'', ANR-11-LABX-0020-01.
The research of the fourth author is partially funded by the Melodia
ANR-20-CE40-0013 project and the 2023-08 Germaine de Sta\"el project.

\section{Preliminaries}
\label{sec:preliminaries}

This section recalls the basic definitions and results on invariant theory,
singular quartics and stable reduction needed for this work.

\subsection{Invariants}
\label{sec:invariants}

In~\cite{Dixmier87}, Dixmier gave a list of 7 homogeneous polynomial
invariants for the equivalence of ternary quartic forms under the action of
$\SL_3(\C)$, denoted $I_3$, $I_6$, $I_9$, $I_{12}$, $I_{15}$, $I_{18}$ and
$I_{27}$. The last invariant satisfies the equality $2^{40} I_{27}=D_{27}$
where $D_{27}$ is the discriminant of the form. This list was later completed
by Ohno with $J_9$, $J_{12}$, $J_{15}$, $J_{18}$, $I_{21}$ and $J_{21}$
(\cite[Theorem 4.1]{Ohno07}, see also~\cite{elsenhans15}) into a list of 13
homogeneous generators of the $\C$-algebra of invariants. These invariants are
defined over $\Z[\tfrac{1}{2 \cdot 3}]$, and are now called the
\emph{Dixmier-Ohno invariants}.  When at least one of their values at a
ternary quartic form $F$ is not zero, we denote by $\DO(F)$ the corresponding
point in the weighted projective space with weights $3$, $6$, $9$, $9$, $12$,
$12$, $15$, $15$, $18$, $18$, $21$, $21$, $27$.

This algebra of invariants is finitely generated as a
$\C[I_3, I_6, I_9,I_{12},I_{15},I_{18}, I_{27}]$-module, or equivalently
Dixmier invariants form a so-called ``homogeneous system of parameters''
(HSOP).

\begin{remark}\label{rmk:DOcharp}
  In fields of characteristic $p > 0$, HSOPs are also known for the
  corresponding invariant algebra, even for $p=3$ if we assume that the huge
  invariants of degree 54 and 81 given in~\cite[Sec.~4]{lllr21} are well
  defined.
  In particular, Dixmier invariants give an HSOP for $p>7$, except for $p=19$,
  $47$, $277$ and $523$ where $I_9$ must be replaced by $I_9 - J_9$.
  Note that it is not yet proven that adding Ohno invariants to these HSOPs is
  sufficient to generate the invariants algebras completely when $p>0$, but it
  is a commonly held belief that this is the case for $p>7$. An analogous result for hyperelliptic curves
    holds~\cite[Prop. 1.9]{LR11}.
\end{remark}

\subsection{Singular quartics}
\label{sec:singular-quartics}

In~\cite{hui79}, Hui describes a complete stratification of the space
of singular plane quartic curves defined over $\C$. He also provides normal
forms up to the action of $\SL_3(\C)$ for all singularity types. We summarise
his results in Tab.~\ref{tab:inf} and Tab.~\ref{tab:rnf}.
The first column\footnote{Drawings are for illustrative purposes only, and do not correspond to graphs of Hui normal forms.} gives the types of the singularities of the quartics given in
the third column, according to Arnold's
classification~\cite{Arnold72}, \textit{i.e.}\ $\sA_1$, \ldots,
$\sA_6$, $\sD_4$, $\sD_5$, $\sE_6$ and $\sE_7$\,. Multiples and powers of these
types denote curves with several singularities. For instance, a quartic of
type $\sA_1^2\sA_2$ has three distinct singular points, two of type $\sA_1$
and one of type $\sA_2$\,.
We prefix by ``$\,^r$'' types of curves which are not irreducible. For
example, a curve of type $\,^r\sA_1^4\,_{cub}$ has two components, a line and
a cubic. For quartics with non-isolated singularities
(Tab.~\ref{tab:noniso}), we use the more explicit notation $\ell$ and $c$,
for a line and a conic. Thus a quartic of type $c^2$ is the square of a conic.
Otherwise, the second column in these tables gives the number of parameters
$\alpha$, $\beta$, $\ldots$ needed to define the normal forms. \smallskip

\begin{table}[htbp]
  \centering
  \small
    \begin{displaymath}
      \renewcommand{\arraystretch}{1.3}
      \setlength\arrayrulewidth{1pt}
      \begin{array}{|c@{}c|c|l|}
        \hline
        \text{Type} && \# & \multicolumn{1}{c|}{\text{Normal forms}} \\
        \hline\hline
        \text{Smooth}~\text{\cite{shioda93}}
        & { \tikzsetnextfilename{SngSmooth} \SngSmooth }
        & 6 & xz^3+z(\alpha\,x^3+\beta\,x^2y+y^3)+\gamma\,x^4+\delta\,x^3y+\epsilon\,x^2y^2+\zeta\,xy^3+y^4 \\
        \hline\hline

        \sA_1
        & { \tikzsetnextfilename{SngAone} \SngAone }
        & 5 & y{z}^{3}+(\alpha\,{y}^{2}+{x}^{2}){z}^{2}+(\beta\,{y}^{3}+\gamma\,{y}^{2}x+y{x}^{2})z+\delta\,{y}^{4}+\epsilon\,{y}^{3}x\\
        \sA_2
        & { \tikzsetnextfilename{SngAtwo} \SngAtwo }
        & 4 & y{z}^{3}+(\alpha\,{y}^{2}+\beta\,yx+{x}^{2}){z}^{2}+(\gamma\,{y}^{3}+\delta\,{y}^{2}x)z+{y}^{3}x\\
        \sA_3
        & { \tikzsetnextfilename{SngAthree} \SngAthree }
        & 3 & {x}^{2}{z}^{2}+\alpha\,{y}^{2}xz+{y}^{4}+\beta\,{y}^{3}x+\gamma\,{y}^{2}{x}^{2}+y{x}^{3}\\
        \sA_4
        & { \tikzsetnextfilename{SngAfour} \SngAfour }
        & 2 & {x}^{2}{z}^{2}+2\,{y}^{2}xz+{y}^{4}+\alpha\,{y}^{3}x+\beta\,{y}^{2}{x}^{2}+y{x}^{3}\\
        \ctbl \sD_4
        & { \tikzsetnextfilename{SngDfour} \SngDfour }
        &  2 &  ({y}^{2}x+y{x}^{2})z+\alpha\,{y}^{4}+\beta\,{y}^{2}{x}^{2}-{x}^{4}\\
        \sA_5
        & { \tikzsetnextfilename{SngAfive} \SngAfive }
        & 1 & {x}^{2}{z}^{2}+2\,{y}^{2}xz+{y}^{4}-{y}^{2}{x}^{2}+\alpha\,y{x}^{3}\\
        \ctbl \sD_5
        & { \tikzsetnextfilename{SngDfive} \SngDfive }
        &  1 &  y{x}^{2}z-{y}^{4}+\alpha\,{y}^{3}x-{x}^{4}\\
        \sA_6 && 0 & {x}^{2}{z}^{2}+2\,{y}^{2}xz+{y}^{4}-y{x}^{3} \\
        \ctbl \sE_6 &&  0 &  {x}^{3}z-{y}^{4}-\alpha\,{y}^{2}{x}^{2}\ \ \
                    (\alpha=0\text{ or }1) \\
        \hline\hline

        \sA_1^2
        & { \tikzsetnextfilename{SngAoneptwo} \SngAoneptwo }
        & 4 & ({y}^{2}+{x}^{2}){z}^{2}+(\alpha\,y{x}^{2}+\beta\,{x}^{3})z+{y}^{2}{x}^{2}+\gamma\,y{x}^{3}+\delta\,{x}^{4}\\
        \sA_1\,\sA_2
        & { \tikzsetnextfilename{SngAoneAtwo} \SngAoneAtwo }
        & 3 & {y}^{2}{z}^{2}+(\alpha\,y{x}^{2}+{x}^{3})z+{y}^{2}{x}^{2}+\beta\,y{x}^{3}+\gamma\,{x}^{4}\\
        \sA_2^2
        & { \tikzsetnextfilename{SngAtwoptwo} \SngAtwoptwo }
        & 2 & {y}^{2}{z}^{2}+(\alpha\,y{x}^{2}+{x}^{3})z+y{x}^{3}+\beta\,{x}^{4}\\
        \sA_1\,\sA_3
        & { \tikzsetnextfilename{SngAoneAthree} \SngAoneAthree }
        & 2 & {x}^{2}{z}^{2}+(\alpha\,{y}^{2}x+y{x}^{2})z+{y}^{4}+\beta\,{y}^{3}x\\
        \sA_2\,\sA_3
        & { \tikzsetnextfilename{SngAtwoAthree} \SngAtwoAthree }
        & 1 & {x}^{2}{z}^{2}+\alpha\,{y}^{2}xz+{y}^{4}+{y}^{3}x\\
        \sA_1\,\sA_4
        & { \tikzsetnextfilename{SngAoneAfour} \SngAoneAfour }
      & 1 & {x}^{2}{z}^{2}+2\,{y}^{2}xz+{y}^{4}-{y}^{3}x+\alpha\,{y}^{2}{x}^{2}\\
        \sA_2\,\sA_4
        & { \tikzsetnextfilename{SngAtwoAfour} \SngAtwoAfour }
        & 0 & {x}^{2}{z}^{2}+2\,{y}^{2}xz+{y}^{4}-{y}^{3}x\\
        \hline\hline

        \sA_1^3
        & { \tikzsetnextfilename{SngAonepthree} \SngAonepthree }
        & 3 & ({y}^{2}+\alpha\,yx+{x}^{2}){z}^{2}+(\beta\,{y}^{2}x+\gamma\,y{x}^{2})z+{y}^{2}{x}^{2}\\
        \sA_1^2\,\sA_2
        & { \tikzsetnextfilename{SngAoneptwoAtwo} \SngAoneptwoAtwo }
        & 2 & ({y}^{2}+\alpha\,yx+{x}^{2}){z}^{2}+(\beta\,{y}^{2}x-2\,y{x}^{2})z+{y}^{2}{x}^{2}\\
        \sA_1\,\sA_2^2
        & { \tikzsetnextfilename{SngAoneAtwoptwo} \SngAoneAtwoptwo }
        & 1 & ({y}^{2}+\alpha\,yx+{x}^{2}){z}^{2}+(-2\,{y}^{2}x-2\,y{x}^{2})z+{y}^{2}{x}^{2}\\
        \sA_2^3
        & { \tikzsetnextfilename{SngAtwopthree} \SngAtwopthree }
        & 0 & ({y}^{2}-2\,yx+{x}^{2}){z}^{2}+(-2\,{y}^{2}x-2\,y{x}^{2})z+{y}^{2}{x}^{2}\\
        \hline
      \end{array}
    \end{displaymath}
  \caption{Irreducible normal forms of singular plane quartic curves}
  \label{tab:inf}
\end{table}

\begin{table}[htbp]
  \centering
  \renewcommand{\arraystretch}{1.3}
  \setlength\arrayrulewidth{1pt}
  \begin{subtable}[h]{0.7\textwidth}
    \small
    \begin{tabular}{|c@{}c|c|l|}
      \hline
      \text{Type} &
      & \#
      & \multicolumn{1}{c|}{\text{Normal forms}} \\
      \hline\hline

      \ctbl $\,^r\sA_5$
      & { \tikzsetnextfilename{SngRedAfive} \SngRedAfive }
      & 1 & $x\,(x{z}^{2}+(\alpha\,yx+{x}^{2})z+{y}^{3})$\\
      \ctbl $\,^r\sX_9(\tilde{\sA}_7)$
      & { \tikzsetnextfilename{SngRedXnine} \SngRedXnine }
      &  1 &  ${y}^{4}+\alpha\,{y}^{2}{x}^{2}+{x}^{4}$\\
      \ctbl $\,^r\sD_6$
      & { \tikzsetnextfilename{SngRedDsix} \SngRedDsix }
      &  0 &  $x\,({z}^{3}+yxz+{x}^{3})$\\
      $\,^r\sA_7$
      & { \tikzsetnextfilename{SngRedAseven} \SngRedAseven }
      & 0 & $(yz+{y}^{2}+{x}^{2})\,(yz+{x}^{2})$\\
      \ctbl $\,^r\sE_7$
      & { \tikzsetnextfilename{SngRedEseven} \SngRedEseven }
      &  0 &  $x({z}^{3}+x{z}^{2}+y{x}^{2})$\\
      \hline\hline

      $\,^r\sA_1\,\sA_3$
      & { \tikzsetnextfilename{SngRedA1A3} \SngRedAoneAthree }
      & 2 & $x\,(\alpha\,x{z}^{2}+({y}^{2}+\beta\,yx+{x}^{2})z+{y}^{2}x)$\\
      \ctbl $\,^r\sA_1\,\sD_4$
      & { \tikzsetnextfilename{SngRedA1D4} \SngRedAoneDfour }
      &  1 &  $x\,(y{z}^{2}+(yx+\alpha\,{x}^{2})z+{x}^{3})$\\
      $\,^r\sA_3^2$
      & { \tikzsetnextfilename{SngRedAthreeAthree} \SngRedAthreeAthree }
      & 1 & $(yz+{x}^{2})\,(\alpha\,yz+{x}^{2})$\\
      $\,^r\sA_1\,\sA_5$
      & { \tikzsetnextfilename{SngRedA1A5Conic} \SngSngRedAoneAfiveConic }
      & 0 & $({z}^{2}+yx)\,({z}^{2}+yz+yx)$\\
      \ctbl $\,^r\sA_1\,\sA_5$
      & { \tikzsetnextfilename{SngRedA1A5Cubic} \SngRedAoneAfiveCubic }
      &  0 &  $x\,({z}^{3}+yxz+{y}^{2}x) $\\
      \ctbl $\,^r\sA_1\,\sD_5$
      & { \tikzsetnextfilename{SngRedA1D5} \SngRedAoneDfive }
      &  0 &  $x\,(y{z}^{2}+\alpha\,{x}^{2}z+{x}^{3}) \ \ \ \ (\alpha=0\text{ or }1) $\\
      \ctbl $\,^r\sA_1\,\sD_6$
      & { \tikzsetnextfilename{SngRedAoneDsix} \SngRedAoneDsix }
      &  0 &  $yz\,(xz+{y}^{2})$\\
      \ctbl $\,^r\sA_2\,\sA_5$
      & { \tikzsetnextfilename{SngRedAtwoAfive} \SngRedAtwoAfive }
      &  0 &  $x\,({z}^{3}+{y}^{2}x)$\\
      \hline\hline

      $\,^r\sA_1^3$
      & { \tikzsetnextfilename{SngRedAonepthree} \SngRedAonepthree }
      &  3 & $x\,((y+\alpha\,x){z}^{2}+({y}^{2}+\beta\,yx+\gamma\,{x}^{2})z+{y}^{2}x)$\\
      $\,^r\sA_1^2\,\sA_3$
      & { \tikzsetnextfilename{SngRedAoneptwoAthreeCubic} \SngRedAoneptwoAthreeCubic }
      &  1 & $x\,(x{z}^{2}+({y}^{2}+\alpha\,yx)z+{y}^{2}x) $\\
      $\,^r\sA_1^2\,\sA_3$
      & { \tikzsetnextfilename{SngRedAoneptwoAthreeConic} \SngRedAoneptwoAthreeConic }
      &  1 & $((y+x)z+yx)\,((\alpha\,y+x)z+yx)$\\
      \ctbl $\,^r\sA_1^2\,\sD_4$
      & { \tikzsetnextfilename{SngRedAoneptwoDfour} \SngRedAoneptwoDfour }
      &   0 &  $yz\,((y+x)z+yx)$\\
      $\,^r\sA_1\,\sA_2\,\sA_3$
      & { \tikzsetnextfilename{SngRedAoneAtwoAthree} \SngRedAoneAtwoAthree }
      &  0 & $x\,(x{z}^{2}+({y}^{2}+2\,yx)z+{y}^{2}x)$\\
      $\,^r\sA_1\sA_3^2$
      & { \tikzsetnextfilename{SngAoneAthreeptwo} \SngAoneAthreeptwo }
      &  0 & y$z\,(yz+{x}^{2})$\\
      \hline\hline

      $\,^r\sA_1^4$
      & { \tikzsetnextfilename{SngAonepfourCubic} \SngAonepfourCubic }
      &  2 & $x\,((y+\alpha\,x){z}^{2}+({y}^{2}+\beta\,yx)z+{y}^{2}x)$\\
      $\,^r\sA_1^4$
      & { \tikzsetnextfilename{SngAonepfourConic} \SngAonepfourConic }
      &  2 & $((y+x)z+yx)\,((\alpha\,y+\beta\,x)z+yx)$\\
      $\,^r\sA_1^3\,\sA_2$
      & { \tikzsetnextfilename{SngAonepthreeAtwo} \SngAonepthreeAtwo }
      &  1 & $x\,({z}^{3}+\alpha\,y{z}^{2}+{y}^{2}z+{y}^{2}x)$\\
      $\,^r\sA_1^3\,\sA_3$
      & { \tikzsetnextfilename{SngRedAonepthreeAthree} \SngRedAonepthreeAthree}
      &  0 & y$z\,((y+x)z+{x}^{2})$\\
      \ctbl $\,^r\sA_1^3\,\sD_4$
      & { \tikzsetnextfilename{SngRedAonepthreeDfour} \SngRedAonepthreeDfour}
      &   0 &  $xyz \left( z+x \right) $\\
      \hline\hline

      $\,^r\sA_1^5$
      & { \tikzsetnextfilename{SngRedAonepfive} \SngRedAonepfive}
      &  1 & y$z((\alpha\,y+x)z+yx+{x}^{2})$\\
      \hline\hline

      $\,^r\sA_1^6$
      & { \tikzsetnextfilename{SngRedAonepsix} \SngRedAonepsix}
      &  0 & $xyz\,(z+y+x)$\\
      \hline
        \end{tabular}
    \caption{Non-irreducible}
    \label{tab:nonirr}
  \end{subtable}
  \hfill
  \begin{subtable}[h]{0.29\textwidth}
    \small
    \begin{tabular}{|c@{\hspace{3pt}}c|l|}
          \hline
          \text{Type} && \multicolumn{1}{c|}{\text{Normal forms}} \\
          \hline\hline
          \ctbl $\ell^2\,c$
          & { \tikzsetnextfilename{SngNonLtwoC} \SngNonLtwoC}
          & $ {x}^{2}\,((y+x)z+yx)$\\
          \ctbl $\ell^2\,c$
          & { \tikzsetnextfilename{SngNonLtwoCBis} \SngNonLtwoCBis}
          & $ {z}^{2}\,(xz+{y}^{2})$\\
          \ctbl $\ell\,\ell\,\ell^2$
          & { \tikzsetnextfilename{SngNonLLLtwo} \SngNonLLLtwo}
          & $ xy{z}^{2}$\\
          \ctbl $\ell\,\ell\,\ell^2$
          & { \tikzsetnextfilename{SngNonLLLtwoBis} \SngNonLLLtwoBis}
          & $ x{z}^{2}(z+x)$\\
          $c^2$
          & { \tikzsetnextfilename{SngNonCtwo} \SngNonCtwo}
          & $({z}^{2}+yx)^{2}$\\
          \ctbl $\ell^2\,\ell^2$
          & { \tikzsetnextfilename{SngNonLtwoLtwo} \SngNonLtwoLtwo}
          & $ {x}^{2}{z}^{2}$\\
          \ctbl $\ell\,\ell^3$
          & { \tikzsetnextfilename{SngNonLLthree} \SngNonLLthree}
          & $ {x}{z}^{3}$\\
          \ctbl $\ell^4$
          & { \tikzsetnextfilename{SngNonLfour} \SngNonLfour}
          & $ {z}^{4}$\\
          \hline
        \end{tabular}
    \caption{Non-isolated}
    \label{tab:noniso}
  \end{subtable}
  \caption{Non-irreducible normal forms of singular plane quartic curves}
  \label{tab:rnf}
\end{table}

\begin{remark}\label{rem:huiforms}
  Often the singular points of Hui normal forms are exactly of the expected
  type for any choice of parameters $\alpha$, $\beta$ \textit{et
    cetera}. This is the case, for example, for the normal forms of $\sA_1$
  and $\sA_2$ at $(1:0:0)$, or $\sA_1^2$ at this point and at $(0:1:0)$. The
  only way for these quartics to specialise to other singularity types is thus
  to have additional singular points (we use this fact in
  Sec.~\ref{sec:spec-hui-norm}).
  Where this is not the case, Hui explicitly states the values of the
  parameters that change the singularity type. For example, for $\sA_1^3$ he
  indicates that by setting $\alpha$, $\beta$ and $\gamma$ to $\pm 2$, the
  type of the three singular points $(1:0:0)$, $(0:1:0)$ and $(0:0:1)$ becomes
  $\sA_2$.
\end{remark}

\begin{remark}\label{rem:huichar}
  In fact, Hui's manuscript has to be checked with care, but the
  classification and the normal forms do hold in characteristic $p > 7$ for
  all $p$.
  For smaller $p$, one of the issues is that the singularities can have
  another type than those known in characteristic zero. For example, the
  singular point $(0:0:1)$ of the Hui normal form for $\sA_6$,
  ${x}^{2}{z}^{2}+2\,{y}^{2}xz+{y}^{4}-y{x}^{3}$, has Milnor number 9 instead of 6 in characteristic 7.

\end{remark}

Furthermore, Hui studies precisely how a given stratum specialises in strata
of smaller dimensions \cite[chap. 7]{hui79}. We summarise his results on
Fig.~\ref{fig:specialize} for the strata we are interested in, \textit{i.e.}
those which are non-unstable (in the GIT sense), starting from quartics
without singularities (top) up to the most singular quartics (bottom).
Furthermore, we have grouped strata that
are not GIT-stable, \textit{i.e.} with a $\sA_3$ singularity or more, into
three sets, $V$\ref{eq:A4}, $V$\ref{eq:A3} and $V$\ref{eq:rA1A3}, which are
delimited by dashed lines. In terms of invariants, strata within the same set
are indistinguishable from the stratum of smallest dimension.

\begin{figure}[htbp]
  \tikzsetnextfilename{SingSpecial}
  \,\hspace*{-1cm}\begin{tikzpicture}[
    scale=0.9,
    every node/.style={black},
    every path/.style={draw,-latex,line width=1pt},
    edge from parent/.style={draw,-latex,line width=1pt, black},
    level 1/.style = {level distance = 1.25cm},
    level 2/.style = {level distance = 1.25cm, sibling distance = 8cm},
    level 3/.style = {level distance = 2cm, sibling distance = 4cm},
    level 4/.style = {level distance = 2cm, sibling distance = 1.5cm},
    level 5/.style = {level distance = 2.5cm, sibling distance = 1.8cm},
    level 6/.style = {level distance = 3cm, sibling distance = 1cm}
    ]

    \renewcommand{\arraystretch}{0.6}

    \node (M3) {\small $\sA_0$}
    child { node (A1) {\Aone{$\sA_1$}} edge from parent[coltypeI]
      child { node (A1p2) {\Aoneptwo{$\sA_1^2$}} edge from parent[coltypeI]
        child { node (rA1p3) [xshift=1.2cm] {\RAonepthree{$\,^r\sA_1^3$}} edge from parent[coltypeI]
          child { node (rA1p4a) [xshift=-0.5cm] {\RAonepfoura{$\begin{array}{c}
                                    \,^r\sA_1^4\\\ \ \ \ \ \mbox{\tiny $cubic$}
                                  \end{array}$}} edge from parent[coltypeI]
            child { node (rA1p5) [xshift=-1cm] {\RAonepfive{$\,^r\sA_1^5$}}  edge from parent[coltypeI]
              child { node (rA1p6) {\RAonepsix{$\,^r\sA_1^6$}} edge from parent[coltypeI]
              }
            }
          }
          child { node (rA1p4b) {\RAonepfourb{$\begin{array}{c}
                                    \,^r\sA_1^4\\\ \ \ \ \ \mbox{\tiny $conic$}
                                  \end{array}$}} edge from parent[draw=none]
          }
        }
        child { node (A1p3) {\Aonepthree{$\sA_1^3$}} edge from parent[coltypeI]
          child { node (A1p2A2) {\AoneptwoAtwo{$\sA_1^2\sA_2$}} edge from parent[coltypeI]
            child { node (rA1p3A2) [xshift=-0.5cm] {\AonepthreeAtwo{$\,^r\sA_1^3\sA_2$}} edge from parent[draw=none]
              child { node (A2p3) [xshift=-0.5cm] {\Atwopthree{$\sA_2^3$}} edge from parent[draw=none]
              }
            }
            child { node (A1A2p2) [xshift=-0.5cm] {\AoneAtwoptwo{$\sA_1\sA_2^2$}} edge from parent[coltypeI]
              child { node (rA1p3A3) [xshift=-0.5cm] {\RAoneAthree{$\,^r\sA_1^3\sA_3$}}  edge from parent[draw=none]
              }
            }
            child { node (rA1p2A3a) [xshift=-0.5cm] {\RAoneAthree{$\begin{array}{c}
                                        \,^r\sA_1^2\sA_3\\\ \ \ \ \ \mbox{\tiny $cubic$}
                                      \end{array}$}} edge from parent[coltypeII]
              child { node (rA1A3p2) [xshift=-0.5cm, yshift=-1cm]
                {\RAoneAthree{$\,^r\sA_1\sA_3^2$}} edge from parent[draw=none]
              }
            }
            child { node (rA1p2A3b) [xshift=-0.5cm] {\Athree{$\begin{array}{c}
                                        \,^r\sA_1^2\sA_3\\\ \ \ \ \ \mbox{\tiny $conic$}
                                      \end{array}$}} edge from parent[draw=none]
              child { node (rA1A2A3) {\RAoneAthree{$\,^r\sA_1\sA_2\sA_3$}} edge from parent[draw=none]
              }
              child { node (rA1A5) [xshift=0.7cm] {$\,^r\sA_1\sA_5$} edge from parent[draw=none]
              }
            }
          }
        }
        child { node (A1A2) {\AoneAtwo{$\sA_1\sA_2$}} edge from parent[coltypeI]
          child { node (A2p2) [xshift=-0.5cm] {\Atwoptwo{$\sA_2^2$}} edge from parent[coltypeI]
          }
          child { node (rA1A3) {\RAoneAthree{$\,^r\sA_1\sA_3$}}  edge from parent[coltypeII]
            child { node (rA3p2) [xshift=0.2cm,yshift=-1cm] {\Athree{$\,^r\sA_3^2$}} edge from parent[draw=none]
              child { node (rA7c2) [xshift=1cm] {$\,^r\sA_7\text{ or }c^2$}  edge from parent[coltypeI]
              }
            }
          }
          child { node (A1A3) [xshift=1cm] {\Athree{$\sA_1\sA_3$}} edge from parent[coltypeI]
            child { node (A2A3) [xshift=0.5cm] {\Athree{$\sA_2\sA_3$}}  edge from parent[coltypeI]
            }
            child { node (A1A4) {$\sA_1\sA_4$}  edge from parent[coltypeI]
              child { node (A2A4) {$\sA_2\sA_4$} edge from parent[coltypeI]
              }
            }
          }
        }
      }
      child { node (A2) [xshift=-1cm] {\Atwo{$\sA_2$}} edge from parent[coltypeI]
        child { node (A3) [xshift=1.5cm] {\Athree{$\sA_3$}} edge from parent[coltypeI]
          child { node (A4) {$\sA_4$} edge from parent[coltypeI]
            child { node (A5)  {$\sA_5$} edge from parent[coltypeI]
              child { node (A6) {$\sA_6$} edge from parent[coltypeI]
              }
            }
          }
        }
      }
    };
    \draw[line width=0.6pt, coltypeII] (A1p2) to [bend left=5] (A3);
    \draw[coltypeI] (A2)   to (A1A2);
    \draw[line width=0.6pt, coltypeI] (rA1p3.330) to [bend left=5] (rA1A3.120); %
    \draw[coltypeI] (A1p3) to (rA1p4a.40); %
    \draw[coltypeI] (A1p3) to (rA1p4b.50);
    \draw[line width=0.6pt, coltypeII] (A1p3) to [bend left=8] (A1A3);
    \draw[coltypeI] (A1A2) to (A1p2A2); %
    \draw[line width=0.6pt, coltypeII] (A1A2) to [bend left=5] (A4);
    \draw[line width=0.6pt, coltypeI] (A3) to [bend right=10] (rA1A3.60); %
    \draw[coltypeI] (A3) to (A1A3);
    \draw[coltypeI] (rA1p4a.285) to (rA1p3A2.90); %
    \draw[line width=0.6pt, coltypeI] (rA1p4a.310) to [bend left=5] (rA1p2A3a.120);
    \draw[coltypeI] (rA1p4b.260) to (rA1p5); %
    \draw[line width=0.6pt, coltypeI] (rA1p4b.323) to [bend left=8] (rA1p2A3b.90);
    \draw[coltypeI] (A1p2A2) to (rA1p3A2.70); %
    \draw[line width=0.6pt, coltypeII] (A1p2A2.330) to [bend left=5] (A2A3.160);
    \draw[line width=0.6pt, coltypeII] (A1p2A2.345) to [bend left=15] (A1A4.140);
    \draw[line width=0.6pt, coltypeI] (A1p2A2.300) to [bend left=5] (rA1p2A3b.75);
    \draw[coltypeI] (rA1A3.230) to (rA1p2A3a.70); %
    \draw[line width=0.6pt, coltypeII] (A1A3) to [bend left=10] (A5.110); %
    \draw[coltypeII] (A1A3.200) to (rA1p2A3b.60);
    \draw[coltypeI] (A1A3.190) to (rA1p2A3a.50);
    \draw[line width=0.6pt, coltypeI] (A2p2.305) to [bend left=5] (A2A3.120); %
    \draw[line width=0.6pt, coltypeII] (A2p2.345) to [bend left=10] (A5.125);
    \draw[line width=0.6pt, coltypeI] (A2p2.220) to [bend right=5] (A1A2p2.40);
    \draw[coltypeI] (A4) to (A1A4); %
    \draw[coltypeI] (rA1p5) to (rA1p3A3.140); %
    \draw[line width=0.6pt, coltypeI] (rA1p5) to [bend left=15]  (rA1A3p2.130);
    \draw[line width=0.6pt, coltypeI] (rA1p3A2.310) to [bend left=10] (rA1A3p2); %
    \draw[coltypeI] (rA1p3A2) to (rA1p3A3.130);
    \draw[line width=0.6pt, coltypeI] (A1A2p2.195) to [bend right=20] (A2p3.90); %
    \draw[line width=0.6pt, coltypeI] (A1A2p2.290) to (rA1A2A3.100);
    \draw[line width=0.6pt, coltypeII] (A1A2p2.300) to [bend left=10] (rA1A5.140);
    \draw[coltypeI] (A1A2p2) to (rA1p3A3.120);
    \draw[coltypeII] (A1A2p2.280) to (rA1A3p2.95);
    \draw[line width=0.6pt, coltypeII] (A1A2p2.330) to [bend left=5] (A2A4.170);
    \draw[line width=0.6pt, coltypeII] (A1A2p2.320) to [bend left=18] (rA7c2.130);
    \draw[line width=0.6pt, coltypeI] (rA1p3A2) to [bend left=5]  (rA1A2A3.130); %
    \draw[line width=0.6pt, coltypeI] (rA1p2A3a.280) to [bend right=10] (rA1p3A3.110); %
    \draw[coltypeI] (rA1p2A3a.290) to (rA1A3p2);
    \draw[coltypeI] (rA1p2A3a.300) to (rA1A2A3.70);
    \draw[coltypeII] (rA1p2A3b.260) to (rA1A5.120); %
    \draw[line width=0.6pt, coltypeI] (rA1p2A3b.240) to [bend right=10] (rA1p3A3.90);
    \draw[coltypeI] (rA1p2A3b.0) to [bend left=30]  (rA3p2.100);
    \draw[line width=0.6pt, coltypeI] (rA3p2.180) to [bend right=35] (rA1A3p2.60);
    \draw[line width=0.6pt, coltypeII] (A2A3) to [bend left=5] (A6.135); %
    \draw[coltypeI] (A2A3) to (A2A4);
    \draw[line width=0.6pt, coltypeI] (A2A3) to [bend left=17]  (rA1A2A3.50);
    \draw[coltypeI] (A2A3) to [bend right=30] (rA3p2);
    \draw[coltypeII] (A1A4) to  (A6.120); %
    \draw[line width=0.6pt, coltypeI] (A1A4.220) to [bend right=10]  (rA1A5.110);
    \draw[line width=0.6pt, coltypeII] (A1A4) to [bend right=10] (rA7c2.65);
    \draw[line width=0.6pt, coltypeI] (A5) to [bend right=8] (rA1A5.85); %
    \draw[line width=0.6pt, coltypeI] (A5) to [bend right=8] (rA7c2.45);
    \node[above left = -1.3cm and 0.3cm of rA1p6]  (dim0) {\footnotesize Dim. 0};
    \node[above = 2.5cm of dim0]  (dim1) {\footnotesize Dim. 1};
    \node      at (rA1p4b-| dim0) {\footnotesize Dim. 2};
    \node      at (rA1p3 -| dim0) {\footnotesize Dim. 3};
    \node      at (A1p2  -| dim0) {\footnotesize Dim. 4};
    \node      at (A1    -| dim0) {\footnotesize Dim. 5};
    \node      at (M3    -| dim0) {\footnotesize Dim. 6};
    \draw [dashed, semithick] plot [smooth cycle] coordinates {
      (A4.45) (A4.125) (A1A4.190) (rA1A5.190) (rA7c2.260) (A2A4.260)  (A6.0)
    };
    \node[below right = -0.1cm and -0.1cm of A6]  (A4Id) {\scriptsize$V$\ref{eq:A4}};
    \draw [dashdotted, semithick] plot [smooth cycle] coordinates {
      (A3.0) (A3.125) (A1A3.180) (rA1p2A3b.120) (rA1p2A3b.250)  (rA3p2.250)
      (A2A3.10) (A1A3.0) };
    \node[above right = -0.2cm and -0.2cm of A3]  (A3Id) {\scriptsize$V$\ref{eq:A3}};
    \draw [dotted, semithick] plot [smooth cycle] coordinates {
      (rA1A3.30) (rA1A3.125) (rA1p2A3a.120)   (rA1p3A3.180) (rA1A3p2.270)
      (rA1A2A3.320) (rA1A2A3.27) (rA1p2A3a.20) (rA1A3.340) };
    \node[below left = 0.3cm and -1cm of rA1p3A3]  (rA1A3Id) {\scriptsize$V$\ref{eq:rA1A3}};
  \end{tikzpicture}
  \caption{Singularity specialisations (non-unstable strata)
  }
  \label{fig:specialize}
\end{figure}
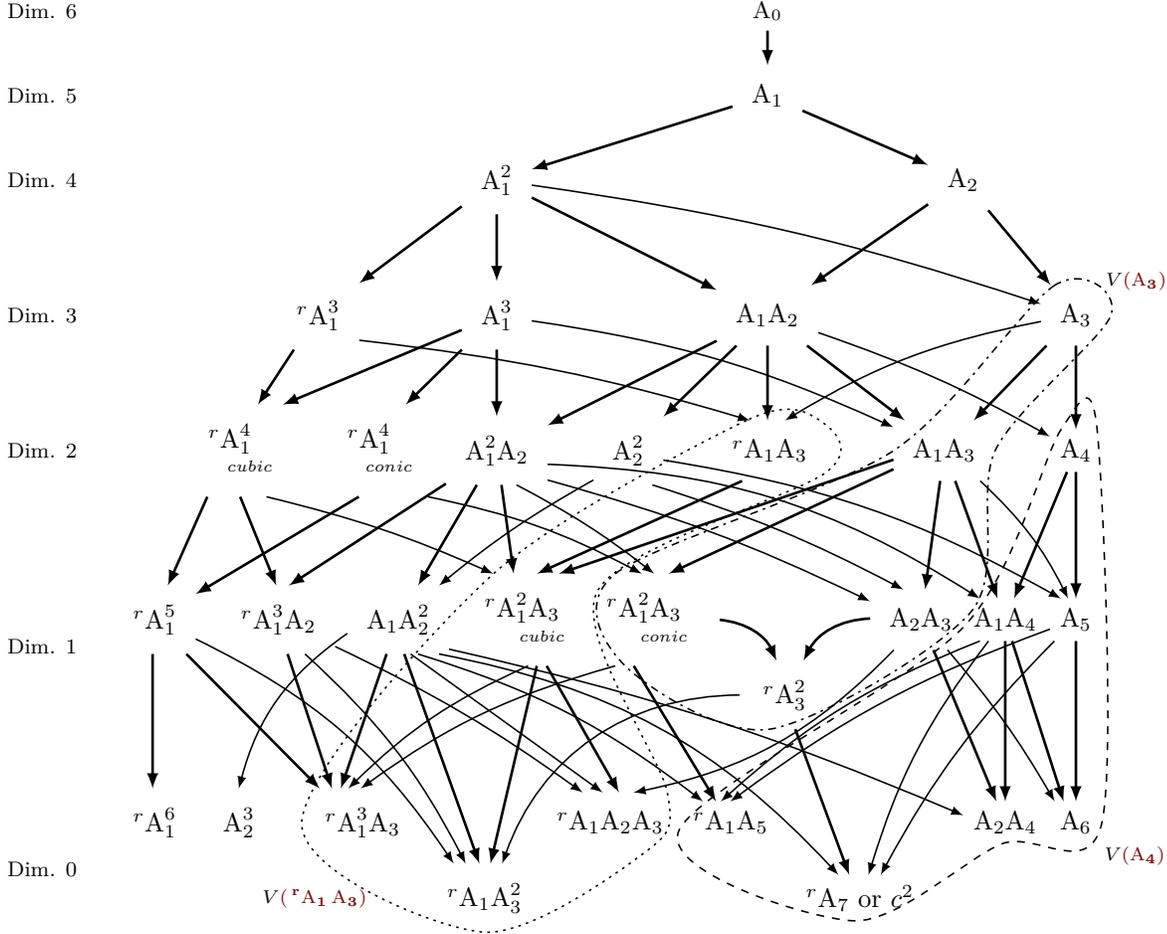

\subsection{Stable reduction of curves}
\label{sec:stable-reduct-curv}

Consider a smooth projective and geometrically connected curve $C$ over a
local field $K$. By
Deligne and Mumford~\cite{DM69}, there exists a finite extension $K'$ of $K$
such that $C \times_K K'$ is the generic fibre of a stable curve $\mathcal C$
over the integral closure of the ring of integers of $K$ inside $K'$.

The singularities of a stable curve are ordinary double points, and its
irreducible components of geometric genus 0 have at least three such double
points, counted with multiplicity.
The \textit{stable reduction} of a curve is the special fibre of such a stable
model, and the \textit{(stable) reduction type} is the stable type of the
stable reduction, including whether the reduction is hyperelliptic or not.

It is possible to obtain all stable types recursively by degeneracy.
In genus~3, without distinction between hyperelliptic and non-hyperelliptic,
this produces a list of 42 types,
which we show in Fig.~\ref{fig:labelledstablemodelgraph} of
Sec.~\ref{sec:stable-reduct-types}.  This graph is derived
from~\cite{BDDLL2023}. It lists the quartic stable reduction types ordered by
degeneration and codimension from top to bottom according to the
stratification of $\overline{\mathcal{M}_3}^{\:\mathrm{DM}}$ they induce.
In addition to small descriptive drawings that represent irreducible
components and their intersections with their genus indicated by the thickness of the component,
we use a compact naming convention where
\texttt{0}, \texttt{1}, \texttt{2} or \texttt{3} refer to curves of that
genus, \texttt{n} means node, \texttt{e} means elliptic curve, \texttt{m}
means multiplicative reduction of elliptic curves, \texttt{X=Y} means that
\texttt{X} and \texttt{Y} intersect at two points, \texttt{Z} is shorthand for two
distinct genus 0 components intersecting at two distinct points or
\texttt{0=0}, and \texttt{X\text{-}\text{-}\text{-}Y} (resp.
\texttt{X\text{-}\text{-}\text{-}\text{-}Y}) means that \texttt{X} and
\texttt{Y} intersect at 3 (resp. 4) points.  The remaining two types are
\texttt{(CAVE)}, 3 lines intersecting at 5 points, and \texttt{(BRAID)}, 4
lines intersecting at 6 points.
When the stable reduction is hyperelliptic, we put a $\mathtt{{H}}$ as
subscript, \textit{e.g.}\ $\mathtt{(1=1)_H}$.

\section{Singularities and invariant strata of plane quartics}
\label{sec:sing-invar-strata}

Let $F$ be a (possibly singular) plane quartic defined over an algebraically
closed field $\overline{K}$ of characteristic~0. Let
$\DO(F) = (I_3:I_6:\cdots:I_{27})$ be its Dixmier-Ohno invariants. Our aim is
to retrieve from them a list of candidates for the singularity type of
$F$. This amounts to determining in each case the equations satisfied by these
invariants.  The simplest example of such results is that a quartic is
singular if and only if its discriminant, \textit{i.e.}\ $I_{27}$, is zero. In
the following sections, we will apply some of these results to get information
on the reduction of a quartic defined over a non-archimedean local
field.\smallskip

Before going into more detail, we can already notice that it is not possible
to derive a tight classification of the singularity
types from the invariants. Already, quartics with singularity multiplicities greater than two are
unstable, so their invariants are identically zero. For those types that correspond
to the grey lines in Tab.~\ref{tab:inf} and Tab.~\ref{tab:rnf}, the question is
thus trivial. In fact, all non-unstable types with an $\sA_4$
singularity or more cannot be separated, as well as those of type $\sA_3$ and
its compounds ($\sA_1\sA_3$, $\sA_2\sA_3$, \textit{etc.}), or of type
$\,^r\sA_1\sA_3$ and its compounds ($\,^r\sA_1^2\sA_3$, $\,^r\sA_1\sA_2\sA_3$,
\textit{etc.})

\begin{example}\label{Ex:diffAtype} The plane quartic $x^2z^2+y^4+yx^3=0$ with
  a singularity of type $\sA_3$ at $(0:0:1)$ and the plane quartic
  $x^2z^2+y^4+y^3z+y^2z^2=0$ with singularities of type $\sA_3$ at $(1:0:0)$ and
  $\sA_1$ at $(0:0:1)$ have the same Dixmier-Ohno invariants $(18 : 0 : 1944 : 648 : 0 : 11664 : 0 : 0 : 0 : 0 : 0 : 0 : 0)$.
\end{example}

\subsection{The algorithm}
\label{sec:algorithm}

Suppose $\mathop\mathrm{char}(K) = 0$. Considering the Dixmier-Ohno invariants as variables, we can define a weighted
projective space $\PP^{12}_{3, 6, \ldots, 27}$ of dimension 13. We consider
the ideals $\ClId{\sA}$ generated by the relations that they satisfy on the
singularity strata $\sA$. These are defined in Sec.~\ref{sec:algebr-char}. By
$V\ClId{\sA}$ we denote the algebraic set defined by these equations. It
corresponds to quartics with singularity type $\sA$, but also possibly to some
other more singular strata below it in the specialisation graph of
Fig.~\ref{fig:specialize}.

The result is an algorithm for characterising singularities, which essentially
consists of checking whether the Dixmier-Ohno invariants of a given plane
quartic lie in $V\ClId{\sA}$ for the different strata $\sA$, starting with the smallest strata.
It yields Alg.~\ref{algo:singularities}. The following theorem states that this algorithm is correct.
\begin{theorem}\label{thm:invtosingtypes}
  Let $C\colon F=0$ be a plane quartic curve over an algebraically closed
  field $\overline{K}$ of characteristic~0. Then its singularity type is
  amongst the types returned by Alg.~\ref{algo:singularities} on input
  $\DO(F)$.
\end{theorem}

\begin{figure}[htbp]
  \centering
  \parbox{0.97\linewidth}{%
    \SetAlFnt{\small\sffamily}%
    \begin{algorithm}[H]%
      \footnotesize
      \caption{Plane quartic singularity types} %
      \label{algo:singularities}%
      \SetKwInOut{Input}{Input} \SetKwInOut{Output}{Output} %
      \Input{Dixmier-Ohno invariants $(I_3 : I_6:\cdots: I_{27})\in\PP^{12}_{3, 6, \ldots, 27}(\overline{K})$ of a
        plane quartic, {char}$(\overline{K})=0$}  %
      \Output{A set of possible singularity types} \BlankLine
      \BlankLine
      \tcp{Easy cases}
      If $I_{27}\ne 0$, then \KwRet{\{ Smooth \}}\;
      If $(I_3 : I_6:\cdots: I_{27}) = 0 $, then \KwRet{\{ Unstable \}}\;
      \BlankLine
      \tcp{Dimension 0}
      If $(I_3 : I_6:\cdots: I_{27})$ is in $V$\ref{eq:rA1p6},
      then \KwRet{$\{\,\,^r\sA_1^6\,\}$}\tcp*{Prop.~\ref{prop:dim0}}%
      If $(I_3 : I_6:\cdots: I_{27})$ is in $V$\ref{eq:A2p3},
      then \KwRet{$\{\,\sA_2^3\,\}$}\;%
      If $(I_3 : I_6:\cdots: I_{27})$ is in $V$\ref{eq:A4}, then
      \KwRet{\{\,$\sA_4$, $\sA_5$, $\sA_6$, $\sA_1\,\sA_4$,
        $\,\sA_2\,\sA_4$, $\,^r\sA_7$, $\,^r\sA_1\,\sA_5$, $c^2$\,\}}\;%
      If
      $(I_3 : I_6:\cdots: I_{27})$ is in $V$\ref{eq:rA1A3}, then
      \KwRet{\{\,$\,^r\sA_1\,\sA_3$, $\,^r\sA_1^2\,\sA_3$, $\,^r\sA_1\,\sA_2\,\sA_3$,
        $\,^r\sA_1\,\sA_3^2$,  $\,^r\sA_1^3\,\sA_3$\,\}}\;%
      \BlankLine
      \tcp{Dimension 1}
      If $(I_3 : I_6:\cdots: I_{27})$ is in $V$\ref{eq:rA1p5},
      then \KwRet{\{\,$\,^r\sA_1^5$ \,\}}\tcp*{Prop.~\ref{prop:dim1}}%
      If $(I_3 : I_6:\cdots: I_{27})$ is in $V$\ref{eq:rA1p3A2},
      then \KwRet{\{\,$\,^r\sA_1^3\,\sA_2$ \,\}}\;
      If $(I_3 : I_6:\cdots: I_{27})$ is in $V$\ref{eq:A1A2p2},
      then \KwRet{\{\,$\sA_1\,\sA_2^2$ \,\}}\;
      If $(I_3 : I_6:\cdots: I_{27})$ is in $V$\ref{eq:A3},
      then \KwRet{\{\,$\sA_3$, $\sA_1\,\sA_3$, $\sA_2\,\sA_3$,
        $\,^r\sA_3^2$, $\,^r\sA_1^2\,\sA_3$ \,\}}\;
      \BlankLine
      \tcp{Dimension 2}
      If $(I_3 : I_6:\cdots: I_{27})$ is in $V$\ref{eq:A1p2A2},
      then \KwRet{\{\,$\sA_1^2\,\sA_2$ \,\}}\tcp*{Prop.~\ref{prop:dim2}}%
      If $(I_3 : I_6:\cdots: I_{27})$ is in $V$\ref{eq:A2p2},
      then \KwRet{\{\,$\sA_2^2$ \,\}}\;
      If $(I_3 : I_6:\cdots: I_{27})$ is in $V$\ref{eq:rA1p4a},
      then \KwRet{\{\,$\,^r\sA_1^4$ (line and cubic) \,\}}\;
      If $(I_3 : I_6:\cdots: I_{27})$ is in $V$\ref{eq:rA1p4b},
      then \KwRet{\{\,$\,^r\sA_1^4$ (two conics) \,\}}\;
      \BlankLine
      \tcp{Dimension 3}
      If $(I_3 : I_6:\cdots: I_{27})$ is in $V$\ref{eq:A1p3},
      then \KwRet{\{\,$\sA_1^3$ \,\}}\tcp*{Prop.~\ref{prop:dim3}}%
      If $(I_3 : I_6:\cdots: I_{27})$ is in $V$\ref{eq:A1A2},
      then \KwRet{\{\,$\sA_1\,\sA_2$ \,\}}\;
      If $(I_3 : I_6:\cdots: I_{27})$ is in $V$\ref{eq:rA1p3},
      then \KwRet{\{\,$\,^r\sA_1^3$ \,\}}\;
      \BlankLine
      \tcp{Dimension 4}
      If $(I_3 : I_6:\cdots: I_{27})$ is in $V$\ref{eq:A2},
      then \KwRet{\{\,$\sA_2$ \,\}}\tcp*{Prop.~\ref{prop:dim4}}%
      If $(I_3 : I_6:\cdots: I_{27})$ is in $V$\ref{eq:A1p2},
      then \KwRet{\{\,$\sA_1^2$ \,\}}\;
      \BlankLine
      \tcp{Dimension 5}
      \KwRet{$\{\sA_1\}$};
    \end{algorithm}
  }
\end{figure}

The proof is based on Prop.~\ref{prop:dim0} to Prop.~\ref{prop:dim5}, given in
Sec.~\ref{sec:algebr-char}. We conjecture that Thm.~\ref{thm:invtosingtypes}
is also true in positive characteristic $p$ for $p$ not too small
(\textit{cf.} Rem.~\ref{rmk:characteristichypotheses}).
\begin{conjecture}[positive characteristic hypothesis]\label{conj:poschar}
  Alg.~\ref{algo:singularities} is valid for any algebraically closed field $\overline{K}$ of
  characteristic $p$ for all $p \neq 2,3,5,7$.
\end{conjecture}

In support of this conjecture, we have automated the verification of
Prop.~\ref{prop:dim0} to Prop.~\ref{prop:dim5} with \textsc{magma} and applied
it for all primes up to $100$. The criterion $p>7$ is necessary for the list
of quartic singularities to be complete (\textit{cf.}  Rem.~\ref{rem:huichar})
and for $\DO(F)$ to be defined differently for $p = 2$, $3$, $5$ and $7$
(\textit{cf.}  Rem.~\ref{rmk:DOcharp}).

\subsection{Algebraic characterisations}
\label{sec:algebr-char}

We present the propositions necessary for the proof of
Thm.~\ref{thm:invtosingtypes} by increasing dimensions.
Here, it is still assumed that the field of definition is an algebraically
closed field $\overline{K}$ of characteristic zero.

We have to understand ``dimension'' as the number of variables that
parameterise the Hui normal forms for a singularity type stratum
(\textit{cf.}\ Tab.~\ref{tab:inf} and Tab.~\ref{tab:rnf}). In the following,
the singularity loci will be viewed as the zero set of an ideal $\ClId{\sA}$
in the projective space $\PP^{12}_{3, 6, \ldots, 27}$.  The dimension of the
singularity loci will then become the Krull dimension of the quotient of the
coordinate ring of the weighted $\PP^{12}$ by the ideal $\ClId{\sA}$, minus
one. Since for singularity type $\sA_3$ (resp. $\sA_4$ and $\,^r\sA_1\sA_3$),
the ideal is of dimension 2 (resp.~1), this singularity is not characterised
as one would expect in Prop.~\ref{prop:dim3} (resp.~Prop.~\ref{prop:dim2})
with the other singularity types of dimension 3 (resp.~2) , but in
Prop.~\ref{prop:dim1} (resp.~Prop.~\ref{prop:dim0}) with the singularity types
of dimension 1 (resp.~0).
In terms of invariants (see Fig.~\ref{fig:specialize}), it's not possible to
distinguish the stratum $\sA_3$ (resp. $\sA_4$ and $\,^r\sA_1\sA_3$) from the
stratum $\,^r\sA_3^2$ (resp. $c^2$ and $\,^r\sA_1\sA_3^2$).

For lack of space, and also because of the technical nature of such results,
we present these 6 propositions in as compact a form as possible.
In particular we do not give the polynomials that define these
ideals explicitly here, but they are available at~\cite[file
\texttt{reductioncriteria.m}]{BGit23}.
The proof of these statements is of a computational nature.
We summarise its main steps in Sec.~\ref{sec:sketches-proof}.

\begin{strata}{0}
  \begin{proposition}\label{prop:dim0}
    Let $C$ be a plane quartic given by a GIT-semi-stable ternary quartic $F$
    over an algebraically closed field $\overline{K}$ of characteristic~0,
    then $C$ is of singularity type
    \begin{itemize}
    \item $\,^r\sA_1^6$ if and only if\\[-0.1cm]
      \begin{tiny}
        \begin{equation*}
          \label{eq:rA1p6}\tag*{\textrm{${\ClId{\,^r\sA_1^6}}$}}%
          \begin{minipage}[t]{1.0\linewidth}\footnotesize
            \centering\begin{math}\displaystyle%
              \DO(F) = \left(1: \frac{\text{-}1}{144}: \frac{1}{9}: \frac{\text{-}1}{3}: \frac{\text{-}4}{81}: \frac{\text{-}1}{18}:
                \frac{\text{-}1}{972}: \frac{1}{36}: \frac{1}{243}: \frac{1}{27}:
                \frac{1}{162}: \frac{\text{-}7}{144}:
                0\right)\,,
            \end{math}
          \end{minipage}
        \end{equation*}
      \end{tiny}
    \item $\sA_2^3$ if and only if\\[-0.1cm]
      \begin{tiny}
        \begin{equation*}
          \label{eq:A2p3}\tag*{\textrm{${\ClId{\sA_2^3}}$}}%
          \begin{minipage}[t]{1.0\linewidth}\footnotesize
            \centering\begin{math}\displaystyle%
              \DO(F) = \left(1: {\frac {\text{-}1}{108}}:{\frac
                  {97}{324}}:{\frac {\text{-}121}{324}}:{\frac {\text{-}
                    325}{2916}}:{\frac {\text{-}47}{324}}:\right.
              \left.{\frac {121}{11664}}:{\frac {7}{1296}}: {\frac
                  {1595}{104976}}:{\frac {985}{34992}}:{\frac {637}{78732}}:{ \frac
                  {\text{-}1057}{314928}}:0 \right)\,,
            \end{math}
          \end{minipage}
        \end{equation*}
      \end{tiny}
    \item $\sA_4$, $\sA_5$, $\sA_6$, $\sA_1\,\sA_4$, $\sA_2\,\sA_4$,
      $\,^r\sA_7$, $\,^r\sA_1\,\sA_5$ or $c^2$ if and only if\\[-0.1cm]
      \begin{tiny}
        \begin{equation*}
          \label{eq:A4}\tag*{\textrm{${\ClId{\sA_4}}$}}%
          \begin{minipage}[t]{1.0\linewidth}\footnotesize
            \centering\begin{math}\displaystyle%
              \DO(F) = \left(
                1:{\frac {1}{180}}:{\frac {49}{36}}:{\frac {49}{60}}:{\frac {343}{
                    1620}}:{\frac {49}{36}}:\right.\left.
                {\frac {1715}{3888}}:{\frac {343}{3600}}:{
                  \frac {2401}{3888}}:{\frac {2401}{10800}}:{\frac {343}{1620}}:{\frac {
                    2401}{720}}:0\right)\,,
            \end{math}
          \end{minipage}
        \end{equation*}
      \end{tiny}
    \item $\,^r\sA_1\,\sA_3$, $\,^r\sA_1^2\,\sA_3$, $\,^r\sA_1\,\sA_2\,\sA_3$,
      $\,^r\sA_1\,\sA_3^2$ or $\,^r\sA_1^3\,\sA_3$ if and only if\\[-0.1cm]
      \begin{tiny}
        \begin{equation*}
          \label{eq:rA1A3}\tag*{\textrm{${\ClId{\,^r\sA_1\,\sA_3}}$}}%
          \begin{minipage}[t]{0.95\linewidth}\footnotesize
            \centering\begin{math}\displaystyle%
              \DO(F) = \left(
                1:{\frac {\text{-}1}{144}}:{\frac {7}{36}}:\frac{1}{6}:{\frac {\text{-}17}{1296}}:{\frac {7
                  }{288}}:\right.\left.
                {\frac {625}{31104}}:{\frac {\text{-}25}{1152}}:{\frac {775}{62208}}:
                {\frac {\text{-}35}{2304}}:{\frac {\text{-}1}{20736}}:{\frac {1}{256}}:0
              \right)\,.
            \end{math}
          \end{minipage}
        \end{equation*}
      \end{tiny}
    \end{itemize}
  \end{proposition}
\end{strata}

\begin{strata}{1}
  \begin{proposition}\label{prop:dim1}
    Let $C$ be a plane quartic given by a GIT-semi-stable ternary quartic $F$
    over an algebraically closed field $\overline{K}$ of characteristic~0 that
    is not in a stratum of dimension 0, then $C$ is of singularity type
    \begin{tiny}
      \begin{flalign}
        \label{eq:rA1p5}\tag*{\textrm{$\ClId{{\,^r\sA_1^5}}$}}
        \begin{minipage}[t]{0.89\linewidth}\normalsize
          \begin{itemize}
          \item $\,^r\sA_1^5$ if and only if $\DO(F)$ is in the algebraic set
            $V$\ref{eq:rA1p5} of
            an ideal~\ref{eq:rA1p5} defined by 13 polynomials whose degree
            profile is $6$, $9$, $12$, $15^2$, $18^2$, $21^3$, $27^2$ and
            $36$,
          \end{itemize}
        \end{minipage}\\
        \label{eq:rA1p3A2}\tag*{\textrm{${\ClId{\,^r\sA_1^3\,\sA_2}}$}}
        \begin{minipage}[t]{0.89\linewidth}\normalsize
          \begin{itemize}
          \item $\,^r\sA_1^3\,\sA_2$ if and only if $\DO(F)$ is in the
            algebraic set of an ideal~\ref{eq:rA1p3A2} defined by 11
            polynomials whose degree profile is $6$, $9$, $12^2$, $15^2$,
            $18^2$, $21^2$ and $27$,
          \end{itemize}
        \end{minipage}\\
        \label{eq:A1A2p2}\tag*{\textrm{${\ClId{\sA_1\,\sA_2^2}}$}}%
        \begin{minipage}[t]{0.89\linewidth}\normalsize
          \begin{itemize}
          \item $\sA_1\,\sA_2^2$ if and only if $\DO(F)$ is in the algebraic
            set $V$\ref{eq:A1A2p2} of an ideal~\ref{eq:A1A2p2} defined by 14
            polynomials whose degree profile is $12$, $15^2$, $18^5$, $21^4$,
            $24$ and $27$,
          \end{itemize}
        \end{minipage}\\
        \label{eq:A3}\tag*{\textrm{${\ClId{\sA_3}}$}}%
        \begin{minipage}[t]{0.89\linewidth}\normalsize
          \begin{itemize}
          \item $\sA_3$, $\sA_1\,\sA_3$, $\sA_2\,\sA_3$, $\,^r\sA_3^2$ or
            $\,^r\sA_1^2\,\sA_3$ if and only if $\DO(F)$ is in the algebraic
            set $V$\ref{eq:A3} of an ideal~\ref{eq:A3} defined by 11
            polynomials whose degree profile is $9$, $12^2$, $15^2$, $18^3$,
            $21^2$ and $27$\,.
          \end{itemize}
        \end{minipage}
      \end{flalign}
    \end{tiny}
  \end{proposition}
\end{strata}

\begin{strata}{2}
  \begin{proposition}\label{prop:dim2}
    Let $C$ be a plane quartic given by a GIT-semi-stable ternary quartic $F$
    over an algebraically closed field $\overline{K}$ of characteristic~0 that
    is not in a stratum of dimension less than or equal to 1, then $C$ is of
    singularity type
    \begin{tiny}
      \begin{flalign}
        \label{eq:A1p2A2}\tag*{\textrm{${\ClId{\sA_1^2\,\sA_2}}$}}%
        \begin{minipage}[t]{0.89\linewidth}\normalsize
          \begin{itemize}
          \item $\sA_1^2\,\sA_2$ if and only if $\DO(F)$ is in the algebraic
            set $V$\ref{eq:A1p2A2} of an ideal~\ref{eq:A1p2A2} defined by 13
            polynomials whose degree profile is $12$, $15$, $18^2$, $21^2$,
            $24^2$, $27^3$ and $30^2$,
          \end{itemize}
        \end{minipage}\\
        \label{eq:A2p2}\tag*{\textrm{${\ClId{\sA_2^2}}$}}%
        \begin{minipage}[t]{0.89\linewidth}\normalsize
          \begin{itemize}
          \item $\sA_2^2$ if and only if $\DO(F)$ is in the algebraic set
            $V$\ref{eq:A2p2} of an ideal~\ref{eq:A2p2} defined by 13
            polynomials whose degree profile is $12$, $15$, $18^3$, $21^3$,
            $24^2$, $27^2$ and $30$,
          \end{itemize}
        \end{minipage}\\
        \label{eq:rA1p4a}\tag*{\textrm{${\ClId{\,^r\sA_1^4\:_{cub}}}$}}%
        \begin{minipage}[t]{0.89\linewidth}\normalsize
          \begin{itemize}
          \item $\,^r\sA_1^4$ (line and cubic) if and only if $\DO(F)$ is in
            the algebraic set $V$\ref{eq:rA1p4a} of an ideal~\ref{eq:rA1p4a}
            defined by 10 polynomials whose degree profile is $6$, $9$, $12$,
            $15^2$, $18$, $21^2$, $27$ and $36$.
          \end{itemize}
        \end{minipage}\\
        \label{eq:rA1p4b}\tag*{\textrm{${\ClId{\,^r\sA_1^4\:_{con}}}$}}%
        \begin{minipage}[t]{0.89\linewidth}\normalsize
          \begin{itemize}
          \item $\,^r\sA_1^4$ (two conics) if and only if $\DO(F)$ is in the
            algebraic set $V$\ref{eq:rA1p4b} of an ideal~\ref{eq:rA1p4b}
            defined by 22 polynomials whose degree profile is $15$, $18^4$,
            $21^5$, $24^4$, $27^4$, $30^2$, $33$ and $36$.
          \end{itemize}
        \end{minipage}
      \end{flalign}
    \end{tiny}
  \end{proposition}
\end{strata}

\begin{strata}{3}
  \begin{proposition}\label{prop:dim3}
    Let $C$ be a plane quartic given by a GIT-semi-stable ternary quartic $F$
    over an algebraically closed field $\overline{K}$ of characteristic~0 that
    is not in a stratum of dimension less than or equal to 2, then $C$ is of
    singularity type
    \begin{tiny}
      \begin{flalign}
        \label{eq:A1p3}\tag*{\textrm{${\ClId{\sA_1^3}}$}}%
        \begin{minipage}[t]{0.89\linewidth}\normalsize
          \begin{itemize}
          \item $\sA_1^3$ if and only if $\DO(F)$ is in the algebraic set
            $V$\ref{eq:A1p3} of an ideal~\ref{eq:A1p3} defined by 26
            polynomials whose degree profile is $24^2$, $27^5$, $30^7$,
            $33^6$, $36^5$ and $39$,
          \end{itemize}
        \end{minipage}\\
        \label{eq:A1A2}\tag*{\textrm{${\ClId{\sA_1\sA_2}}$}}%
        \begin{minipage}[t]{0.89\linewidth}\normalsize
          \begin{itemize}
          \item $\sA_1\sA_2$ if and only if $\DO(F)$ is in the algebraic set
            $V$\ref{eq:A1A2} of an ideal~\ref{eq:A1A2} defined by 9
            polynomials whose degree profile is $12$, $15$, $18^2$, $21^2$,
            $24$, $27$ and $45$,
          \end{itemize}
        \end{minipage}\\
        \label{eq:rA1p3}\tag*{\textrm{${\ClId{\,^r\sA_1^3}}$}}%
        \begin{minipage}[t]{0.89\linewidth}\normalsize
          \begin{itemize}
          \item $\,^r\sA_1^3$ if and only if $\DO(F)$ is in the algebraic set
            $V$\ref{eq:rA1p3} of an ideal~\ref{eq:rA1p3} defined by 9
            polynomials whose degree profile is $6$, $9$, $12$, $15^2$, $18$,
            $21^2$ and $27$.
          \end{itemize}
        \end{minipage}
      \end{flalign}
    \end{tiny}
  \end{proposition}
\end{strata}

\begin{strata}{4}
  \begin{proposition}\label{prop:dim4}
    Let $C$ be a plane quartic given by a GIT-semi-stable ternary quartic $F$
    over an algebraically closed field $\overline{K}$ of characteristic~0 that
    is not in a stratum of dimension less than or equal to 3, then $C$ is of
    singularity type
    \begin{tiny}
      \begin{flalign}
        \label{eq:A2}\tag*{\textrm{${\ClId{\sA_2}}$}}%
        \begin{minipage}[t]{0.89\linewidth}\normalsize
          \begin{itemize}
          \item $\sA_2$ if and only if $\DO(F)$ is in the algebraic set
            $V$\ref{eq:A2} of an ideal~\ref{eq:A2} defined by 8 polynomials
            whose degree profile is $12$, $15$, $18^2$, $21^2$, $24$ and $27$,
          \end{itemize}
        \end{minipage}\\
        \label{eq:A1p2}\tag*{\textrm{${\ClId{\sA_1^2}}$}}%
        \begin{minipage}[t]{0.89\linewidth}\normalsize
          \begin{itemize}
          \item $\sA_1^2$ if and only if $\DO(F)$ is in the complement of
            $V$\ref{eq:A2} with respect to the algebraic set $V$\ref{eq:A1p2}
            of the ideal~\ref{eq:A1p2} defined by 6 polynomials whose degree
            profile is $30$, $33$, $36^2$, $39$ and $42$\,.
          \end{itemize}
        \end{minipage}
      \end{flalign}
    \end{tiny}
  \end{proposition}
\end{strata}

\begin{strata}{5}
  \begin{proposition}\label{prop:dim5}
    Let $C$ be a plane quartic given by a GIT-semi-stable ternary quartic $F$
    over an algebraically closed field $\overline{K}$ of characteristic~0 that
    is not in a stratum of dimension less than or equal to 4, then $C$ is of
    singularity type
    \begin{tiny}
      \begin{flalign}
        \label{eq:A1}\tag*{\textrm{${\ClId{\sA_1}}$}}%
        \begin{minipage}[t]{0.89\linewidth}\normalsize
          \begin{itemize}
          \item $\sA_1$ if and only if $\DO(F)$ is in the algebraic set
            $V$\ref{eq:A1} of the ideal~\ref{eq:A1} $= (I_{27})$\,.
          \end{itemize}
        \end{minipage} \end{flalign}
    \end{tiny}
  \end{proposition}
\end{strata}
\medskip

\begin{remark}
  We may notice that, unlike the other singularity types, the ideal
  \ref{eq:A1p2} used in Prop.~\ref{prop:dim4} is not sufficient to fully
  characterise quartics with at least 2 nodes, because $V$\ref{eq:A1p2} also
  contains the set $V$\ref{eq:A2}.  The reason why we have done this is chiefly a
  computational one: generators for the ideal corresponding to
  $\overline{V\textrm{\ref{eq:A1p2}} \backslash V\textrm{\ref{eq:A2}}}$ do not
  seem to be computable within a reasonable time, whilst it was feasible to find
  generators for $V$\ref{eq:A1p2}.
\end{remark}

\begin{remark}\label{rmk:gbcharp}
  Prop.~\ref{prop:dim5} is also clearly true for any characteristic $p>0$,
  even for $p=2$ or $3$, by scaling the expression of $I_{27}$ known for $\Q$
  to the right power of 2 or 3, ensuring that it is equal to the discriminant
  of $F$. Provided the Dixmier-Ohno invariants generate the invariant algebra,
  Prop.~\ref{prop:dim0} could also be proved in characteristic $p>7$.  We
  believe that the same should hold for the other propositions, yielding
  Conj.~\ref{conj:poschar}. However, such an assertion is difficult to prove, as the proof in characteristic $0$
  is based on Gröbner basis calculations (see Sec.~\ref{sec:sketches-proof}).
  Nonetheless, we invite the reader to compare with Rem.~\ref{rmk:characteristichypotheses}.
\end{remark}

\section{Proof of Theorem~\ref{thm:invtosingtypes}}
\label{sec:sketches-proof}

Throughout this section, it is still assumed that the field of definition is
an algebraically closed field $\overline{K}$ of characteristic zero.

We obtained the generators of the ideal given in
Sec.~\ref{sec:sing-invar-strata} by process of evaluation and interpolation on
randomly chosen Hui normal forms, in the manner of~\cite{LR11}. It is
therefore difficult to derive a direct proof from this.
Instead, the proofs we devise below do not depend on the way we obtained
them. They consist mainly in checking that the only normal forms with
Dixmier-Ohno invariants that satisfy stratum equations for a prescribed types
are of this type.

In order to reduce the number of verifications to be done, we first prove that
the inclusion tree of the algebraic sets defined by the ideals of
Prop.~\ref{prop:dim0} to Prop.~\ref{prop:dim5} is almost the same as the
specialisation tree of the singular quartics given in
Fig.~\ref{fig:specialize}. This is the content of Lem.~\ref{lem:inctree} in
Sec.~\ref{sec:dixmier-ohno-stratum}.

It is then relatively easy to check computationally that the Dixmier-Ohno
invariants of Hui normal forms for a given type generically satisfy the
stratum relations we propose for that type.
The most difficult point is to further verify that if these invariants cancel
the relations of a substratum, this necessarily implies conditions on the
parameters of the normal form so that the forms they define have the expected
singularities for that substratum.
This can be done almost directly, as explained in
Sec.~\ref{sec:other-strata}.

The most difficult case is certainly that of Hui normal forms of type $\sA_1$; to
show that such a quartic necessarily admits other singularities when its
invariants verify the equations of a substratum, in particular the equations
of the $\sA_1^2$ or $\sA_2$ strata.
We take advantage of the fact that these normal forms, whatever their
parameters, always have at least one singularity of exactly $\sA_1$ type
(see Rem.~\ref{rem:huiforms}). We therefore expect them to
have one or more other singularities if their invariants satisfy the relations
of a substratum.
In the proof, we thus use a polynomial in the parameters of the $\sA_1$ family
whose cancellation is equivalent to having at least two singularities. This is
the content of Lem.~\ref{lem:singA1} in Sec.~\ref{sec:spec-hui-norm}. We
treat the quartics of type $\sA_1^2$ and $\sA_2$ in the same way
(Lem.~\ref{lem:singA1p2} and Lem.~\ref{lem:singA2}).

\subsection{Dixmier-Ohno stratum inclusion tree}
\label{sec:dixmier-ohno-stratum}

The following lemma shows that the ideals given in
Sec.~\ref{sec:sing-invar-strata} define an inclusion tree that is almost the
same as Fig.~\ref{fig:specialize}.
\begin{lemma}\label{lem:inctree}
  With the notations of Sec.~\ref{sec:sing-invar-strata}, the algebraic
  sets defined by the 18 ideals
  \begin{displaymath}
    \setlength{\arraycolsep}{0.5cm}
    \begin{array}{ll}
      \text{dim. 0,\ \ \ref{eq:rA1p6}, \ref{eq:A2p3}, \ref{eq:rA1A3}, \ref{eq:A4},} &%
      \text{dim. 3,\ \ \ref{eq:A1p3}, \ref{eq:A1A2}, \ref{eq:rA1p3}} \\
      \text{dim. 1,\ \ \ref{eq:rA1p5}, \ref{eq:rA1p3A2}, \ref{eq:A1A2p2}, \ref{eq:A3},} &%
      \text{dim. 4,\ \ \ref{eq:A2}, \ref{eq:A1p2},} \\
      \text{dim. 2,\ \ \ref{eq:A1p2A2}, \ref{eq:A2p2}, \ref{eq:rA1p4a}, \ref{eq:rA1p4b},}&%
      \text{dim. 5,\ \ \ref{eq:A1}.} \\
    \end{array}
  \end{displaymath}
  define an inclusion tree faithful to the quartic singularity type
  specialisation tree, with the notable exception of $V$\ref{eq:A2}
  $\subset V$\ref{eq:A1p2}.
\end{lemma}

The proof consists in checking that these 18 algebraic sets intersect as
expected, which amounts to computing radicals of sums of two ideals according
to the usual correspondence between algebraic sets and polynomial ideal
(\emph{Hilbert's Nullstellensatz}),
\begin{center}
  \begin{tabular}{ccc}
    ideals & $\longleftrightarrow$ & algebraic sets\,,\\[0.1cm]
    $\ClId{\sA} + \ClId{\sA'}$ & $\longrightarrow$ & $V(\,\ClId{\sA} + \ClId{\sA'}\,) = V\ClId{\sA} \cap V\ClId{\sA'}$\,,\\
    $\Id(V \cap V') = \Rad(\Id(V)+\Id(V'))$ & $\longleftarrow$ & $V \cap V'\,.$
  \end{tabular}
\end{center}
\smallskip

We proceed by increasing dimension. Given a dimension $d$, we check that the pairwise intersections of the algebraic sets of ideals of dimension $d$ with the algebraic sets of ideals of
dimension $d-1$, or other ideals of dimension $d$,  decompose as illustrated in
Fig.~\ref{fig:specialize}.

For instance, the two by two intersections of the algebraic sets
of~\ref{eq:rA1p6}, \ref{eq:A2p3}, \ref{eq:rA1A3} and~\ref{eq:A4} are all equal
to the unstable locus $(O)=(I_3, I_6, \ldots, I_{27})$. If we now consider~\ref{eq:rA1p5}, we can similarly check that we have %
$V$\ref{eq:rA1p5}$ \,\cap\, V$\ref{eq:A4}$ = V(O)$\,, %
$V$\ref{eq:rA1p5}$ \,\cap\, V$\ref{eq:rA1A3}$ = V$\ref{eq:rA1A3}\,, %
$V$\ref{eq:rA1p5}$ \,\cap\, V$\ref{eq:A2p3}$ = V(O)$ %
and %
$V$\ref{eq:rA1p5}$ \,\cap\, V$\ref{eq:rA1p6}$ = V$\ref{eq:rA1p6}\,.
When the dimension increases, more complex situations can occur, such as
for example
\begin{center}
  $V$\ref{eq:A3}$ \,\cap\, V$\ref{eq:rA1p5}$ = V$\ref{eq:rA1A3}\,, %
  \text{ and }
  $V$\ref{eq:A3}$ \,\cap\, V$\ref{eq:A1A2p2}$ = V$\ref{eq:A4}$\,\cup\,V$\ref{eq:rA1A3}\,. %
\end{center}
\smallskip

From a computational point of view, the sum of two ideals is obtained by
joining generators, from which a basis can be extracted by a Gr\"obner basis
calculation.
It takes a handful of seconds to compute in \textsc{magma} a Gr\"obner basis of
those ideals, over $\Q$, for the graded reverse lexicographical (or
``grevlex'') order $I_3 < I_6 < \cdots < I_{27}$ with weights 3, \ldots , 27.

\begin{remark}
  Upon performing the same computations in positive characteristic $p>0$, we
  have observed that, depending on $p$, some of the expected inclusions may
  not occur. For example, $V\ref{eq:A3} \cap\, V\ref{eq:A1A2p2}$ is equal to
  $V\ref{eq:A4} \cup V\ref{eq:rA1A3}$ over $\Q$, while we only have
  $V\ref{eq:A4}$ modulo $p=11$.
\end{remark}

\subsection{Strata of small dimensions}
\label{sec:other-strata}

We first consider Prop.~\ref{prop:dim0} to Prop.~\ref{prop:dim3}.
We can already notice that it is immediate to check that a Hui normal form for
those types $\sA$ have their Dixmier-Ohno invariants in $V\ClId{\sA}$.
The main task is therefore to show that there is no quartic with singularity
type $\sA$ whose invariants are in the algebraic set defined by one of the
substrata below $\sA$ in the specialisation graph on
Fig.~\ref{fig:specialize}.

Since the dimensions are sufficiently small here, we can follow a direct
approach. Suppose we start with a Hui normal form $F_\sA$ for some
type $\sA$ such that $\DO(F_\sA)\in V\ClId{\sA'} \subseteq V\ClId{\sA}$, where $\sA'$ is a substratum which is more degenerate than $\sA$. Then we want to show that the singularity type of $F_\sA$ is more degenerate than just $\sA$.
We start by computing a decomposition into primary ideals of the ideal
obtained from the generators of $\ClId{\sA'}$ that we evaluate in
$\DO(F_\sA)$. We then reduce the equations which define the locus of the
singular points of $F_\sA$ modulo each component of this decomposition.  By
further Gr\"obner basis calculations, applied this time to each of these
polynomial systems, we determine explicitly the singular points of $F_\sA$.
It remains to quantify the type of each of them to finally
ensure that the singularity type of forms $F_\sA$ whose invariants are
restricted to V$\ClId{\sA'}$ cannot remain $\sA$.

Most of these computations can be easily done by the \textsc{magma}
commutative algebra engine, except the determination of the type of a singular
point. This part relies on another native \textsc{magma} package, developed by
M.~Harrison in full generality for schemes.

\begin{example}
  As an illustration, we take the example of a normal form for $\sA_1^3$,
  \begin{displaymath}
    F_{\alpha,\beta,\gamma}(x,y,z)
    =({y}^{2}+\alpha\,yx+{x}^{2}){z}^{2}+(\beta\,{y}^{2}x+\gamma\,y{x}^{2})z+{y}^{2}{x}^{2}\,,
  \end{displaymath}
  that we specialise to~\ref{eq:rA1p4b}.
  The decomposition in primary ideals of ~\ref{eq:rA1p4b} evaluated at
  $\DO(F_{\alpha,\beta,\gamma})$ yields only one radical ideal,
  \begin{displaymath}
    \alpha^2 + \beta^2 + \gamma^2 - \alpha\,\beta\,\gamma - 4 \,=\,0\,,
  \end{displaymath}
  and modulo this ideal, the form $F_{\alpha,\beta,\gamma}$ has generically 4
  singular points,
  \begin{multline*}
    (1 : 0 : 0)\,, (0 : 1 : 0)\,, (0 : 0 : 1) \text{ and }\\
    \left(\,(2\,\beta-\alpha\,\gamma)\,(\beta^2-4): (2\,\gamma-\alpha\,\beta)\,(\gamma^2-4): (\gamma^2-4)\,(\beta^2-4)\,\right)\,.
  \end{multline*}
  It remains to determine the singularity type of these points, and it turns
  out to be $\sA_1$.  In other words, we proved that a normal form for
  $\sA_1^3$ specialised to~$V$\ref{eq:rA1p4b} is generically of type
  $\sA_1^4$. Whether it is $\,^r\sA_1^4\:_{con}$ or $\,^r\sA_1^4\:_{cub}$ does
  not really matter here, the important point is that
  $F_{\alpha,\beta,\gamma}$ is no longer of type $\sA_1^3$.

  Note that for particular values of the parameters, the fourth singular point
  can collide with one of the first three. But we can check that one of the
  singularities is then more complex than $\sA_1$. Typically, for
  $\beta=\gamma=0$, and thus $\alpha=\pm 2$, it turns out that
  $F_{\alpha,\beta,\gamma}$ is of type $\sA_1^2\sA_3$.
\end{example}

\subsection{Specialisations of Hui normal forms for
  \texorpdfstring{$\sA_1$}{A1}, \texorpdfstring{$\sA_1^2$}{A1\^{}2} and
  \texorpdfstring{$\sA_2$}{A2} }
\label{sec:spec-hui-norm}
The method described in Sec.~\ref{sec:other-strata} does not work for the strata $\sA_1$, $\sA_1^2$, and $\sA_2$, because the Gr\"obner basis computations needed for the method did not finish within a reasonable time. In this section, we discuss a different method that we used to prove Prop.~\ref{prop:dim4} and \ref{prop:dim5}, \textit{i.e.}\ the correctness of the algorithm for the strata of dimensions 4 and 5.

We first notice that all normal forms for $\sA_1$ have more than one
singularity, or not, depending on whether a certain polynomial in its
parameters is zero, or not. In particular, because of the special form of
these quartic equations, the node singularity cannot degenerate further. It
yields this technical lemma.

\begin{lemma}\label{lem:singA1}
  Let
  \begin{displaymath}
    F_{\sA_1}(x,y,z) =
    y{z}^{3}+(\alpha\,{y}^{2}+{x}^{2}){z}^{2}+(\beta\,{y}^{3}+\gamma\,{y}^{2}x+y{x}^{2})z+\delta\,{y}^{4}+\epsilon\,{y}^{3}x
  \end{displaymath}
  be a Hui normal form for $\sA_1$ over an algebraically closed
  field $\overline{K}$ of characteristic~0, then $F_{\sA_1}(x,y,z)$ has at least two
  singularities if and only if $\alpha$, $\beta$, $\gamma$, $\delta$ and
  $\epsilon$ are a zero of a polynomial
  $\iota_{\sA_1}(\alpha, \beta, \gamma, \delta, \epsilon)$ of degree 6 in
  $\alpha$, $\beta$ and $\delta$, degree 8 in $\epsilon$, and degree 9 in
  $\gamma$ (cf.~\cite[file
\texttt{G3SingularProof.m}]{BGit23} for its expression).
  In all cases, $F_{\sA_1}$ has at least one $\sA_1$-singularity.
\end{lemma}
\begin{proof}
  Let $(x:y:z)\in\PP^2$ be a singular point of $F_{\sA_1}$. We thus have
  \begin{displaymath}
    \left\{
      \begin{array}{rcl}
        0&=&x^2\,y + 2\,x^2\,z + \gamma\,x\,y^2 + \beta\,y^3 + 2\,\alpha\,y^2\,z + 3\,y\,z^2\,,\\
        0&=&x^2\,z + 3\,\epsilon\,x\,y^2 + 2\,\gamma\,x\,y\,z + 4\,\delta\,y^3 +
             3\,\beta\,y^2\,z + 2\,\alpha\,y\,z^2 + z^3\,,\\
        0&=&2\,x\,y\,z + 2\,x\,z^2 + \epsilon\,y^3 + \gamma\,y^2\,z\,,\\[0.1cm]
        0&=&x^2\,y\,z + x^2\,z^2 + \epsilon\,x\,y^3 + \gamma\,x\,y^2\,z +
             \delta\,y^4 + \beta\,y^3\,z + \alpha\,y^2\,z^2 + y\,z^3\,.
      \end{array}\right.
  \end{displaymath}
  Let us first assume that $y=0$, and thus
  $0 = 2\,x^2\,z = x^2\,z+z^3 = 2\,x\,z^2 = x^2\,z^2$. As expected, it yields
  the singular point $(1:0:0)$ which is of type $\sA_1$ whatever the
  parameters $\alpha$, $\beta$, $\gamma$, $\delta$ and $\epsilon$ are.

  Now, assume that $y\neq 0$, \textit{i.e.}\ without loss of generality
  $y=1$. By eliminating $x$ and $z$ in the system of equations thanks to a
  Gr\"obner basis computation, we arrive at a single equation, which
  corresponds to $\iota_{\sA_1}(\alpha, \beta, \gamma, \delta,
  \epsilon)$. Restricted to parameters $\alpha$, \ldots, $\epsilon$ that are
  zeros of this polynomial, the form $F_{\sA_1}$ has one or several
  singular points of the form $(x:1:z)$. Otherwise, the only singular point of
  $F_{\sA_1}$ is $(1:0:0)$.
\end{proof}

Let us now return to our original problem, \textit{i.e.}\ to show that there is
no quartic with a unique singularity $\sA_1$ whose invariants are in the
algebraic set defined by one of the ideals $\ClId{\sA'}$ of
Sec.~\ref{sec:sing-invar-strata}. So, let us evaluate the generators of
those ideals into the Dixmier-Ohno invariants of the Hui normal form. The
obtained polynomials define an ideal in
$\Q[\alpha, \beta, \gamma, \delta, \epsilon]$ of which we can compute a
Gr\"obner basis, this time for the \textit{grevlex} order
$\alpha < \epsilon < \gamma < \delta < \beta$. It remains then to reduce
$\iota_{\sA_1}(\alpha, \beta, \gamma, \delta, \epsilon)$ modulo this basis and
to check that the result is indeed $0$.

We made these calculations for the 17 ideals $\ClId{\sA'} = \ref{eq:A1p2}$,
$\ref{eq:A2}\ldots$ and check modulo each of them that
$\iota_{\sA_1}(\alpha, \beta, \gamma, \delta, \epsilon)$ is indeed zero. In
\textsc{magma} this is a matter of seconds, except for a very small number
of cases. The longest calculation is for~\ref{eq:A1A2p2}, which takes half an
hour on a standard laptop.

We can verify in the same way that a quartic of type $\sA_1^2$, or $\sA_2$,
cannot have its Dixmier-Ohno invariants in the algebraic set of one of the
ideals below~$V$\ref{eq:A1p2}, or~$V$\ref{eq:A2}. For this we need an
equivalent of Lem.~\ref{lem:singA1}. These are the two technical lemmas that
follow. We omit their proofs, because they are quite similar to that of
Lem.~\ref{lem:singA1}.

\begin{lemma}\label{lem:singA1p2}
  Let
  \begin{math}
    F(x,y,z) =
    ({y}^{2}+{x}^{2}){z}^{2}+(\alpha\,y{x}^{2}+\beta\,{x}^{3})z+{y}^{2}{x}^{2}+\gamma\,y{x}^{3}+\delta\,{x}^{4}
  \end{math}
  be a Hui normal form for $\sA_1^2$ over an algebraically closed
  field $\overline{K}$ of characteristic~0. Then $F(x,y,z)$ has at least three
  singularities if and only if $\alpha$, $\beta$, $\gamma$ and $\delta$ are a
  zero of a polynomial $\iota_{\sA_1^2}(\alpha, \beta, \gamma, \delta)$ of
  degree 5 in $\delta$, degree 6 in $\beta$ and $\gamma$, and degree 8 in
  $\alpha$ (cf.~\cite[file
\texttt{G3SingularProof.m}]{BGit23} for its expression).
  In all cases, $F$ has at least two $\sA_1$-singularities. %
\end{lemma}

\begin{lemma}\label{lem:singA2}
  Let
  \begin{math}
    F(x,y,z) =
    y{z}^{3}+(\alpha\,{y}^{2}+\beta\,yx+{x}^{2}){z}^{2}+(\gamma\,{y}^{3}+\delta\,{y}^{2}x)z+{y}^{3}x
  \end{math}
  be a Hui normal form for $\sA_2$ over an algebraically closed
  field $\overline{K}$ of characteristic~0. Then $F(x,y,z)$ has at least two
  singularities if and only if $\alpha$, $\beta$, $\gamma$ and $\delta$ are a
  zero of a polynomial $\iota_{\sA_2}(\alpha, \beta, \gamma, \delta)$ of
  degree 5 in $\alpha$ and $\gamma$, and degree 7 in $\beta$ and $\delta$
  (cf.~\cite[file
\texttt{G3SingularProof.m}]{BGit23} for its expression).
  In all cases, $F$ has at least one $\sA_2$-singularity. %
\end{lemma}

Finally, once we check by an easy calculation that the Hui normal forms for
$\sA_1^2$ and $\sA_2$ have their respective Dixmier-Ohno invariants in
$V$\ref{eq:A1p2} and $V$\ref{eq:A2}, we know for sure that
Prop.~\ref{prop:dim4} is valid.

\section{From the singularity type to the reduction type}
\label{sec:quart-stable-reduct}

The Hui stratification of plane quartics has been considered in previous
sections. Unfortunately, this stratification does not behave well under the
action of (the group-scheme)~$\operatorname{SL}_3$. Indeed, if we were to try
to construct a quotient space for $\overline{\mathcal{F}}_{3,4}^{\:\mathrm{Hui}}$ by the
action of $\operatorname{SL}_3$, this space would not satisfy the valuative
criterion for properness~\cite[\href{https://stacks.math.columbia.edu/tag/03IX}{Tag 03IX}]{stacks-project}, as shown in Ex.~\ref{ex:5_1}.
Nevertheless, we manage in Sec.~\ref{sec:sing-invar-strata} to partially
relate the Hui stratification and the GIT compactification.

In this section, we aim to characterise the stable reduction type of a curve
in terms of the singularities of GIT-semi-stable quartics, see Thm.~\ref{thm:GIT-MD}. Using the results in Sec.~\ref{sec:sing-invar-strata},
Corollary~\ref{cor:stabsingtypes} relates then Dixmier-Ohno invariants and
(Deligne-Mumford stable) reduction types.

\subsection*{Notation}
Throughout this section, $K$ is a non-archimedean local field with valuation $v$ of residue
characteristic $p \neq 2$, $3$, $5$, $7$. Let $R$ denote its valuation
ring, $\pi$ a uniformiser and $k$ its residue field.

\begin{example}\label{ex:5_1}
  Let us consider the smooth plane quartic
  $\pi^5z^4+x^2z^2+\pi^2y^2z^2+\pi y^3z+y^4+yx^3=0$. The
  $\operatorname{SL}_3(K)$-change of variables
  $(x:y:z)\mapsto(\pi^{-1}x:y:\pi z)$ yields the model
  $\pi z^4+x^2z^2+y^2z^2+y^3z+y^4+\pi^3yx^3=0$. But these
  two $K$-isomorphic smooth plane quartics are not $\mathcal{O}_K$-isomorphic, because their special fibres are the the curves in
  Ex.~\ref{Ex:diffAtype}, which have different singularity types.
\end{example}

\subsection{Stable reduction types and invariants}
\label{sec:stable-reduct-types}

In this section we relate the singularity type and the (stable) reduction type of a plane quartic. We refer to Sec.~\ref{sec:stable-reduct-curv} for the notation for representing stable curves in Fig.~\ref{fig:labelledstablemodelgraph}.
\smallskip

\renewcommand{\ifcolored}{true}
\begin{figure}[htbp]
  \resizebox{1.0\linewidth}{0.5\textheight}{%
    \tikzsetnextfilename{singularitiesandstablemodelgraph}
    \begin{tikzpicture}[scale=1.4,g3lattice,%
      legend/.style={align=left,text width=2cm}]

      \node[above] at (6,9) (3) {\stbtype{\Smooth{\T}}{3}};

      \node[above] at (4,7.5) (2n) {\stbtype{\Aone{\Dn}}{2n}};
      \node[above] at (8,7.5) (2e) {\stbtype{\Atwo{\DeU}}{2e}};

      \draw[lin]
      (3) to (2n)
      (3) to (2e) ;

      \node[above] at (2,6) (1nn) {\stbtype{\Aoneptwo{\Unn}}{1nn}};
      \node[above] at (5,6) (1=1) {\stbtypehyp{\Athree{\UeU}}{1=1}};
      \node[above] at (6,6) (2m) {\stbtype{\Atwo{\DeZn}}{2m}};
      \node[above] at (7,6) (1ne) {\stbtype{\AoneAtwo{\UeUn}}{1ne}};
      \node[above] at (10,6) (1ee) {\stbtype{\Atwoptwo{\UeUeU}}{1ee}};

      \draw[lin]
      (2n) to (1nn)
      (2n) to (1=1)
      (2n) to[bend left=5] (1ne.130)
      (2n) to[bend left=5] (2m)
      (2e) to (1ne)
      (2e) to[bend right=5] (2m)
      (2e) to (1ee) ;

      \node[above] at (0,4.5) (0nnn) {\stbtype{\Aonepthree{\Znnn}}{0nnn}};
      \node[above] at (2.8,4.5) (1---0) {\stbtype{\RAonepthree{\UeeeZ}}{1\text{-}\text{-}\text{-}0}};
      \node[above] at (3.5,4.5) (1=0n) {\stbtypehyp{\Athree{\UeZn}}{1=0n}};
      \node[above] at (4.4,4.5) (0nne) {\stbtype{\AoneptwoAtwo{\ZeUnn}}{0nne}};
      \node[above] at (5.3,4.5) (1nm) {\stbtype{\AoneAtwo{\ZneUn}}{1nm}};
      \node[above] at (7.3,4.5) (1=0e) {\stbtypehyp{\Athree{\UeZeU}}{1=0e}};
      \node[above] at (8,4.5) (1me) {\stbtype{\Atwoptwo{\UneZeZ}}{1me}};
      \node[above] at (9,4.5) (0nee) {\stbtype{\AoneAtwoptwo{\ZeUneZ}}{0nee}};
      \node[above] at (12,4.5) (0eee) {\stbtype{\Atwopthree{\UeUeUeZ}}{0eee}};

      \draw[lin]
      (1nn) to (0nnn)
      (1nn) to (1---0.110)
      (1nn) to (1=0n.110)
      (1nn) to[bend left=5] (0nne.120)
      (1nn.350) to[bend left=10] (1nm)
      (1=1.230) to[bend right=10] (1=0n.85)
      (1=1.310) to[bend left=15] (1=0e)
      (1ne.210) to[bend right=10] (0nne.50)
      (1ne) to (1nm)
      (1ne) to (1=0e)
      (1ne) to[bend left=5] (1me.90)
      (1ne.340) to[bend left=5] (0nee.90)
      (2m.290) to[bend left=15] (1me.120)
      (1ee) to[bend right=10] (1me.70)
      (1ee) to (0nee) ;

      \draw[lin,shorten <=0.6em] (2m.270) to (1nm.50);
      \draw[lin,shorten >=0.6em] (1ee) to (0eee);

      \node[above] at (0.8,3) (0----0) {\stbtype{\RAonepfourb{\DNA}}{0\text{-}\text{-}\text{-}\text{-}0}};
      \node[above] at (1.6,3) (0---0n) {\stbtype{\RAonepfoura{\UeeeUn}}{0\text{-}\text{-}\text{-}0n}};
      \node[above] at (2.4,3) (0n=0n) {\stbtypehyp{\Athree{\ZneZn}}{0n=0n}};
      \node[above] at (3.2,3) (0nnm) {\stbtype{\AoneptwoAtwo{\UnneUn}}{0nnm}};
      \node[above] at (4.5,3) (Z=1) {\stbtypehyp{\RAoneAthree{\ZZeU}}{Z=1}};
      \node[above] at (5.2,3) (0---0e) {\stbtype{\AonepthreeAtwo{\UeeeUeZ}}{0\text{-}\text{-}\text{-}0e}};
      \node[above] at (5.9,3) (1=0m) {\stbtypehyp{\Athree{\UeZeZn}}{1=0m}};
      \node[above] at (6.6,3) (0n=0e) {\stbtypehyp{\RAoneAthree{\ZneZeU}}{0n=0e}};
      \node[above] at (7.3,3) (1mm) {\stbtype{\Atwoptwo{\ZneUeZn}}{1mm}};
      \node[above] at (8.2,3) (0nme) {\stbtype{\AoneAtwoptwo{\ZneZneU}}{0nme}};
      \node[above] at (9.8,3) (0e=0e) {\stbtypehyp{\Athree{\UeZZeU}}{0e=0e}};
      \node[above] at (10.7,3) (0mee) {\stbtype{\Atwopthree{\ZneUeUeZ}}{0mee}};

      \draw[lin]
      (0nnn) to (0----0.130)
      (0nnn) to[bend left=10] (0---0n.100)
      (0nnn) to[bend left=10] (0n=0n.100)
      (0nnn) to[bend left=15] (0nnm.100)
      (1---0.240) to[bend right=10] (0---0n.80)
      (1---0.298) to[bend left=10] (Z=1)
      (1---0.303) to[bend left=10] (0---0e.130)
      (1=0n.310) to[bend left=10] (0n=0e.135)
      (1=0n.300) to[bend left=10] (1=0m.130)
      (0nne.235) to[bend right=5] (0nnm)
      (0nne.320) to[bend left=10] (0nme.130)
      (1nm.220) to[bend right=10] (0nnm)
      (1nm.320) to[bend left=15] (0nme.110)
      (1=0e) to[bend right=10] (1=0m)
      (1=0e) to (0n=0e.80)
      (1=0e.305) to[bend left=10] (0e=0e.130)
      (1me) to (1mm.80)
      (1me.310) to[bend left=15] (0mee.110)
      (0nee) to[bend left=15] (0mee.90)
      (0eee) to (0mee) ;

      \draw[lin,shorten <=1em] (1=0n.250) to (0n=0n.70) ;
      \draw[lin,shorten <=1em] (1=0n.265) to (Z=1.120) ;
      \draw[lin,shorten >=-1em] (0nne.300) to[bend left=10] (0n=0e.120) ;
      \draw[lin,shorten <=0.5em] (0nne.295) to (0---0e) ;
      \draw[lin,shorten <=1.5em] (1nm.245) to[bend left=10] (1mm.100) ;
      \draw[lin,shorten <=0.2em] (1nm.270) to (1=0m.100) ;
      \draw[lin,shorten <=0.7em] (1me) to (0nme) ;
      \draw[lin,shorten <=0.7em] (0nee.270) to (0nme) ;
      \draw[lin,shorten <=0.7em] (0nee.280) to (0e=0e) ;

      \node[above] at (2.6,1.5) (CAVE) {\stbtype{\RAonepfive{\CAVE}}{CAVE}};
      \node[above] at (3.4,1.5) (Z=0n) {\stbtypehyp{\RAoneAthree{\ZZeUn}}{Z=0n}};
      \node[above] at (4.1,1.5) (0---0m) {\stbtype{\AonepthreeAtwo{\ZeeeZeZn}}{0\text{-}\text{-}\text{-}0m}};
      \node[above] at (4.8,1.5) (0n=0m) {\stbtypehyp{\Athree{\ZneZeZn}}{0n=0m}};
      \node[above] at (5.7,1.5) (0nmm) {\stbtype{\AoneAtwoptwo{\ZneZneZn}}{0nmm}};
      \node[above] at (7.4,1.5) (Z=0e) {\stbtypehyp{\RAoneAthree{\ZZeZeU}}{Z=0e}};
      \node[above] at (8.2,1.5) (0m=0e) {\stbtypehyp{\Athree{\UeZZeZn}}{0m=0e}};
      \node[above] at (9,1.5) (0mme) {\stbtype{\Atwopthree{\ZneZneUeZ}}{0mme}};

      \draw[lin]
      (0----0.300) to (CAVE)
      (0---0n.280) to (CAVE)
      (0---0n.295) to[bend left=10] (Z=0n)
      (0---0n.305) to[bend left=10] (0---0m.140)
      (0n=0n.270) to[bend left=10] (Z=0n.110)
      (0n=0n.290) to[bend left=10] (0n=0m)
      (Z=1.250) to[bend right=5] (Z=0n.90)
      (Z=1.300) to[bend left=10] (Z=0e)
      (0nnm) to (0---0m.115)
      (0nnm.310) to[bend left=10] (0n=0m.110)
      (0nnm.320) to[bend left=10] (0nmm.100)
      (0---0e.260) to (0---0m.70)
      (0---0e.280) to[bend left=10] (Z=0e.135)
      (1=0m.260) to[bend right=10] (0n=0m.80)
      (1=0m.270) to[bend left=10] (0m=0e.120)
      (0n=0e.260) to[bend right=10] (0n=0m.70)
      (0n=0e.270) to (Z=0e.120)
      (0n=0e.280) to[bend left=10] (0m=0e.100)
      (1mm.260) to[bend right=10] (0nmm.65)
      (1mm.280) to[bend left=10] (0mme)
      (0nme.250) to[bend right=10] (0nmm)
      (0nme) to (0m=0e)
      (0nme.290) to[bend left=10] (0mme)
      (0e=0e.240) to[bend right=10] (0m=0e.70)
      (0mee) to (0mme) ;

      \node[above] at (4.6,0) (BRAID) {\stbtype{\RAonepsix{\BRAID}}{BRAID}};
      \node[above] at (5.3,0) (Z=Z) {\stbtypehyp{\RAoneAthree{\ZZeZZ}}{Z=Z}};
      \node[above] at (6,0) (Z=0m) {\stbtypehyp{\RAoneAthree{\ZZeZeZn}}{Z=0m}};
      \node[above] at (6.7,0) (0m=0m) {\stbtypehyp{\Athree{\ZneZZeZn}}{0m=0m}};
      \node[above] at (7.4,0) (0mmm) {\stbtype{\Atwopthree{\ZneZneZneZ}}{0mmm}};

      \draw[lin]
      (CAVE.300) to (BRAID.140)
      (CAVE.310) to[bend left=15] (Z=Z.100)
      (Z=0n.290) to[bend left=15] (Z=Z.90)
      (Z=0n.305) to[bend left=15] (Z=0m.150)
      (0---0m.290) to[bend left=15] (Z=0m.140)
      (0n=0m.280) to[bend left=10] (Z=0m)
      (0n=0m.290) to[bend left=10] (0m=0m)
      (0nmm.280) to[bend left=10] (0m=0m.100)
      (0nmm.300) to[bend left=10] (0mmm.100)
      (Z=0e.240) to[bend right=10] (Z=0m.90)
      (0m=0e.250) to[bend right=10] (0m=0m.80)
      (0mme.250) to (0mmm) ;

      \draw[line width = 4pt, SkyBlue] (0,2.35) -- (0.4,2.35); \node[legend] at (1.2,2.35) {\footnotesize $\sA_1$};
      \draw[line width = 4pt, Cyan] (0,2.1) -- (0.4,2.1); \node[legend] at (1.2,2.1)  {\footnotesize $\sA_1^2$};
      \draw[line width = 4pt, NavyBlue] (0,1.85) -- (0.4,1.85); \node[legend] at (1.2,1.85)  {\footnotesize $\sA_1^3$};
      \draw[line width = 4pt, MidnightBlue] (0,1.6) -- (0.4,1.6); \node[legend] at (1.2,1.6)  {\footnotesize $\,^r\sA_1^3$};
      \draw[line width = 4pt, Blue] (0,1.35) -- (0.4,1.35); \node[legend] at (1.2,1.35)  {\footnotesize $\,^r\sA_1^4\:_{cub}$};
      \draw[line width = 4pt, Periwinkle] (0,1.1) -- (0.4,1.1); \node[legend] at (1.2,1.1)  {\footnotesize $\,^r\sA_1^4\:_{con}$};
      \draw[line width = 4pt, Orchid] (0,0.85) -- (0.4,0.85); \node[legend] at (1.2,0.85)  {\footnotesize $\,^r\sA_1^5$};
      \draw[line width = 4pt, Fuchsia] (0,0.6) -- (0.4,0.6); \node[legend] at (1.2,0.6)  {\footnotesize $\,^r\sA_1^6$};

      \draw[line width = 4pt, pink] (0,0.25) -- (0.4,0.25); \node[legend] at (1.2,0.25)  {\footnotesize $\sA_2$};
      \draw[line width = 4pt, Lavender] (1.4,0.25) -- (1.8,0.25); \node[legend] at (2.6,0.25)  {\footnotesize $\sA_1\sA_2$};
      \draw[line width = 4pt, Magenta] (0,0) -- (0.4,0); \node[legend] at (1.2,0)  {\footnotesize $\sA_1^2\sA_2$};
      \draw[line width = 4pt, Red] (1.4,0) -- (1.8,0); \node[legend] at (2.6,0)  {\footnotesize $\,^r\sA_1^3\sA_2$};

      \draw[line width = 4pt, Goldenrod] (9.4,1.2) -- (9.8,1.2); \node[legend] at (10.6,1.2) {\footnotesize $\sA_2^2$};
      \draw[line width = 4pt, Apricot] (9.4,0.95) -- (9.8,0.95); \node[legend] at (10.6,0.95) {\footnotesize $\sA_1\sA_2^2$};

      \draw[line width = 4pt, orange] (9.4,0.7) -- (9.8,0.7); \node[legend] at (10.6,0.7) {\footnotesize $\sA_2^3$};

      \draw[line width = 4pt, Green] (9.4,0.25) -- (9.8,0.25); \node[legend] at (10.6,0.25) {\footnotesize
        $\sA_3$,~$\sA_1\sA_3$,~$\,^r\sA_1\sA_3$,~\textit{etc.}
      };

    \end{tikzpicture}
  }
  \caption{ \label{fig:labelledstablemodelgraph}
    \protect\begin{tabular}[t]{l}
              Stable model and singularity types of non-hyperelliptic genus 3 curves\\[-0.2cm]
              {\footnotesize (the genus of a component is indicated by its thickness)}
              \protect\end{tabular}
          }
  \renewcommand{\ifcolored}{false}
\end{figure}
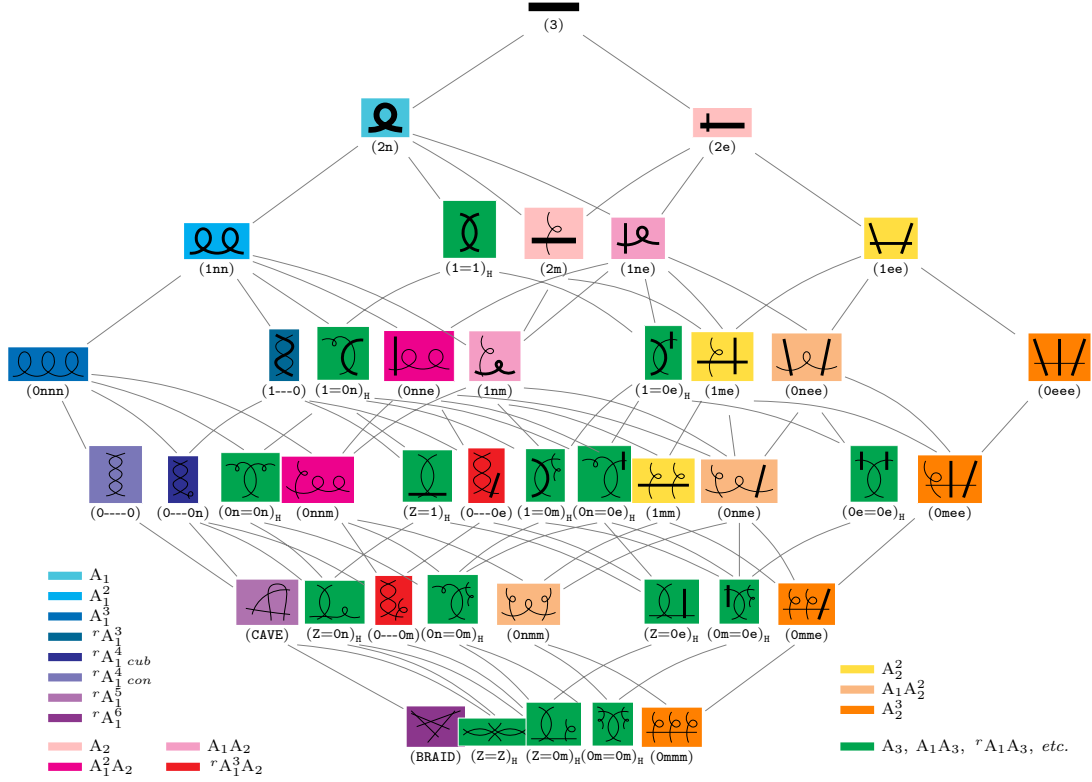

\begin{theorem}\label{thm:GIT-MD}
  Let $C/K$ be a plane quartic curve given by a smooth ternary quartic $F$ over a non-archimedean
  local field $K$ of residue characteristic $p \neq 2, 3, 5, 7$.
  Suppose that the reduction $\bar{F}$ of the ternary quartic is GIT-semi-stable.
  Then the stable reduction type of $C$ is, in terms of the singularity type
  of $\bar{F}$, as in Tab.~\ref{tab:SngToRed}.
\end{theorem}

\begin{remark}\label{rem:wildcard}
  The wildcard \texttt{*} used in Tab.~\ref{subtab:SngToRedA3} and
  Tab.~\ref{subtab:SngToRedrA3} represents any possible component in a
  reduction type, restricted to the set of 42 labels as  defined exactly in
  Fig.~\ref{fig:labelledstablemodelgraph}. For example, $\mathtt{(1=*)_H}$ can
  be any of the four types $\mathtt{(1=1)_H}$, $\mathtt{(1=0n)_H}$,
  $\mathtt{(1=e)_H}$ and $\mathtt{(1=0m)_H}$. But it cannot be
  $\mathtt{(Z=1)_H}$, which is covered by the expression $\mathtt{(*=1)_H}$
  instead.
\end{remark}

\begin{table}
  \renewcommand{\arraystretch}{0.85}
  \begin{subtable}[c]{0.19\textwidth}
    \begin{tabular}[t]{@{}l>{\footnotesize}c@{}}
      \toprule
      Sing. type & {\normalsize Red. type} \\
      \midrule
      \ \ $\sA_1$ & $\mathtt{(2n)}$ \\
      \ \ $\sA_1^2$    & $\mathtt{(1nn)}$  \\
      \ \ $\sA_1^3$   & $\mathtt{(0nnn)}$    \\
      $\,^r\sA_1^3$ & $\mathtt{(1\text{-}\text{-}\text{-}0)}$ \\
      $\,^r\sA_1^4\,_{con}$ & $\mathtt{(0\text{-}\text{-}\text{-}\text{-}0)}$  \\
      $\,^r{\sA_1^4}\,_{cub}$ & $\mathtt{(0\text{-}\text{-}\text{-}0n)}$  \\
      $\,^r\sA_1^5$ & $\mathtt{(CAVE)}$ \\
      $\,^r\sA_1^6$   & $\mathtt{(BRAID)}$   \\
      \bottomrule
    \end{tabular}
    \caption{$\sA_1$-singularities}
    \label{subtab:SngToRedA1}
  \end{subtable}
  \hspace{\fill}
  \begin{subtable}[c]{0.21\textwidth}
    \begin{tabular}[t]{@{}l>{\footnotesize}c@{}}
      \toprule
      Sing. type & {\normalsize Red. type} \\
      \midrule
      Smooth  & $\mathtt{(3)}$ \\
      $\geq\sA_4${\footnotesize, incl. $c^2$} & $\mathtt{(*)_H}$\\
      \bottomrule
    \end{tabular}
    \caption{Other types}
    \label{subtab:SngToRedOthers}
  \end{subtable}
  \hspace{\fill}
  \begin{subtable}[c]{0.45\textwidth}
    \begin{tabular}[t]{@{}l>{\footnotesize}c@{}}
      \toprule
      Sing. type & {\normalsize Red. type} \\
      \midrule
      \ \ $\sA_2$        & $\mathtt{(2e)}$ or $\mathtt{(2m)}$ \\
      \ \ $\sA_1\sA_2$   & $\mathtt{(1ne)}$ or $\mathtt{(1nm)}$ \\
      \ \ $\sA_1^2\sA_2$ & $\mathtt{(0nne)}$ or $\mathtt{(0nnm)}$ \\
      $\,^r\sA_1^3\sA_2$ & $\mathtt{(0\text{-}\text{-}\text{-}0e)}$ or $\mathtt{(0\text{-}\text{-}\text{-}0m)}$ \\
      \ \ $\sA_2^2$      & $\mathtt{(1ee)}$  or $\mathtt{(1em)}$ or $\mathtt{(1mm)}$ \\
      \ \ $\sA_1\sA_2^2$ & $\mathtt{(0nee)}$  or $\mathtt{(0nem)}$ or $\mathtt{(0nmm)}$ \\
      \ \ $\sA_2^3$      & $\mathtt{(0eee)}$  or $\mathtt{(0eem)}$  or $\mathtt{(0emm)}$ or $\mathtt{(0mmm)}$ \\
      \bottomrule
    \end{tabular}
    \caption{$\sA_2$-singularities}
    \label{subtab:SngToRedA2}
  \end{subtable}
  \bigskip\smallskip
  \hspace{\fill}
  \begin{subtable}[t]{0.45\textwidth}
    \begin{tabular}[t]{@{}l>{\footnotesize}c@{}}
      \toprule
      Sing. type    & {\normalsize Red. type} \\
      \midrule
      \ \ $\sA_3$       & $\mathtt{(1=*)_H}$ \\
      \ \ $\sA_1\sA_3$  & $\mathtt{(0n=*)_H}$ or $\mathtt{(*=0n)_H}$ \\
      $\,^r\sA_1^2\sA_3\,_{con}$ &  $\mathtt{(Z=*)_H}$ \\
      \ \ $\sA_2\sA_3$  & $\mathtt{(*=0e)_H}$ or $\mathtt{(*=0m)_H}$ \\
      $\,^r\sA_3^2$ & $\mathtt{(1=*)_H}$ or $\mathtt{(0n=*)_H}$ or $\mathtt{(Z=*)_H}$ \\
      \bottomrule
    \end{tabular}
    \caption{$\sA_3$-singularities}
    \label{subtab:SngToRedA3}
  \end{subtable}
  \hspace{\fill}
  \begin{subtable}[t]{0.45\textwidth}
    \begin{tabular}[t]{@{}l>{\footnotesize}c@{}}
      \toprule
      Sing. type    & {\normalsize Red. type} \\
      \midrule
      $\,^r\sA_1\sA_3$ &   $\mathtt{(1=*)_H}$ or $\mathtt{(*=1)_H}$ \\
      $\,^r\sA_1^2\sA_3\,_{cub}$ &  $\mathtt{(0n=*)_H}$ or $\mathtt{(*=0n)_H}$ \\
      $\,^r\sA_1^3\sA_3$ & $\mathtt{(Z=*)_H}$ \\
      $\,^r\sA_1\sA_2\sA_3$  & $\mathtt{(*=0e)_H}$ or $\mathtt{(*=0m)_H}$ \\
      $\,^r\sA_1\sA_3^2$&  $\mathtt{(1=*)_H}$ or $\mathtt{(0n=*)_H}$ or $\mathtt{(Z=*)_H}$ \\
      \bottomrule
    \end{tabular}
    \caption{$\,^r\sA_3$-singularities}
    \label{subtab:SngToRedrA3}
  \end{subtable}

  \caption{From singularity to reduction types}
  \label{tab:SngToRed}
\end{table}

Tab.~\ref{tab:SngToRed} shows that curves with stable reduction type of the
form $\mathtt{(*=*)_H}$ can have GIT-semi-stable special fibres of different
singularity types.
Ex.~\ref{ex:5_1} and Ex.~\ref{Ex:diffAtype} illustrate this phenomenon for a
quartic $F$ with stable reduction type $\mathtt{(1=0n)}$. Linear changes of
variables yield quartics with singularity types $\sA_3$ or $\sA_1\sA_3$, while
the normalised invariants are the same for both types of singularities.\smallskip

Our proof of this theorem is based on the Hui stratification, which as it
stands is only valid in residue characteristic $p \neq 2, 3, 5, 7$
(\textit{cf.}  Rem.~\ref{rem:huichar}). But this limitation can be removed if
we restrict ourselves to $\sA_1$ and Tab.~\ref{subtab:SngToRedA1}
(\textit{cf.}  Prop.~\ref{prop:node-only-sing}). More generally, $p \ne2 $,
$3$ seems to be sufficient, but this might require revising the Hui
stratification modulo $5$ and $7$, as we cannot preclude the existence of other singularity types.

\begin{example}\label{ex:plquarticsAnyHyperRed} It is easy to produce smooth plane quartics whose normalised invariants modulo a uniformiser $\pi$ are those of $c^2$ and  with stable reduction type any prescribed type that can be obtained from a hyperelliptic curve. Indeed, let $f\in K[u]$ be a degree $8$ polynomial with roots
producing the desired stable reduction type (see~\cite{m2d2}). Let $F(x,y,z)$
be a homogeneous degree 4 polynomial such that $f(u)=F(1,u,u^2)$. Such an $F$ exists
and is unique modulo $xz-y^2$. Then the genus 3 non-hyperelliptic curve
defined by
\begin{math}
  t^2=F(x,y,z) \text{ and } \pi^s t= xz-y^2
\end{math}, %
with $s$ large enough, has the desired stable reduction type too and admits a plane quartic model as
$(xz-y^2)^2-\pi^{2s}F(x,y,z)=0$.
\end{example}

It is difficult to find plane quartics satisfying the assumption of Thm.~\ref{thm:GIT-MD} of having a GIT-semi-stable reduction,
since the creation of such plane quartics often involves large degree field extensions.
However, its mere existence is enough to obtain explicit criteria in terms of
invariants. Indeed, by~\cite[Prop.~3.13]{lllr21}, for each quartic $F$, there exists an equivalent
GIT-semi-stable plane quartic model such that its invariants are minimal in the sense that their valuations are as small as possible among all quartics equivalent to $F$.
The connection in Sec.~\ref{sec:sing-invar-strata} between invariants and
singularity types will allow us to drop the GIT-semi-stable assumption from Thm.~\ref{thm:GIT-MD}.

To explain this further, we need to introduce the following notation.
Suppose that $F$ is a smooth ternary quartic over $K$ and that $I_w$ is an invariant of
weight $w$ for the action of $\SL_3(\C)$. We define the normalised valuation
$v_{\mathrm{S}}$ with respect to a set of invariants $\mathrm{S}$ as in~\cite{lllr21},
\textit{i.e.}\
\begin{displaymath}
  v_\mathrm{S}(I_w(F))=v(I_w(F))/w-\min_{J_\omega\in \mathrm{S}}\{v(J_\omega(F)) / \omega \}
\end{displaymath}
(in this formula, $J_\omega$ denotes an invariant of weight $\omega$).  The two sets
relevant here are the set $\DO{}$ of Dixmier-Ohno invariants or the set
$\iota$ of Shioda invariants written in terms of Dixmier-Ohno
invariants~\cite[Prop.\ 5.6]{lllr21}. We denote the corresponding normalised
valuations by $v_{\DO{}}$ and $v_{\iota}$.  We also say that $F$ has
normalised invariants if $v(I_w(F))=v_{\DO{}}(I_w(F))$ for one (and therefore all)
invariants of positive weight.

We know that the special fibre of the stable model of the curve $F=0$ is a
smooth plane quartic if and only if
$v_{\DO{}}(I_{27}(F))=0$~\cite[Thm.\ 1.9]{lllr21}. When this is not the case,
we still know that this special fibre is a smooth hyperelliptic curve of genus
3 if and only if $v_{\DO{}}(I_{3}(F))=0$ and
$v_{\iota}(I_{3}(F)^5\,I_{27}(F))=0$~\cite[Thm.\ 1.10]{lllr21}.

We extend these definitions to a set of invariants $\ClId{\sA}$, \textit{i.e.}\
\begin{math}
  v_{\mathrm{S}}(\ClId{\sA}(F)) = \min_{J_w \in \ClId{\sA}} v_{\mathrm{S}}(J_w(F))
\end{math}
Especially, the condition $v_{\DO{}}(\ClId{\sA}(F))>0$ may be interpreted as
the reduction of a plane quartic, with invariants the normalised ones of $F$,
being in the algebraic set defined by the reduction of the ideal generated by
$\ClId{\sA}$.
\smallskip

We can now reformulate Thm.~\ref{thm:GIT-MD} in terms of Dixmier-Ohno
invariants.
\begin{corollary}\label{cor:stabsingtypes}
  Let $C\colon F=0$ be a plane quartic curve over a non-archimedean local field $K$ of
  residue characteristic $p \neq 2, 3, 5, 7$. Then, under
  Conjecture~\ref{conj:poschar}, candidates for its reduction type  are returned by
  Alg.~\ref{algo:singularities} modified in such a way that:
  \begin{itemize}
  \item Membership tests ``$(I_3:I_6:\cdots:I_{27}) \in V\ClId{\sA}$'',
    are replaced by ``$v_{\DO{}}\ClId{\sA}(F)~>~0$'';
  \item The outputs are converted according to Tab.~\ref{tab:SngToRed}.
  \end{itemize}
\end{corollary}

We note that the reduction modulo $\pi$ of normalised invariants encodes
exactly the reduction type for singularity types $\sA_1$ (\textit{e.g.}
Tab.~\ref{subtab:SngToRedA1}), offers very few possibilities for $\sA_2$
(\textit{e.g.}\ Tab.~\ref{subtab:SngToRedA2}), and gives only the hint
$\mathtt{(*=*)_H}$ for $\sA_3$ (\textit{e.g.}\  Tab.~\ref{subtab:SngToRedA3}
and Tab.~\ref{subtab:SngToRedrA3}).

\begin{example}\label{ex:braid}
  Consider the Dixmier-Ohno invariants
  \begin{footnotesize}
    \begin{multline*}
     (\,%
      \text{-}2^{\text{-}4}\, 3^{\text{-}2},\ %
      \text{-}2^{\text{-}12}\, 3^{\text{-}6},\ %
      \text{-}2^{\text{-}12}\, 3^{\text{-}8}+O(\pi),\ %
      2^{\text{-}12}\, 3^{\text{-}7}+O(\pi),\ %
      \text{-}2^{\text{-}14}\, 3^{\text{-}12},\ %
      \text{-}2^{\text{-}17}\, 3^{\text{-}10}+O(\pi),\ %
      2^{\text{-}22}\, 3^{\text{-}15}+O(\pi),\ \\%
      \text{-}2^{\text{-}22}\, 3^{\text{-}12}+O(\pi),\ %
      2^{\text{-}24}\, 3^{\text{-}17}+O(\pi),\ %
      2^{\text{-}24}\, 3^{\text{-}15}+O(\pi),\ %
      \text{-}2^{\text{-}29}\, 3^{\text{-}18}+O(\pi),\ %
      \ 2^{\text{-}32}\, 3^{\text{-}16}\, 7+ O(\pi),\ %
      0\,) \,.
    \end{multline*}
  \end{footnotesize}

  \noindent They can be those
of ternary quartics whose reduction modulo $\pi$ is
  a Hui normal form with 6 nodes, for instance
  $F=x\,y\,z\,(z+y+x) + \pi\,z^4$\,.
  Moreover, these invariants are already normalised, as they all have non-negative valuation, and at least one of them has valuation 0 at $\pi$, as the residue characteristic was supposed to be greater than 7.
  Consider then the generators of the ideal~\ref{eq:rA1p6},
  \begin{footnotesize}
    \begin{multline*}
      \langle\, 2^{4}\, 3^{2}\, I_{6} + I_{3}^2,\ \
      3^{2}\, I_{9} - I_{3}^3,\ \
      3\, J_{9} + I_{3}^3,\ \
      3^{4}\, I_{12} + 4\,I_{3}^4,\ \
      2\, 3^{2}\, J_{12} + I_{3}^4,\ \
      2^{2}\, 3^{5}\, I_{15} + I_{3}^5,\ \ \\
      2^{2}\, 3^{2}\, J_{15} - I_{3}^5,\ \
      3^{5}\, I_{18} - I_{3}^6,\ \
      3^{3}\, J_{18} - I_{3}^6,\ \
      2\, 3^{4}\, I_{21} - I_{3}^7,\ \
      2^{4}\, 3^{2}\, J_{21} + 7\, I_{3}^7,\ \
      I_{27}\,\rangle\,.
    \end{multline*}
  \end{footnotesize}

  \noindent Generically, their normalised valuations with respect to the Dixmier-Ohno invariants $(I_3$\,,\ \ldots $I_{27})$ are
  \begin{math}
    (\infty, 1, 1, \infty, 1, 1, 1, 1, 1, 1, 1, \infty)\,.
  \end{math}
  Indeed, $2^43^2I_6 + I_3^2$ is identically equal to zero, causing the first normalised valuation to be $\infty$, and $3^2I_9 - I_3^3 = -2^{-12}3^{-6} + \mathcal{O}(\pi) + 2^{-12}3^{-6}$ generically has valuation $1$, et cetera.
  In other words, $v_{\DO{}}\ref{eq:rA1p6}(F)=1$, and we deduce from
  Cor.~\ref{cor:stabsingtypes} and Alg.~\ref{algo:singularities} that the
  reduction type is \texttt{(BRAID)} (corresponding to the purple
  label on Fig.~\ref{fig:labelledstablemodelgraph}).

\end{example}

\subsection{Proof of Theorem~\ref{thm:GIT-MD}}
\label{sec:proof-theorem}

We split the proof of Theorem \ref{thm:GIT-MD} into different cases according to only having singularities of type $\sA_1$, at worst $\sA_2$, at worst $\sA_3$ or at least an $\sA_4$ singularity.
\subsubsection{Nodal singularities}
\label{sec:sing-type-A1}

Let $C\colon F=0$ be a plane quartic such that its reduction $\bar{F}=0$
modulo $\pi$ defines a plane quartic with singularities at worst of type
$\sA_1$, \textit{i.e.}\ at any of the cases in %
Tab.~\ref{subtab:SngToRedA1}. In this case the plane quartic gives already a
stable model for the curve and it is straightforward to compute the reduction
type of the special fibre.

More precisely, we develop here for the reader's convenience the details for
the $\sA_1$ case. The Hui classification (see Tab.~\ref{tab:inf}) allows us to write
$$
F=(z^2+yz)x^2+(\gamma y^2z+\epsilon y^3)x+(yz^3+\alpha y^2z^2+\beta y^3z+\delta y^4)+\pi\, G=0\,,
$$
and then the genus $2$ curve is given by the hyperelliptic model
$$
t^2+(\gamma y^2z+\epsilon y^3)t+(z^2+yz)(yz^3+\alpha y^2z^2+\beta y^3z+\delta y^4)=0\,.
$$
This is a geometric genus $2$ curve, further singularities would produce more singularities in $\bar{F}$. For example, following the same strategy, the $\sA_1^2$ case gives the arithmetic genus 2 curve
$$
t^2+x^2(\alpha y+\beta x)t+x^2(y^2+x^2)(y^2+\gamma yx+\delta x^2)
$$
consisting of an elliptic curve with a node.
This proves Thm.~\ref{thm:GIT-MD} for the cases in Tab.~\ref{subtab:SngToRedA1}.

\begin{remark}[see Prop.~\ref{prop:node-only-sing} in App.~\ref{sec:node-only-sing}]\label{rem:AoneProof}
In fact, we can give another proof for the converse of this result, of a more geometric nature. If one starts with a plane quartic whose stable reduction is of one of the types $\mathtt{(2n)}$, $\mathtt{(1nn)}$, $\mathtt{(0nnn)}$, $\mathtt{(1\text{-}\text{-}\text{-}0)}$, $\mathtt{(0nnn)}$, $\mathtt{(1\text{-}\text{-}\text{-}0)}$, $\mathtt{(0\text{-}\text{-}\text{-}\text{-}0)}$, $\mathtt{(0\text{-}\text{-}\text{-}\text{-}0n)}$, $\mathtt{(CAVE)}$ and $\mathtt{(BRAID)}$, then the canonical embedding of the stable model is an embedding of the stable genus 3 curve into $\PP^3_R$ whose only singularities in the special fibre are of type $\sA_1$.
This can be verified by doing the appropriate Riemann-Roch dimension
computations on the dualising sheaf of the stable curve.
This is done in the proof of Prop.~\ref{prop:node-only-sing} for some but not all of the cases.
\end{remark}

\subsubsection{Singularity type  \texorpdfstring{$\sA_2$}{A2}}
\label{sec:sing-type-A2}

We consider now quartics with at least a singularity of type $\sA_2$ and not
worse singularities, that is the 7 cases in Tab.~\ref{subtab:SngToRedA2}.
\begin{lemma}
We can place the $\sA_2$ singularity in such a way that the equation of the plane quartic is of the shape
\begin{equation}\label{eq:A2ex}
z^2x^2+(\beta yz^2+\delta y^2z+y^3)x+(yz^3+\alpha y^2z^2+\gamma y^3z)+\pi^s\, G=0\,.
\end{equation}
\end{lemma}

\begin{proof}
In the cases of $\sA_1\sA_2, \,^r\sA_1^3\sA_2,$ and $\sA_2^2$, this is relatively straightforward. In the cases of $\sA_1^2\sA_2, \sA_1\sA_2^2$, and $\sA_2^3$, we need to do a bit more work.
We get a Hui model of the shape $(y^2 + \alpha yx + x^2)z^2 + (\beta y^2x + \gamma yx^2)z + y^2x^2$, and for each $\alpha, \beta,$ and $\gamma$ that is equal to $-2$ we have a cusp.
If for example $\alpha = -2$, then the change of variables $(x,y,z) = (x',x'+y',z')$ will give a model in the shape of Eq.~\eqref{eq:A2ex}, up to permuting the variables.
Note that we use here that $\beta+\gamma \neq 0$ in this case, as required in \cite{hui79}.
\end{proof}
Starting with an equation of the shape of Eq.~\eqref{eq:A2ex}, we can take $t = z^2x$ and making the reduction modulo $\pi$, we recover an arithmetic genus 2 curve piece given by
$$
t^2+(\beta yz^2+\delta y^2z+y^3)t+z^2(yz^3+\alpha y^2z^2+\gamma y^3z)=0\,.
$$
We also find another component, of arithmetic genus $1$. Write Eq.~\eqref{eq:A2ex} as
\begin{equation*}
\begin{split}
a_{400}\pi^{e_1}x^4+(a_{310}\pi^{e_2}y+a_{301}\pi^{e_3}z)x^3+ (z^2+y(a_{220}\pi^{e_4}y+a_{211}\pi^{e_5}z))x^2+\\
(y^3+\delta y^2z+\beta yz^2+a_{103}\pi^{e_6}z^3)x+(a_{004}\pi^{e_7}z^4+yz^3+\alpha y^2z^2+\gamma y^3z+a_{040}\pi^{e_8}y^4)=0\,,
\end{split}
\end{equation*}
with the $a_{ijk}$ equal to $0$ or with $v(a_{ijk})=0$.
We now multiply $x$ by $\pi^{-2s}$ and $z$ by $\pi^s$ where $s=\operatorname{min}\{{e_1}/{6}, {e_2}/{4}, {e_3}/{3}, {e_4}/{2}, {e_5}\}$. We multiply the equation by $\pi^{2s}$ and we obtain
\begin{displaymath}
x(a_{400}\pi^{e_1-6s}x^3+a_{310}\pi^{e_2-4s}yx^2+a_{301}\pi^{e_3-3s}x^2z+ xz^2+a_{220}\pi^{e_4-2s}xy^2+a_{211}\pi^{e_5-s}xyz+y^3)=O(\pi).
\end{displaymath}

Under the inverse of this transformation, any point, except for the ones on the component $x = 0$, is mapped to the $\sA_2$ singularity at $(1:0:0)$ on the original Hui model.
As in subsection \ref{sec:sing-type-A1}, further singularities on the arithmetic genus 2 curve can be read from further singularities on the original plane quartic. The singularities on the arithmetic genus 1 component cannot be read from the singularities of the reduction of the plane quartic.
The above changes of variables produce a construction of the different components of the stable reduction of the curve and the reduction type of the special fibre is easily computable in each case.

\begin{remark}
A generalised version of Lem.~A.1 in the Appendix of \cite{KLS} implies the components of positive arithmetic genus should appear in the special fibre of any stable model and are not contracted. To check that this generalised version of the lemma holds, we should check this for nodal curves of geometric genus 0 and positive arithmetic genus, \textit{e.g.}\ a genus 0 curve $T$ with one node which intersects the rest of the special fibre at least once. The proof is similar to the proof in \cite{KLS} and is done by careful bookkeeping.
The idea is to make a regular model. Following Lem.~3.21 of \cite{liu02}, the node turns into a chain of $\P^1$'s and these $\P^1$'s get subsequently contracted when we are making the stable model. So there will still be a geometric genus 0 component in the special fibre with one self-intersection and at least one more intersection with the rest of the stable model. %
\end{remark}

Note that in the case of multiple $\sA_2$-singularities, the different components of arithmetic genus 1 do not intersect in the special fibre, as they are contracted to their distinct respective cusps on the Hui model.
Therefore, we get that each $\sA_2$-singularity gives rise to a different ``genus 1 tail'', \textit{i.e.}\ a component of arithmetic genus 1, intersecting the rest of the stable reduction in exactly 1 point. This proves Thm.~\ref{thm:GIT-MD} for the cases in Tab.~\ref{subtab:SngToRedA2}.

\subsubsection{Singularity type  \texorpdfstring{$\sA_3$}{A3}}
\label{sec:sing-type-A3}

Next, we consider the cases in Tab.~\ref{subtab:SngToRedA3} and Tab.~\ref{subtab:SngToRedrA3} with at least a singularity of type $\sA_3$ and no worse singularities. %
According to Tab.~\ref{tab:inf}, we can place the $\sA_3$ singularity in such a way that the equation of the plane quartic is of the shape
\begin{equation*}
\begin{split}
F= a_{004}\pi^{e_1} z^4+ (a_{103}\pi^{e_2}x+a_{013}\pi^{e_3}y)z^3 + (x^2+a_{112}\pi^{e_4}xy+a_{022}\pi^{e_5}y^2)z^2+\\%
(\alpha y^2x+a_{031}\pi^{e_6}y^3)z+ y^4+\beta y^3x+\gamma y^2x^2+yx^3+x^2z\,G=0\,,
\end{split}
\end{equation*}
where  $G$ is a degree $1$ polynomial in $x$ and $y$ with coefficients of positive valuation.
Consider the new variables $(X,\,Y,\,Z)=(x,\,\pi^sy,\,\pi^{2s}z)$ where $s=\min\{e_1/4,e_2/2, e_3/3,e_4, e_5/2, e_6\}$.  This gives an embedding of $C = \{F=0\} \subseteq\mathbb{P}^2\hookrightarrow\mathbb{P}^2\times\mathbb{P}^2:\,(x:y:z)\mapsto ((x:y:z),\, (X:Y:Z))$.

Modulo $\pi$, and if $x\neq0$, we get the arithmetic genus $1$ component
$$
\{ ((1:y:z), (1:0:0)) : z^2+\alpha y^2z+ y^4+\beta y^3+\gamma y^2 + y\}\,.
$$
Notice that the homogenisation of this equation is a singular model of the arithmetic genus $1$ curve mentioned above and that the point at infinity actually corresponds to two points.
If $Z\neq 0$, modulo $\pi$ we obtain the points $((0:0:1),(X:Y:1))$ where $X$ and $Y$ satisfy the equation
\begin{equation*}
\begin{split}
X^2+\alpha XY^2+ Y^4+ a_{031}\pi^{e_6-s}Y^3+a_{022}\pi^{e_5-2s}Y^2+a_{112}\pi^{e_4-s}XY+ a_{013}\pi^{e_3-3s}Y+\\ a_{103}\pi^{e_2-2s}X+a_{004}\pi^{e_1-4s}=0\,.
\end{split}
\end{equation*}
This is again an arithmetic genus $1$ curve, whose projective closure is again a singular model. The point at infinity corresponds to two points.
These two arithmetic genus $1$ components intersect at the double point
$((0:0:1), (1:0:0))$.

If no more singularities appear on any of the arithmetic genus 1 curves we obtain reduction type $\mathtt{(1=1)_H}$ for the quartic. From the original quartic equation we can only control the singularities appearing in one of the arithmetic genus 1 components. We do not have any information about the possible singularities of the second component.
A case by case inspection gives the reduction type possibilities in Tab.~\ref{subtab:SngToRedA3}.

For the cases in Tab.~\ref{subtab:SngToRedrA3}, we proceed as follows: we study first the cases $\,^r\sA_1\sA_3$ and $\,^r\sA_1^2\sA_3$ and then their degenerations.
In the first case we can write the quartic as
$$
x(y^2z+x(\alpha z^2+\beta yz +xz+y^2))+\pi^s\,G=0\,.
$$
The naive reduction of this model gives an arithmetic genus 1 curve. The change of variables $X=x$, $Y=\pi^{s/4}$ and $Z=\pi^{s/2}z$ gives the quartic
\begin{math}
  Z(XY^2+\alpha X^2Z+gZ^3)=0\,,
\end{math}
where $g$ is the coefficient of $z^4$ in $G$, defining another arithmetic genus 1 curve.
The map $C\hookrightarrow \mathbb{P}^2\times\mathbb{P}^2$ given by $(x:y:z)\mapsto ((x:y:z),\,(X:Y:Z))$ produces two arithmetic genus 1 curves intersecting at two points. That is, a degeneration of $\mathtt{(1=1)_H}$. Further singularities on the first arithmetic genus 1 component can be read from further singularities of the original plane quartic. A priori we cannot see further singularities on the second arithmetic genus 1 component.

For the last case, namely $\,^r\sA_1^2\sA_3$, which is very degenerated, we change the strategy and we use Tim Dokchitser's machinery \cite{Dokchitser21} to compute regular models. This strategy is especially suitable for this case, but this is also an opportunity to show how it could have been used to prove the previous cases in a different way. The valuation polytope associated to a plane curve of the shape
$$
(xz-y^2)(xz-\alpha y^2)+\pi^s\, G=0\,,
$$
corresponding to a plane quartic with a $\,^r\sA_1^2\sA_3$ singularity type has an edge (after switching $y$ and $z$ and  making $z=1$ as in~\cite[Fig.~7.2a]{BDDLL2023}) joining the points  corresponding to the monomials $1$, $xy$ and $x^2y^2$. From~\cite[Fig.~7.3]{BDDLL2023}, we see that we can (only) obtain any of the degeneration types $\mathtt{(*=*)_H}$ by playing with the valuations of the coefficients of $G$.

All of this together proves Thm.~\ref{thm:GIT-MD} for the cases in Tab.~\ref{subtab:SngToRedA3} and \ref{subtab:SngToRedrA3}.

\subsubsection{Singularity type  \texorpdfstring{$\sA_4$}{A4}}
\label{sec:sing-type-A4}

If a plane quartic has an $\sA_4$ singularity, we can assume, according to Hui, that it is given by an equation of the form
$$
{x}^{2}{z}^{2}+2\,{y}^{2}xz+{y}^{4}+\alpha\,{y}^{3}x+\beta\,{y}^{2}{x}^{2}+y{x}^{3}+\pi\, G=0\,.
$$
After multiplying $x$ by a $\pi^s$, dividing $z$ by $\pi^s$ with $4s<1$ and renaming, we obtain a quartic of the form
\begin{math}
  ({x}{z}+{y}^{2})^2+\pi\, G=0\,,
\end{math}
with singularity type $c^2$. We introduce the new variable $t=(xz+y^2)/\pi^{1/2}$ to get the equation $t^2+G=0$. Modulo $\pi$ this produces a hyperelliptic equation $\bar{t}^2+\bar{G}(1,\bar{s},-\bar{s}^2)=0$ and hence, a hyperelliptic reduction type $\mathtt{(*)_H}$. The exact reduction type is encoded in the ``cluster picture'' of $G(1,s,-s^2)$ as explained in \cite{m2d2}.
 This proves Thm.~\ref{thm:GIT-MD} for the cases in Tab.~\ref{subtab:SngToRedOthers}.

\subsection{Applications to databases}
\label{sec:modform}
We applied Alg.~\ref{algo:singularities} to a dataset consisting of more than
82\,000 curves by Sutherland~\cite{Database3}. Our method only works for
primes $p > 7$ and there are 137\,496 pairs $(C,p)$ for which $p > 7$ and $p$
is a prime of bad reduction of a plane quartic~$C$ appearing in the
dataset. For 131\,673 of them, the reduction has only one $\sA_1$ singularity,
and the stable reduction type is \texttt{(2n)}. Other common types were
$\sA_1^2$ (\,\texttt{(1nn)}, 3511 cases) and $\sA_2$ (\,\texttt{(2e)} or
\texttt{(2n)}, 1829 cases). Not all the singularity types occurred. For
example, there was no curve with three cusps in the dataset. The total
computation took a little less than 1650 seconds or about 10 milliseconds per
case on a machine with an \textsc{amd} \textsc{epyc} {\footnotesize 7713}
\textsc{cpu}. For the computation, which was done with \textsc{magma} version
2.28-8, only 74~\textsc{mb} of memory was needed.  The largest part of the computation time is spent on evaluating the equations for the $\sA_1^3$ stratum on the Dixmier-Ohno invariants of the curves.

\appendix

\section{Node-only singularity quartics}
\label{sec:node-only-sing}

Thm.~\ref{thm:GIT-MD} reports eight stable reduction types whose
characterisation with singularities is one-to-one, they correspond to quartics with
only node singularities. The following result gives an independent and more
geometric proof of this.
\begin{proposition}\label{prop:node-only-sing}
  Let $C\colon F=0$ be a smooth plane quartic over a non-archimedean
  local field $K$.
  Suppose that the reduction $\bar{F}$ of the ternary quartic is GIT-semi-stable.
  Then the stable reduction type of the plane quartic $C$ is $\mathtt{(2n)}$, respectively
  $\mathtt{(1nn)}$, $\mathtt{(0nnn)}$,
  $\mathtt{(1\text{-}\text{-}\text{-}0)}$, $\mathtt{(0nnn)}$,
  $\mathtt{(1\text{-}\text{-}\text{-}0)}$,
  $\mathtt{(0\text{-}\text{-}\text{-}\text{-}0)}$,
  $\mathtt{(0\text{-}\text{-}\text{-}\text{-}0n)}$, $\mathtt{(CAVE)}$ and
  $\mathtt{(BRAID)}$, if and only if the singularity type of the reduction
  $\bar{F}$ is $\sA_1$, respectively $\sA_1^2$, $\sA_1^3$, $\,^r\sA_1^3$,
   $\,^r\sA_1^4\,_{con}$, $\,^r{\sA_1^4}\,_{cub}$, $\,^r\sA_1^5$
  and $\,^r\sA_1^6$\,.
\end{proposition}
\begin{proof}
  For ternary quartics $F$ such that $\bar{F}$ is GIT-semi-stable with $\sA_1^k$
  singularities, reducible or not, we know that there exists a plane quartic
  model whose special fibre only has nodes and therefore is semi-stable (it is
  actually stable).  Moreover, we immediately know the stable reduction
  type.

  Conversely, let us start with a curve whose stable reduction type is one of
  the eight stated, and which we assume to be non-hyperelliptic. Since we have
  to relate it to certain quartic singularities, we must first embed its
  special fibre in $\PP^2$. This can be done using techniques of a geometric
  nature, which depend on the type considered. We do this for the two extremal
  types, \texttt{(2n)} and \texttt{(BRAID)}, the other types can be done very
  similarly.

  For \texttt{(2n)}, we are looking at a curve $C$ of genus 2 with two points
  $R$ and $S$ glued to each other. We want to show that the canonical sheaf
  embeds it into $\PP^2$. The canonical sheaf is obtained by taking
  $K_{\widetilde{C}} + R + S$ and considering its sections, which we view as differentials with at most single poles in $R$ and $S$ and no other poles, for which the sum
  of the residues at $R$ and $S$ is zero. We know that
  $h^0_{\widetilde{C}}(K_{\widetilde{C}} + R + S) = 3$ and
  $h^0_{\widetilde{C}}(K_{\widetilde{C}} + R) =
  h^0_{\widetilde{C}}(K_{\widetilde{C}} + S) =
  h_{\widetilde{C}}(K_{\widetilde{C}}) = 2$ by Riemann-Roch. The sum of the
  residues of a meromorphic differential is 0, so the functions will automatically
  satisfy the residue condition. By Riemann-Roch
  $h^0_{\widetilde{C}}(K_{\widetilde{C}} + R + S - P) = 2$ for all $P$, so
  there will be a well-defined map to $\PP^2$. The embedding condition is not
  satisfied at the points $R$ and $S$, but that makes sense, as they are
  mapped to the same point. For any other two points $P$ and $Q$, we claim
  that $h_{\widetilde{C}}(K_{\widetilde{C}} + R + S - P - Q) = 1$. It suffices to check that
  $P+Q \not\sim R+S$. As there is only one degree 2 map from $\widetilde{C}$ to $\PP^1$,
  the hyperelliptic involution, the only way in which $P+Q \sim R+S$ can hold
  is if both $P$ and $Q$ and $R$ and $S$ are hyperelliptic involutions of each
  other. As we were looking at non-hyperelliptic reduction, $R$ and $S$ are
  not hyperelliptic involutions of each other. Therefore, we get an embedding
  outside of $R$ and $S$. To see that the image in $\PP^2$ is a nodal curve,
  we note that there exists a section vanishing twice at $R$, but only once at $S$.
  This follows from the fact that $h_{\widetilde{C}}(K_{\widetilde{C}} - S) = 1 > 0 = h_{\widetilde{C}}(K_{\widetilde{C}} - S - R)$.

  For \texttt{(BRAID)}, we look at four copies of $\PP^1$ glued in total at
  six pairs of points. Assume that the glued points are the disctinct points $0$, $1$, and
  $\alpha$, extending the field if necessary. For each $\PP^1$, we can take $\mathrm{div}(dx)$ as canonical divisor,
  which will give $-2 \, \infty$. The canonical sheaf is now obtained by
  putting gluing conditions on four-tuples of differential forms of the shape
  $\frac{ax + b}{x(x-1)(x-\alpha)} dx$. This will give rise to 6 gluing conditions
  on an 8-dimensional vector space, but the sum of all the gluing conditions
  will be 0 as the sum of the residues on any $\PP^1$ is 0. So there are at
  most 5 gluing conditions, and it is easy to check by linear algebra that
  there are really 5 linear independent conditions so that we get a potential
  map to $\PP^2$. In this case, it
  is possible to explicitly compute the image of the map, see that it is
  four lines and check that it is an embedding, similarly to the \texttt{(2n)}
  case.
\end{proof}

\printbibliography

\end{document}